\documentclass{amsart}
\usepackage{amssymb}
\usepackage{yfonts}
\usepackage{color}  
\definecolor{darkgreen}{cmyk}{1,0,1,.5}
\definecolor{green}{cmyk}{1,0,1,0}
\definecolor{m}{rgb}{1,0.1,1}
\definecolor{test}{rgb}{0,0,1}
\definecolor{cmyk}{cmyk}{0,1,1,0}

\newtheorem{Equation}{}[section]

\newtheorem{example}[Equation]{Example}
\newtheorem{theorem}[Equation]{Theorem}
\newtheorem{proposition}[Equation]{Proposition}
\newtheorem{lemma}[Equation]{Lemma}
\newtheorem{corollary}[Equation]{Corollary}

\newtheorem{definition}[Equation]{Definition}
 
\newtheorem{remark}[Equation]{Remark}

\def\I{\operatorname{I}}

\def\End{\operatorname{End}}

\def\oH{\operatorname{H}}

\def\Supp{\operatorname{Supp}}

\def\vol{\operatorname{vol}}

\def\C{\mathbb C}

\def\R{\mathbb R}

\def\Z{\mathbb Z}

\def\Q{\mathbb Q}

\def\maE{{\mathcal E}}
\def\maF{{\mathcal F}}

\def\maG{{\mathcal G}}

\def\maG{{\mathcal G}}

\def\maJ{{\mathcal J}}

\def\cS{{\mathcal S}}
\def\maS{{\mathcal S}}

\def\what{\widehat}
\def\wtit{\widetilde}

\def\blangle{\big{\langle}}
\def\brangle{\big{\rangle}}

\def\dd{\displaystyle}
\def\pa{\partial}

\def\ep{\epsilon}

\begin{document}



\title[symbol calculus for  foliations \today]
{A symbol calculus for  foliations\\
\today}


\author{Moulay Tahar Benameur}
\address{Institut Montpellierain Alexander Grothendieck, UMR 5149 du CNRS, Universit\'e de Montpellier}
\email{moulay.benameur@umontpellier.fr}

\author[J.  L.  Heitsch \today]{James L.  Heitsch}
\address{Mathematics, Statistics, and Computer Science, University of Illinois at Chicago} 
\email{heitsch@uic.edu}

\thanks{MSC (2010): 47G30, 53C12. 
Key words: foliations, symbol calculus, asymptotic operators}

\begin{abstract} The classical Getzler rescaling theorem of \cite{Getz} is extended to the transverse geometry of foliations.  More precisely, a Getzler rescaling calculus, \cite{Getz}, as well as a Block-Fox calculus of asymptotic pseudodifferential operators (A$\Psi$DOs), \cite{BF}, is constructed for all transversely spin foliations.  This calculus applies to operators of degree $m$ globally times degree $\ell$ in the leaf directions, and is thus an appropriate tool for a better understanding of the index theory of transversely elliptic operators on foliations \cite{CM95}. The main result is that the composition of A$\Psi$DOs is again an A$\Psi$DO, and includes a formula for the leading symbol. Our formula is more complicated due to its wide generality but its form is essentially the same, and it simplifies notably for Riemannian foliations. In short, we construct an asymptotic   pseudodifferential calculus for the ``leaf space" of any foliation. Applications will be derived in \cite{BH2012, BH2012b} where we give a Getzler-like proof of a local topological formula for the Connes-Chern character of the Connes-Moscovici spectral triple of \cite{K97},  as well as the (semi-finite) spectral triple given  in \cite{BH2012}, yielding an extension of the seminal Atiyah-Singer $L^2$ covering index theorem, \cite{A}, to coverings of  ``leaf spaces" of foliations.   
\end{abstract}

\maketitle
\tableofcontents

\section{Introduction} 

In \cite{Getz}, Getzler gave an elegant proof of the local Atiyah-Singer index formula for a twisted Dirac operator on a compact spin manifold $M$, which was first proposed by Alvarez-Gaum\'{e}, \cite{AG}.  It is well known that such a result leads immediately to the local Atiyah-Singer index theorem.  To do so, he used a pseudodifferential calculus based on work of Bokobza-Haggiag \cite{B-H} and Widom \cite{Wid1,Wid2}, and he introduced a grading and a corresponding rescaling on the  space of symbols of pseudodifferential operators ($\Psi$DOs) on the twisted spin bundle, which treated Clifford multiplication by a $k$-co-vector  as a differential operator of degree $k$.  These results have been generalized in a number of directions, including  hypoelliptic operators and singular geometric situations.  See for instance \cite{BismutHypo}, \cite{Lescure}, \cite{Ponge},... 

 In this paper we extend the Getzler rescaling calculus of \cite{Getz} to transversely spin foliations, as well as the Block-Fox calculus of asymptotic pseudodifferential operators (A$\Psi$DOs), \cite{BF}, and we generalize them to operators of type $(m,\ell)$, that is of degree $m$ globally times degree $\ell$ in the leaf directions.  This is in the spirit of Kordyukov, \cite{K97}.  In particular, the space of symbols of  type $(m,-\infty)$ are the symbols of grading $m$ on the ``space of leaves" of the foliation, so they are the appropriate space to use for transverse index theory of foliations. 
 In short, we extend these theories to the ``leaf spaces" of foliations.   
  
Note that while the extension of the rescaling formulae to leafwise operators on foliations, in the presence of say a holonomy invariant measure, is a routine (although interesting) exercise, the rescaling theorem in the transverse directions has remained an open problem. 

We concentrate here on the formula for the symbol of an A$\Psi$DO which is the composition of two such operators.  The classical cases considered by Getzler and Block-Fox are the case of a foliation by points. Even though our situation is more general, the  formula is essentially the same. 
 
Recall that given a symbol $p$, Getzler defined a quantization map $\theta$ which produces a $\Psi$DO $\theta(p)$, and given a $\Psi$DO $P$, there is an associated symbol $\varsigma(P)$.   Denote the Getzler rescaling of a symbol $p$ by $p_t$.  Then the composition of two symbols is defined as
$$
p \circ_t q \,\, = \,\,   \varsigma(\theta(p_t) \circ \theta(q_t))_{t^{-1}}.
$$
The power of this calculus lies in the computability of  $\lim_{t \to 0} p \circ_t q $, and that the trace of $\theta(p_t) \circ \theta(q_t)$ (which of course is intimately related to the index of operators)  is determined by this limit.   Block and Fox define an asymptotic symbol $p(t)$ of grading $n$ to be one which has an asymptotic expansion as $t \to 0$ of the form
$$
p(x,\xi,t) \,\, = \,\, \sum_{k=0}^{\infty} t^k p_k(x,\xi),
$$
where $p_k$ is a symbol of grading $n-k$.   Here $(x,\xi) \in T^*M_x$ and $p_k(x,\xi)$ acts on the fiber of the twisted spin  bundle at $x \in M$.  The operator $p_0$ is the leading symbol of $p$.    An A$\Psi$DO is then one which is  of the form $P_t = \theta(p(t)_t)$.    The main theorem of \cite{BF}, which is an extension of the Getzler theorem, states that the composition of A$\Psi$DOs is again an A$\Psi$DO, and its leading symbol is given by the formula
$$
p_0 \,\, = \,\, e^{-\frac{1}{4}\Omega(\pa/\pa \xi,\pa/\pa \xi')}
p_0(x,\xi) \wedge q_0(x, \xi') \, |_{\xi'=\xi},
$$
where $p_0$ and $q_0$ are the leading symbols of $p$ and $q$, and $\Omega(\pa/\pa \xi,\pa/\pa \xi')$ is a differential operator constructed out of the curvature of the connection used to define the twisted Dirac operator.\\

We  introduce here a symbol calculus,  as well as an asymptotic pseudodifferential calculus, for a bounded-geometry foliated manifold $(M, F)$  whose normal bundle is  spin.  The spin hypothesis plays the crucial role of simplifying the computations in the proof of the associated index theorem, but is not required for the final local index formula, exactly as in the Getzler case of a foliation by points. Our symbols are of the form $p(x,\xi,\sigma)$, where $\xi$ is a global co-vector, and $\sigma$ is a leafwise co-vector.  We use the grading which treats Clifford multiplication by a {\em transverse} $k$-co-vector as a differential operator of order $k$.  Our main result is that the composition of A$\Psi$DOs is again an A$\Psi$DO, and its leading symbol $a_0(p,q)$ is given by the formula
$$  
a_0(p,q)(x, \xi, \sigma) \,\, = \,\, e^{-\frac{1}{4}\Omega_{\nu}(\pa/\pa (\xi,\sigma),\pa/\pa (\xi',\sigma'))}
p_0(x,\xi,\sigma) \wedge q_0(x, \xi',\sigma') \, |_{(\xi',\sigma')=(\xi,\sigma)},
$$
where $p_0$ and $q_0$ are the leading symbols of $p$ and $q$. See Definition \ref{Omega} for the notation.   The operator $\Omega_{\nu}(\pa/\pa (\xi,\sigma),\pa/\pa (\xi',\sigma'))$ is quite similar to $\Omega(\pa/\pa \xi,\pa/\pa \xi')$, but involves more terms, and is determined by the curvature $\Omega_{\nu}$ of the connection on the normal bundle to the foliation $F$. 

When $F$ is a Riemannian foliation, this formula for the leading symbol  has a particularly simple form.  Write $\xi = (\eta,\zeta)$,  where $\eta$ is the projection of $\xi$ to the co-normal bundle of $F$, and $\zeta$ is its projection to the co-tangent bundle of $F$.  Then the leading symbol simplifies to
$$
a_0(p,q)(x, \eta, \zeta, \sigma) \,\, = \,\,  
e^{-\frac{1}{4}\Omega_{\nu}(\pa/\pa \eta,\pa/\pa \eta')}
p_0(x,\eta, \zeta, \sigma) \wedge q_0(x, \eta', \zeta, \sigma) \, |_{\eta'=\eta},
$$
and the operator $\Omega_{\nu}(\pa/\pa \eta,\pa/\pa \eta')$ is identical to the one in \cite{Getz} and \cite{BF}. In the non-Riemannian case, as expected, interesting extra terms do contribute, compare \cite{CM95,ConnesMoscoviciHopf}. 

{{The assumption of bounded geometry allows us to work on non-compact manifolds as well as compact ones, for which the assumption is automatic.  It allows us to obtain local estimates which are uniform over the whole manifold, with respect to the canonical coordinate charts we use, which is not true in general for non-compact manifolds. This category is stable under taking Galois coverings as well as passing to the foliation of the holonomy or monodromy groupoids as far as these groupoids are Hausdorff. }}

More specifically, our two main results are as follows.  Denote by $\varsigma(P)$ the symbol of a differential operator $P$, (Definition \ref{SymbOnM}), and by $\theta^{\alpha}(p)$ the quantization of a symbol $p$ (Definition \ref{Quantization}).\\

\noindent
{\bf Theorem  \ref{comp}} [Theorem 2.7 of \cite{Getz}, Theorem 2.1 of \cite{BF}]
Let $p(x,\xi,\sigma)$ and $q(x,\xi,\sigma)$ be symbols which are polynomial in $\xi$ and $\sigma$, with gradings $k_p$ and $k_q$, respectively.  Set $p \circ q   =  \varsigma(\theta^{\alpha}(p) \circ\theta^{\alpha}(q))$.  Then there are differential operators $a_k$,   $k \geq 0$,   which act on pairs of symbols so that: 
\begin{enumerate}

\item  $\dd a_k(p,q)$ is a symbol which is polynomial in $\xi$ and $\sigma$, with grading $k_p + k_q - k$; 

\item
$p \circ q \,\, = \,\, \sum_0^{\infty} a_k(p,q)$, which is actually a finite sum;

\item
$a_0(p,q)(x, \xi, \sigma)  \,\, = \,\, e^{-\frac{1}{4}\Omega_{\nu}(\pa/\pa (\xi,\sigma),\pa/\pa (\xi',\sigma'))}
p(x,\xi,\sigma) \wedge q(x, \xi',\sigma') \, |_{(\xi',\sigma')=(\xi,\sigma)}$.
\end{enumerate}

\medskip

Our proof of this theorem is a direct computation, using Proposition 3.7 of Atiyah-Bott-Patodi \cite{ABP}, a paper whose results we use heavily, so we avoid the use of the daunting Baker-Campbell-Hausdorff formula used by  \cite{Getz} and \cite {BF}.   There is a cancellation of operators in the proof of Theorem \ref{comp}, which is obvious if one uses \cite{ABP}, and is not at all obvious if one uses BCH.  This leads to a much simpler proof and to a nicer formula. It was the fact that the BCH formula does not adapt well to our case which led us to a deeper understanding of Proposition 3.7 of \cite{ABP}, which is essential to the proof of Theorem \ref{comp}, which in turn is essential to the proof of our second main result:  

\medskip\noindent
{\bf Theorem \ref{mainlemma}} [Theorem 3.5 of \cite{Getz}, Lemma 3.10 of \cite{BF}]\\
{\em Let $p(x,\xi,\sigma,t)$ and $q(x,\xi,\sigma,t)$ be asymptotic symbols, with associated  A$\Psi$DOs $P_t = \theta^{\alpha}(p_t)$ and $Q_t = \theta^{\alpha}(q_t)$.  Then
\begin{enumerate}
\item  $P_t \circ Q_t$ is an A$\Psi$DO.

\item  The leading symbol of $P_t \circ Q_t$ is $a_0(p_0,q_0)$, where $p_0$ and $q_0$ are the leading symbols of $p$ and  $q$.
\end{enumerate}}
\medskip

An important difficulty we encounter is how to define the symbol map for operators adapted to foliations.  This occurs because the variables $\zeta$ (the $T^*F$ part of $\xi$) and $\sigma$ correspond to the same spaces, namely the leaves of $F$.  The solution is to make $\sigma$ correspond to an independent space, namely the tangent spaces of the leaves of $F$.  We do this by replacing $M$ by $TF$, the total space of the tangent bundle of $F$.    We must take great care when we do so that the definitions of the symbol and the quantization used in $TF$ respect the definitions of the symbol and the quantization we use in Sections \ref{symbs} and \ref{Sodo} on $M$. The reader should note that the fact that $F$ is a foliation is used in the proof of the  Theorem  \ref{mainlemma}, as it gives us the control we need over the changes of variables used in that proof.

\medskip

As mentioned above,  the above theorems allowed us to extend  the proof of the Atiyah-Singer index theorem given in \cite{Getz}, and the calculation of the cyclic cocycle for the Dirac operator to the transversely elliptic case.  This leads immediately to an extension of the Atiyah-Singer $L^2$ covering index theorem, \cite{A}, to leaf spaces. The results of this paper are also used in the integrality consequences of the  main theorem of \cite{BH-Lefschetz1}, which are treated, among other applications, in \cite{BH-Lefschetz2}.

\medskip
\noindent
{\em Acknowledgements.} We benefited from discussions with many colleagues during the preparation of this work and we thank them all.  We are in particular grateful to the following people:   A. Carey, T. Fack, V. Gayral, G. Hector, D. Perrot,  H. Posthuma, G. Skandalis,  and S. Zelditch.  MB wishes to thank the french National Research Agency for support via the project ANR-14-CE25-0012-01 (SINGSTAR).  {{Finally, we would like to thank the referee whose careful reading and cogent suggestions were very helpful.}}
 
\section{Symbols adapted to foliations}\label{symbs}

Let $F$ be a smooth  foliation of the  smooth Riemannian manifold $M$, where  the dimension of the leaves is $p$ and the dimension of  $M$ is $n$. So the codimension of the foliation is $q=n-p$, which we assume to be even.    The normal bundle of $F$ is denoted $\nu$ and its co-normal bundle  $\nu^*$.  We assume that the metric on $M$ as well as the metrics induced on the leaves of $F$ are of bounded geometry.  In particular the manifold $M$ is complete and so are all the leaves of the foliation $F$. We also assume that all the bundles we use have bounded geometry, which, of course, is automatic for the usual geometric bundles.
 
Assume that $F$ is transversely spin with a fixed spin structure on  $\nu^*$, and denote by $\cS_{\nu}$ the associated spin bundle.   The  connection on  $\nu^*$ is denoted $\nabla^{\nu}$ and the associated spin connection it induces on $\cS_{\nu}$ is denoted $\nabla^{\cS}$.   Let $E \to M$ be a smooth complex vector bundle over $M$ endowed with a Hermitian structure and associated connection $\nabla^E$. We denote by $\nabla$  the resulting tensor product connection on the bundle $\cS_{\nu} \otimes E$. We will also use the Levi-Civita connection on $M$, which is denoted $\nabla^{LC}$. 
 
Denote by $T^*M$ and $T^*F$ the cotangent bundles, and by $TM$ and $TF$ the tangent bundles.  Let $\pi:T^*M  \oplus  T^*F \to M$ be the projection. 

\begin{definition} \label{symbdef}
The symbol space $S^{m,\ell} (M, E)$ consists of all $p \in  C^{\infty}( T^*M  \oplus  T^*F, \pi^*(\End (\cS_{\nu} \otimes E)))$  so that, for any multi-indices $\alpha$, $\beta$ and $\lambda$,  there is a constant $C_{\alpha,\beta,\lambda} > 0$,  such that
$$
||  \,   \pa^{\alpha}_{\xi} \pa^{\beta}_{\sigma}  \pa^{\lambda}_{x}p(x,\xi,\sigma) \, || \,\, \leq \,\, 
C_{\alpha,\beta,\lambda}(1 +  |\xi|)^{m- |\alpha|}(1 + |\sigma|)^{\ell-|\beta|}.
$$
The topology on $S^{m,\ell} (M, E)$ is given by the semi-norms 
$$
\rho_{\alpha,\beta,\lambda}  \,\, = \,\  \inf  \big{\{}  C_{\alpha,\beta,\lambda} \, | \,\,   ||  \,   \pa^{\alpha}_{\xi} \pa^{\beta}_{\sigma}  \pa^{\lambda}_xp(x,\xi,\sigma) \, || \,\, \leq \,\, 
C_{\alpha,\beta,\lambda}(1 +  |\xi|)^{m- |\alpha|}(1 + |\sigma|)^{\ell-|\beta|}  \big{\}} .
$$
\end{definition}
\medskip
Of course, $\pa^{\alpha}_{\xi} \pa^{\beta}_{\sigma}  \pa^{\lambda}_{x}p(x,\xi,\sigma)$ only makes sense if we specify local coordinates.  We will use the so-called ``normal coordinates" associated to the metric on $M$ {{and the above estimates are assumed to hold in all such normal charts of a(ny) given distinguished atlas}}.  Normal coordinates at a point $x \in M$ are given by choosing a neighborhood $U_x$ of $0 \in TM_x$ on which $\exp:TM_x \to M$ is a diffeomorphism, and an orthonormal basis of $TM_x$, which defines coordinates $(x_1,\ldots,x_n)$ on $TM_x$.   This then defines coordinates (also denoted $(x_1,\ldots,x_n)$) in the neighborhood $\exp(U_x)$ of $x$.  In addition, it also induces coordinates on $T^*M$ and $T^*F$.  Because of our assumptions of bounded geometry, no pathologies occur.  The reader should note that for $(x_1,...,x_n)$ normal coordinates at $x \in M$, we  always assume that $\nu_x$ is spanned by $\pa/\pa x_1, ..., \pa/\pa x_q$.

{{So in particular,  a sequence $p_n \in  C^{\infty}( T^*M  \oplus  T^*F, \pi^*(\End (\cS_{\nu} \otimes E)))$ converges to $p$ if and only if,  for each set of  multi-indices $\alpha$, $\beta$ and $\lambda$ and each $n$, there is a constant $C_{n, \alpha,\beta,\lambda} > 0$,  such that $\lim_{n \to \infty}  C_{n, \alpha,\beta,\lambda} = 0$ and 
$$
||  \,   \pa^{\alpha}_{\xi} \pa^{\beta}_{\sigma}  \pa^{\lambda}_{x}[ p(x,\xi,\sigma) -p_n(x,\xi,\sigma)] \, || \,\, \leq \,\, 
C_{n,\alpha,\beta,\lambda}(1 +  |\xi|)^{m- |\alpha|}(1 + |\sigma|)^{\ell-|\beta|},
$$
for all elements of a fixed atlas of normal coordinates for $T^*M  \oplus  T^*F$.}}

It is easy to check that pseudodifferential symbols of order $m$ on $M$, using the variable $\xi$, are  symbols of type $(m, 0)$ while  leafwise pseudodifferential symbols of order $\ell$, using the variable $\sigma$,  are  symbols of type $(0, \ell)$.  See for instance \cite{K97}.

We denote by $\wedge^* \nu^*$ the complexified Grassmann algebra bundle. Then
$$
C^{\infty}(T^*M  \oplus   T^*F, \pi^*(\End(\cS_{\nu} \otimes E)))  \,\, \cong \,\,   
C^{\infty}(T^*M  \oplus  T^*F,  \pi^*(\wedge^* \nu^* \otimes  \End( E))),
$$
as $\wedge^* \nu^*  \cong {\rm Cliff}(\nu^*) $, and since $q$ is even,  $\End(\cS_{\nu})  \cong  {\rm Cliff}(\nu^*)$, where ${\rm Cliff}(\nu^*)$ is the Clifford algebra.  The reader should note carefully that when we represent endomorphisms of $\cS_{\nu}$ as elements of $C^{\infty}(\wedge^* \nu^*)$ and we compose them, the operation we use is Clifford multiplication and not wedge product.  In this regard,  recall  that Clifford multiplication on elements of $C^{\infty}(\wedge^* \nu^*)$ {{is given in terms of the wedge product and inner product by the equation}}
\begin{Equation}\label{wedge}
$\hspace{1cm}
\omega_a \cdot \omega_b \,\, = \,\,   \omega_a  \wedge  \omega_b   \,\, \pm \,\,   
\sum_k (i_{e_k}  \omega_a ) \wedge (i_{e_k}  \omega_b )   \,\, \pm \,\,   
\sum_{k,\ell} (i_{e_k} i_{e_{\ell}} \omega_a ) \wedge (i_{e_k} i_{e_{\ell}}  \omega_b )   \,\, + \,\,
\cdots
$
\end{Equation}
\noindent
where $e_1,...,e_q$ is a local orthonormal basis for $\nu$.   
Note also that we use the convention $\omega \cdot \omega =  -\langle \omega, \omega \rangle =-||\omega||^2$ for co-vectors. 

Following \cite{Getz}, we treat Clifford multiplication by a {\em normal} $k$-co-vector as a differential operator of order $k$.  Thus if  $p \in  C^{\infty}( T^*M  \oplus  T^*F, \pi^*(\End (E))) \cap S^{m-k,\ell} (M, E)$, (so involves no Clifford multiplication) and $\omega \in C^{\infty}(\pi^*(\wedge^k \nu^*))$, we say $\omega \otimes  p$ has grading $m + \ell$.  Note that any symbol can be written as a sum of such symbols.

\begin{definition}
The symbol space $SC^{m,\ell} (M, E)$ is 
$$
SC^{m,\ell} (M, E) \,\, = \,\,   \sum_{k=0}^q  S^{m-k,\ell} (M, E) \cap 
C^{\infty}(T^*M   \oplus  T^*F,  \pi^*(\wedge^k \nu^* \otimes  \End(E))).
$$
{{An element  $p \in SC^{m,\ell}(M, E)$ has grading $m+ \ell$.}}
\end{definition}
Set  
$$
SC^{\infty,\infty}(M, E)=\bigcup_{m,\ell} SC^{m,\ell} (M, E)  \text{ and }
SC^{-\infty,-\infty}(M, E)=\bigcap_{m,\ell} SC^{m,\ell} (M, E).
$$
 Note that  if $p$ has  support  where $|\xi|^2 + |\sigma|^2 < C$ for $C \in \R^+$, then $p\in SC^{-\infty,-\infty} (M, E)$.  

Each element of $SC^{m,\ell} (M, E)$ defines an operator  on smooth sections of $\cS_{\nu} \otimes E$ as follows.  Choose a smooth bump function $\alpha$ on $M \times M$ which is  supported in a neighborhood of the diagonal, and equals one on a neighborhood of the diagonal.  We require that the support of $\alpha$ is close enough to the diagonal that $(\pi, \exp)^{-1}:\Supp(\alpha) \to TM$ is a diffeomorphism onto the component of $(\pi, \exp)^{-1}(\Supp(\alpha))$ which contains the zero section, where  $\pi:TM \to M$ is the projection.  For each $(x,x') \in \Supp(\alpha)$ there is a unique $X \in TM_{x}$ so that $x' = \exp_{x}(X)$.  Denote by $ \mathcal{T}_{x,x'}$ the parallel translation for the bundle $\cS_{\nu} \otimes E$ along the geodesic  $\exp_{x}(tX)$, $t \in [0,1]$, from $x$ to $x'$.

For $x \in M$, $X \in TM_x$, and $u \in C^{\infty}(\cS_{\nu} \otimes E)$, set 
$$
\overline{u}_{x}(X) \,\, = \,\, \alpha(x, \exp_{x}(X)) \mathcal{T}^{-1}_{x,\exp_{x}(X)}  (u(\exp_{x}(X)),
$$
which is an element of $(\cS_{\nu} \otimes E)_x$. 

\begin{definition}\label{Quantization}
The quantization of $p \in SC^{m,\ell}(M, E)$ is denoted $\theta^{\alpha}(p)$.  It is the operator which on $u \in C^{\infty}(\cS_{\nu} \otimes E)$ is given at $x$ by
$$
\theta^{\alpha}(p)(u)(x) \,\, = \,\, (2 \pi )^{-n-p}   \int_{TM_x \times T^*M_x \times 
TF_x \times T^*F_x} \hspace{-3.0cm}
e^{-i\langle  X,\xi\rangle} e^{-i\langle  Y, \sigma -\zeta \rangle}p(x,\eta, \zeta ,\sigma)\alpha(x,\exp_{x}(Y)) \overline{u}_{x}(X)\,  dY  d\sigma dX d\xi.
$$
{{Here $p =p(x,\xi ,\sigma)  = p(x,\eta, \zeta ,\sigma)$, where $\sigma \in T^*F_x$, $\xi \in T^*M_x = \nu^*_x \oplus T^*F_x$,  $\eta \in \nu^*_x$, $\zeta \in T^*F_x$, $X \in TM_x$, and $Y \in TF_x$.}}
\end{definition}

 {{This is well defined since, for $x$ fixed, $\alpha(x,\exp_{x}(Y))\overline{u}_{x}(X)$ has compact support (near zero) in $TM_x \times TF_x$. {{Exactly as for the {{Riemannian}} pseudodifferential calculus {{{developed}} by Getzler for a single manifold \cite{Getz, BF} {{(which  also provides a description in normal coordinates for the usual pseudodifferential operators)}}, our operator $\theta^{\alpha}(p)$ is {{automatically}} a {\underline{uniformly supported}}  operator of class $\Psi^{m, \ell}$, and therefore extends to a bounded operator between any anisotropic Sobolev spaces $\oH^{s, k}$ and $\oH^{s-m, k-\ell}$, see \cite{K97} and Appendix \ref{Bifiltered} for the bounded geometry extension. However, since the operator $\theta^{\alpha}(p)$  is not compactly supported in general, when $m<0$ and $\ell<0$, it  only extends to a locally compact operator on the $L^2$-sections in the sense of \cite{RoeBook} and is not  compact in general.}}}}

One can, of course, integrate out $Y$ and $\sigma$ as indicated below, but this is at the cost of losing control of the transverse aspect of the symbols.  In particular, the space of symbols of grading $m$ on the ``space of leaves" would then not be at all obvious. 
\begin{remark}\label{inverses1}
Set 
$$
\wtit{p}(x,\xi) \,\, = \,\,  (2 \pi )^{-p}  \int_{TF_x \times T^*F_x}  \hspace{-1.0cm}
e^{-i\langle  Y,\sigma -\zeta  \rangle}p(x,\eta, \zeta ,\sigma) \alpha(x,\exp_{x}(Y))\,  dY d\sigma .
$$
{{Then
$$
\theta^{\alpha}(p)(u)(x) \,\, = \,\, 
 FT_x ^{-1} \Big[ \wtit{p}(x, \xi) \Big(FT_x (\overline{u}_{x})(\xi)\Big)\Big] (0), 
$$
where $FT_x$ is the Fourier Transform on $TM_x$.}} We are of course working here in the world of distributions, and we have normalized the metrics.
\end{remark}

\begin{remark}\label{FT1}
If $p(x,\xi ,\sigma) = \sigma^{\beta}$, then a simple calculation using the change of coordinates $\sigma \mapsto \sigma + \zeta$ shows that 
$ \wtit{p}(x, \xi) = \zeta^{\beta}$.   More generally, if  $p(x,\xi ,\sigma)$ is polynomial in $\sigma$, then $\wtit{p}(x, \xi) = p(x,\xi ,\zeta)$.
\end{remark}

\section{Symbols of differential operators} \label{Sodo} 

It is classical that {{any smooth  differential operator on $C^{\infty}(\cS_{\nu} \otimes E)$ with uniformly bounded coefficients \cite{Shubin}}} is in the image of $SC^{\infty,\infty} (M, E)$ under the quantization map $\theta^{\alpha}$, since it is supported on the diagonal.    We want to define  a symbol map for operators which satisfies the usual compatibility conditions \cite{Getz, BF}.  Unfortunately, the naive definition does not work, and we explain in this section how to adjust it for differential operators so that it does work.  The construction of our symbol map for general operators is postponed to Section \ref{Asymptotic}, where we extend that given here for differential operators.

In this section, we also give the basic  examples of symbols which will be used in the sequel, and make note of their gradings.  Note that symbols are not uniquely defined.  In particular, our examples will show, as does Remark \ref{FT1}, that we have a choice of how to represent the associated operators as quantizations of symbols.  We will define the symbols of differential operators so that the symbol of the quantization gives back the original symbol. 

\medskip
We first note that any differential operator {{at a fixed point}} $x \in M$ may be written as a sum of operators of the form $g\nabla_X^m$, where $X \in C^{\infty}(TM)$ is given in normal coordinates at $x$ by $X = \sum^n_{i = 1}c_i \pa/\pa x_i$,  the $c_i$ are constants,  and $g \in C^{\infty}( \wedge \nu^* \otimes \End( E))$.

{{The fact that we have two variables, $\zeta$ and $\sigma$, which correspond to the same space $TF$ causes some complications in computing symbols.  We need to have a way to decide whether to use  $\zeta$ or $\sigma$.  To do so, we divide the differential operators into two classes:  the leafwise operators, which are given at $x$ by $g\nabla_X^m$, where $X = \sum^n_{i=q+1} c_i \pa/\pa x_i$, so $X_x \in TF_x$, and that fact is part of the data;   the global operators, where $X = \sum^n_{i=1} c_i \pa/\pa x_i$, and there is no restriction on the $c_i$, so a priori $X_x \in TM_x$, (and if $X_x$ happens to be in $TF_x$,  that fact is  {\underline{not}} part of the data).  So if $X_x \in TF_x$, and that fact is part of the data, the calculations will be expressed using the variable $\sigma$.  For $X_x \in TM_x$ with no restrictions, the calculations will be expressed using the variable $\xi$, (even if $X_x$ happens to be in $TF_x$).}} 

\begin{definition}\label{derivatives}
Suppose $X \in C^{\infty}(TM)$ is given in normal coordinates at $x$ by $X = \sum^n_{i=1} c_i \pa/\pa x_i$ where the $c_i$ are constants.  For $X = \sum^n_{i=q+1} c_i \pa/\pa x_i$, so $X_x \in TF_x$, and that is part of the data, set {{(for $x'$ close to $x$)}}
$$
\nabla_{X_x}^m \langle  \exp^{-1}_{x}(x'),(\xi,\sigma) \rangle  \,\, = \,\,  \nabla_{X_x}^m \langle  \exp^{-1}_{x}(x'),\sigma \rangle,
$$
where $\sigma \in T^*F_x \subset T^*M_x$.  Otherwise, set
$$
\nabla_{X_x}^m \langle  \exp^{-1}_{x}(x'),(\xi,\sigma) \rangle \,\, = \,\,  \nabla_{X_x}^m \langle  \exp^{-1}_{x}(x'),\xi \rangle,
$$
{ { where $\xi \in T^*M_x$.}}
\end{definition}

{{ Note that since $X$ has constant coefficients, $\nabla_{X_x}^m$ is well defined at $x \in M$.}}

\begin{definition}\label{SymbOnM}  The symbol $\varsigma(P)$ of a differential operator $P$ acting on sections of $\cS_{\nu} \otimes E$  is defined as follows.  Let $x \in M$, $\xi \in T^*M_x$, $\sigma \in T^*F_x$, and $u_{x} \in (\cS_{\nu} \otimes E)_x$.  Then 
$$
\varsigma(P)(x,\xi,\sigma)(u_{x}) \,\, = \,\,  
P(x' \mapsto e^{i\langle  \exp^{-1}_{x}(x'),(\xi,\sigma) \rangle}
\alpha(x, x') \mathcal{T}_{x,x'}(u_{x})) \, |_{ x'=x}.
$$
\end{definition}
 
{{ The expression $\exp^{-1}_{x}(x'),(\xi,\sigma) \rangle$  is purely formal as is $e^{i\langle  \exp^{-1}_{x}(x'),(\xi,\sigma) \rangle}$.  To compute $\varsigma(P)(x,\xi,\sigma)(u_{x})$, we only need $\nabla_{X_x}^m (e^{i\langle  \exp^{-1}_{x}(x'),(\xi,\sigma) \rangle})$}} to make sense for $X \in C^{\infty}(TM)$ as in Definition \ref{derivatives}, and the meaning of that is obvious.

\begin{lemma}\label{sigexp1}
Suppose $X \in C^{\infty}(TM)$ is given in normal coordinates at $x$ by $X = \sum_i c_i \pa/\pa x_i$ where the $c_i$ are constants, and $g \in C^{\infty}( \wedge^k \nu^* \otimes \End(E))$.   If $X_x \in TF_x$, and that is part of the data, then
$$
\varsigma(g\nabla^m_X)(x,\xi,\sigma) \,\, = \,\, {{g\langle  iX_x, \sigma \rangle^m.}}
$$
Otherwise,
$$
\varsigma(g\nabla^m_X)(x,\xi,\sigma) \,\, = \,\, {{g\langle  iX_x, \xi \rangle^m}}. 
$$

In  both cases, the symbol has grading $m+k$, and $\theta^{\alpha} (\varsigma(g\nabla^m_X))(x) \,\, = \,\, g\nabla^m_{X}(x).$
\end{lemma}

\begin{proof}  We first note that $\alpha = 1$ near $x$ so we may ignore it.  To compute the symbol for $X$ in the second case, we may restrict to the geodesic $\exp_{x}(tX)$  ($= tX$ in the normal coordinates) determined by  $X$ through $x$, and setting $x' = \exp_{x}(tX)$ in $\langle  \exp^{-1}_{x}(x'),\xi \rangle $ shows that 
$$
X\langle  \exp^{-1}_{x}(x'),\xi \rangle   =  \langle X,\xi \rangle    \quad \text{and} \quad  
X^k\langle  \exp^{-1}_{x}(x'),\xi \rangle  \,\, = \,\, 0,  \,\, \text{for  } k \geq2.
$$
Then a simple induction argument gives
$$
\varsigma(\nabla^m_X)(x,\xi,\sigma) u_{x} \,\, = \,\,   \sum_{k=0}^m  \Big(\hspace{-0.1cm}\begin{array}{c}m \\ k \end{array}\hspace{-0.1cm}\Big)   \Big[  X^k(e^{i\langle  \exp^{-1}_{x}(x'),(\xi,\sigma) \rangle})
 \nabla_X^{m-k}  \mathcal{T}_{x,x'}(u_{x})(x') \Big] \, | \,_{x' = x}   \,\, = \,\, 
$$
$$ \sum_{k=0}^m  \Big(\hspace{-0.1cm}\begin{array}{c}m \\ k \end{array}\hspace{-0.1cm}\Big) \langle  iX,\xi\rangle^k \Big[ \nabla_X^{m-k}  \mathcal{T}_{x,x'}(u_{x})(x')  \, | \,_{x' = x}\Big].
$$
Now $ \mathcal{T}_{x,x'}(u_{x})$ is the parallel translate of $u_{x}$ along geodesics through $x$.  The vector field  $X$ is tangent to the geodesic  $\exp_{x}(tX) = tX$, and on that geodesic $ \nabla_{X}   \mathcal{T}_{x,x'}(u_{x}) = 0$ identically, so for all $m-k > 0$, $ \nabla_X^{m-k}   \mathcal{T}_{x,x'}(u_{x})(x') = 0$.  

In the first case, where $X_x \in TF_x$ and that is part of the data, we need only substitute $\sigma$ for $\xi$ in $\langle  iX,\xi\rangle^{m-j}$ above.

\medskip
For the second part, first note that $\theta^{\alpha} (g\langle  iX, \xi\rangle^m) = g\theta^{\alpha} (\langle  iX, \xi\rangle^m)$, and similarly for $\theta^{\alpha} (g\langle  iX, \sigma \rangle^m)$, so we may ignore $g$.  Next note that for both $p(x,\xi,\sigma) = \langle  iX, \sigma\rangle^m$, (so $X \in TF_x$), and $p(x,\xi,\sigma) = \langle  iX, \xi \rangle^m$, (for general $X$),  $\wtit{p}(x,\xi) =  \langle  iX, \xi \rangle^m$.  Thus we have, by Remark \ref{inverses1} for both cases, that for functions $u$,  
$$
\theta^{\alpha} (\langle  iX, \xi\rangle^m)(u)(x) \,\, = \,\,  X^m (\overline{u}_{x})  (0).
$$
To know what $\theta^{\alpha} (\langle  iX, \xi\rangle^m)(u)(x)$ is for general $u$, we only need to know it for the local sections of $\cS_{\nu} \otimes E$  given by $u_j(x') =  \mathcal{T}_{x,x'}(u_{j,x})$, where the $u_{j,x}$ are a basis of $(\cS_{\nu} \otimes E)_x$.  This is because an arbitrary local section can be written as a functional linear combination of the $u_j$ and the mapping $u_x \mapsto \overline{u}_{x}$ is functionally linear.  Now for such $u_j$,
$$
\overline{u}_{j,x}(V) \,\, = \,\, \alpha(\exp_{x}(V), x) 
\mathcal{T}_{x, \exp_{x}(V)}^{-1}\mathcal{T}_{x,\exp_{x}(V)}(u_{j,x})
\,\, = \,\, \alpha(\exp_{x}(V), x)u_{j,x},
$$
so $\theta^{\alpha} (\langle  iX, \xi \rangle^m)(u_j)(x) = 0$, and $\theta^{\alpha} (\langle  iX, \xi\rangle^m)(g_j u_j)(x) = (X^m g_j)(x) u_{j,x}$, where $g_j$ is a smooth  local  function on $M$.  So for arbitrary $u(x')  =  \sum_j g_j(x') u_j(x')$, we have 
$$
\theta^{\alpha} (\langle  iX, \xi\rangle^m)(u)(x) \,\, = \,\,\sum_j  (X^m g_j)(x) u_{j,x}.
$$
But as above  $\nabla^m_X u_j = 0$ on $\exp_{x}(tX)$, so
$$
(\nabla^m_X u)(x) \,\, = \,\,  \sum_j  (X^m g_j)(x) u_j(x)  \,\, = \,\,
\sum_j  (X^m g_j)(x) u_{j,x}  \,\, = \,\, \theta^{\alpha} (\langle iX, \xi\rangle^m)(u)(x).
$$
\end{proof}

The following is  immediate.
\begin{corollary}\label{smdifop}
For any smooth differential operator $D$ acting on $C^{\infty}(\cS_{\nu} \otimes E)$, 
$\theta^{\alpha} (\varsigma(D )) \,\, = \,\,    D.$
\end{corollary}

\medskip
\begin{definition}
If $X \in C^{\infty}(TM)$, $\xi \in T^*M_x$ and $\sigma \in T^*F_x$, set
$$
\langle  X, (\xi,\sigma) \rangle \,\, = \,\, \langle  X, \sigma \rangle,
$$
if $X_x \in TF_x$, and that is part of the data.  Otherwise, set
$$
\langle  X, (\xi,\sigma) \rangle \,\, = \,\, \langle  X, \xi \rangle.
$$
\end{definition}

\begin{lemma} \label{twovecs}  
Suppose $X, Y \in C^{\infty}(TM)$.  Then
 $$
\varsigma(\nabla_X\nabla_Y)(x,\xi,\sigma) \,\, = \,\, \langle  iX, (\xi,\sigma) \rangle \langle   iY,(\xi,\sigma) \rangle  + \frac{1}{4}\Omega_{\nu} (X,Y)   + iX_{x}(\nabla_Y\langle  \exp^{-1}_{x}(x'),(\xi,\sigma) \rangle) +  \frac{1}{2}\Omega_E (X,Y),
$$
where  $\Omega_{\nu}$ is the curvature of $\nabla^{\nu}$ on $\nu^*$ and $\Omega_E$ is the curvature of $\nabla^E$ on $E$.   
\end{lemma}
The first term has grading $2$, as does the second term, see \cite{Getz}, Example 2.3 b), or \cite{BF}, Example 2.  The third term has grading $1$, and the last has grading $0$.  

\medskip
\begin{proof}
We do only the case involving $\xi$ and leave the other cases to the reader.  So, let  $x,x' \in M$ with $x$ fixed, $X, Y \in C^{\infty}(TM)$, $\xi \in T^*M_x$ and $u_x \in (\cS_{\nu} \otimes E)_x$.   Since the operator is local we may, as above, ignore the term $\alpha(x, x')$.  Denote the curvature of $\nabla$ by $\Omega$.  Then 
$$
\nabla_X\nabla_Y(e^{i\langle  \exp^{-1}_{x}(x'),\xi \rangle} \mathcal{T}_{x,x'}(u_{x})) \, |_{x' = x} \,\, = \,\,
$$
$$
\nabla_X\Big[  i(\nabla_Y\langle  \exp^{-1}_{x}(x'),\xi \rangle)  e^{i\langle  \exp^{-1}_{x}(x'),\xi \rangle}  \mathcal{T}_{x,x'}(u_{x})   +      e^{i\langle  \exp^{-1}_{x}(x'),\xi \rangle}  \nabla_Y \mathcal{T}_{x,x'}(u_{x})\Big]  \, |_{x' = x}  \,\, = \,\,
$$
$$
 i(\nabla_X\nabla_Y\langle  \exp^{-1}_{x}(x'),\xi \rangle)  u_{x}
 +\langle  iY,\xi \rangle  \langle  iX,\xi\rangle u_{x}   +      \langle  iX,\xi\rangle \nabla_Y \mathcal{T}_{x,x'}(u_{x}) +
\nabla_X\nabla_Y \mathcal{T}_{x,x'}(u_{x})  \, |_{x' = x}  \,\, = \,\,
$$
$$
\Big[ iX_x (\nabla_Y\langle  \exp^{-1}_{x}(x'),\xi \rangle)   +  \langle iX,\xi\rangle\langle  iY,\xi\rangle  +  \frac{1}{2}\Omega(X,Y)  \Big] u_{x}.
$$
The fact that $ \nabla_X\nabla_Y \mathcal{T}_{x,x'}(u_{x})  \, |_{x' = x} =  \frac{1}{2}\Omega(X,Y)  u_{x}$ follows from Proposition 3.7 of \cite{ABP} (see Proposition \ref{ABPprop} below).  To finish, we have  (see \cite{LawsonM}, Theorem 4.15),
$$
\Omega(X,Y)  \,\, = \,\,\Omega_{\cS}(X,Y) +\Omega_E(X,Y)
\,\, = \,\, \frac{1}{2}\Omega_{\nu}(X,Y) + \Omega_E(X,Y),
$$
where $\Omega_{\cS}$ is the curvature of $\nabla^{\cS}$ on $\cS_{\nu}$.
\end{proof} 

To finish this section, we consider the transverse  Dirac operator, as well as its square. This is an important example  for the local index theorem treated in \cite{BH2012b}, but it also has an obvious independent interest so we include it here.  Recall that the transverse Dirac  operator $D$ with coefficients in $E$ is given as follows.  Choose a local orthonormal basis $f_1,...,f_q$   of $\nu^*$ and denote by $e_1,...,e_q$ the dual  orthonormal basis of $\nu$ on an open set $U$ of $M$.  Given any element $u \in C^\infty(\cS_{\nu} \otimes E)$, then on $U$ set
$$
\what{D}(u) \,\, = \,\, \sum_{1\leq i\leq q} f_i \cdot \nabla_{e_i}u,
$$
where $f_i \cdot$  is the operator $c(f_i)$, Clifford multiplication by $f_i$.  It is immediate from the definition that this is independent of the basis used,  and it is an easy calculation to show that  the bases need not be orthonormal.

Note that $\what{D}$ is not self adjoint in general.   To correct for this we need to add the Clifford multiplication operator $-c(\mu)/2$, where  $\mu \in \nu$ is  the mean curvature vector field of $F$, that is $\mu=p_\nu (\sum_{1\leq i\leq p} \nabla^{LC}_{X_i} X_i)$, where $p_\nu:TM \to \nu$ is the projection,  $\nabla^{LC}$ is  the Levi-Civita connection on $M$, and $X_1, X_2, . . . , X_p$ is a local orthonormal frame for $TF$.  See \cite{K07,GlK91}.  In particular, the proof in \cite{K07} that perturbation by $\mu/2$ yields a self-adjoint operator does not depend on the foliation being Riemannian.  Using the classic equation defining $\nabla^{LC}$, we may also write $\mu= \sum_{j=1}^p  \sum_{i=1}^q \langle  [e_i,X_j],X_j\rangle e_i.$  Then set
$$
c(\mu) \,\, = \,\,  \sum_{j=1}^p  \sum_{i=1}^q \langle  [e_i,X_j],X_j \rangle f_i.
$$    
The transverse Dirac operator $D$ of $F$ is
$$
 D \,\, = \,\,  \what{D} \,\, - \,\, \frac{1}{2}c(\mu).
 $$

\begin{example}\label{Dtype}
It is easy to see  that
$$
\varsigma(D)(x,\xi,\sigma) \,\, = \,\,  
\varsigma(\sum_{j=1}^q f_j \cdot \nabla_{e_j})(x,\xi,\sigma)  \,\, - \,\, \frac{1}{2} c(\mu)  \,\, = \,\, 
i \sum_{j=1}^q f_j \otimes \langle  e_j, \xi \rangle   \,\, - \,\, \frac{1}{2} c(\mu) \,\, = \,\,  
i \sum_{j=1}^q  f_j   \otimes \langle  e_j, \eta \rangle    \,\, - \,\, \frac{1}{2} c(\mu),
$$
where $\eta$ is the projection of $\xi$ to $\nu^*$.
\end{example}
The first term has grading $2$, and the second grading $1$.  Note that while $D$ is a differential operator of order one, its symbol contains elements of grading two.  Note also that $D$ does not contain any differential operators defined using vectors in $TF$ where that is part of the data.

\begin{definition}
Choose a framing $e_1,...,e_q$ of $\nu_x$, with dual framing $f_1,...,f_q$.  (These framings are not assumed to be orthonormal.)  Extend them to local framings which are  parallel (using $\nabla^{\nu}$) along geodesics  through $x$.  Then the non-integrability tensor $\vartheta_{\nu}$  is the smooth global section of $\wedge^2\nu^* \otimes TM$ which is  given at $x$ by:
$$
\vartheta_{\nu,x} \,\, = \,\, \sum_{j<k}( f_j \wedge f_k \otimes [e_j,e_k])_{x}.
$$
\end{definition}
We leave it to the reader to show that this does not depend on the choice of framing.  Note that $\vartheta_{\nu}$ does depend on the choice of normal bundle $\nu$.   

\begin{remark}
The tensor $\vartheta_{\nu}$ is smooth because the solutions to the differential equations used in defining the $e_i$ are smooth as functions of their initial data, and the $e_i(x)$ can be chosen to vary smoothly in $x$, so their parallel translates along geodesics vary smoothly in all variables.
\end{remark}

\begin{proposition}\label{dsquared}
The symbol of $D^2$ is given by 
$$
\varsigma(D^2)(x,\xi,\sigma) \,\, = \,\,    |\eta|^2     \, - \,  i\sum_j e_{j,x} e_j\langle  \exp^{-1}_{x}(x'),\xi \rangle   \,  +  \,
\frac{1}{2}\sum_{j<k} f_j \wedge f_k   \wedge  \Omega_{\nu} (e_j,e_k)   \,  +  \,
$$
$$
\sum_{j<k} f_j \wedge f_k   \otimes \Omega_E(e_j,e_k)  \,\, + \,\,
 i\langle  \vartheta_{\nu}, \xi \rangle   \,\, + \,\,  \frac{1}{2}\sum  f_i \cdot f_k  \otimes e_k(\langle  [e_i,X_j],X_j\rangle)\,\, + \,\,   \langle i \mu ,  \eta \rangle  \,\, - \,\,  \frac{1}{4} |\mu|^2,      
$$
where $\xi= (\eta,\zeta)$.
\end{proposition}
The first and fourth terms have grading $2$, the second and seventh at most $1$, and the fifth at most $3$.  and the eighth $0$.  The third term in general will have at most grading $4$.  For Riemannian foliations however, it has grading $0$, see Remark \ref{Riemannian} below.  The sixth term has terms of grading $2$ (those where $k \neq i$), and terms of grading $0$  (those where $k=i$, since we are using Clifford multiplication so $f_i \cdot f_i = -1$ in this expression).  
{{Finally, the eighth term has grading $0$, since if $c(\mu) = \sum a_i f_i$, then  $c(\mu)^2  = -\sum a_i^2 = -|\mu|^2$ is a scalar.}}
\begin{proof}
We may assume that our  dual orthonormal bases, $f_1,...,f_q$ of 
$\nu^*$ and $e_1,...,e_q$ of $\nu$,  are  parallel  (using $\nabla^{\nu}$) along geodesics  through $x$. Then 
$$
\varsigma(D^2) \,\, = \,\,  \varsigma([ \what{D} - \frac{1}{2} c(\mu)]^2)   \,\, = \,\,  
\varsigma( \what{D}^2)  \,\, - \,\,  \frac{1}{2} \varsigma( \what{D} c(\mu))     \,\, - \,\, 
\frac{1}{2} \varsigma( c(\mu) \what{D})   \,\, + \,\,   \varsigma(\frac{1}{4} c(\mu)^2).
$$
A simple calculation shows that $\what{D} c(\mu) + c(\mu) \what{D} = \what{D} (c(\mu)) -2\nabla_{\mu}$, and 
$$
\varsigma(\what{D}(c(\mu))) \,\, =  \,\,
-\sum f_i \wedge f_k \otimes e_k(\langle  [e_i,X_j],X_j\rangle).
$$
Thus we get 
$$  
\varsigma(D^2)(x,\xi,\sigma)      \,\, = \,\,
\varsigma( \what{D}^2)(x,\xi,\sigma)   \,\, + \,\,  \frac{1}{2}\sum  f_i \wedge f_k \otimes e_k(\langle  [e_i,X_j],X_j\rangle) \,\, + \,\,   \langle  i\mu ,  \eta \rangle     \,\, - \,\,    \frac{1}{4} |\mu|^2.
$$
As $\mu \in \nu$, and $D$ does not contain any differential operators defined by vectors in $TF$ where that is part of the data,  $\sigma$ plays no role here, and we need only compute $\varsigma( \what{D}^2)(x,\xi,\sigma)$.  As $(\nabla_{e_j} f_k)_{x} = 0$, we have   
$$
\varsigma(\what{D}^2)(x, \xi,\sigma) \,\, =  \,\,  
\varsigma\Big(\sum_{j,k} f_j \cdot \nabla_{e_j} f_k \cdot \nabla_{e_k}\Big)(x,\xi)  \,\, =  \,\,  
\varsigma\Big(\sum_{j,k} f_j   \cdot f_k \cdot \nabla_{e_j} \nabla_{e_k}\Big)(x,\xi)
\,\, = \,\,
$$
$$ - \sum_{j}   \varsigma(\nabla_{e_j}  \nabla_{e_j})(x, \xi)  \,\, + \,\, 
 \sum_{j<k}  f_j \wedge  f_k \wedge \varsigma(\nabla_{e_j}  \nabla_{e_k} - \nabla_{e_k}  \nabla_{e_j}) (x, \xi)  \,\, = \,\,
$$
$$
|\eta|^2  \,\, - \,\,  i\sum_j e_{j,x} e_j\langle  \exp^{-1}_{x}(x'),\xi \rangle  \,\, +  \,\, 
\sum_{j<k}  f_j \wedge f_k \wedge\frac{1}{2}\Omega_{\nu} (e_j,e_k)   \,\, +  \,\, 
$$
$$
\sum_{j<k}  f_j \wedge f_k \otimes \Omega_E (e_j,e_k)    \,\, +  \,\,     
\sum_{j<k}  f_j \wedge f_k \otimes  i \langle  [e_j,e_k],\xi \rangle.
$$
\end{proof}
 
\begin{remark}\label{Riemannian}
The term  $\sum_{j<k} f_j \wedge f_k   \wedge  \Omega (e_j,e_k)$ has order zero when $F$ is Riemannian, since in that case, it is  locally the pull-back of the same expression on any transversal $W$.   But it is classical, see \cite{LawsonM}, p.\ 161, that on $W$ this expression is just $\frac{1}{4} \kappa$, where $\kappa$ is the scalar curvature of $W$.  
\end{remark}
 
The example treated in Proposition \ref{dsquared}  accords well with the case of a foliation by points, the case in \cite{Getz}, where the symbol of the square of the Dirac operator is given as $\varsigma(D^2)(x,\xi) \,\, = \,\,   - |\xi|^2 +  \frac{1}{2}\Omega_E  +  $     lower graded terms.
The minus sign occurs because Getzler uses the convention $f \cdot f = \, \langle  f,f \rangle$ rather than$f \cdot f = - \langle  f,f\rangle$ which we use.  In this case, the transverse operator $\what{D}$ is self adjoint, so the terms involving $\mu$ disappear.  The tangent bundle $TF$ is the zero bundle, so the term involving $\vartheta_{\nu}$ also disappears.   

\section{Composition of polynomial symbols} \label{polythm}

In this section, we concentrate  on polynomial symbols,  that is, symbols associated to differential operators, and an important result for us is Proposition 3.7 of \cite{ABP}.  The set up there is the following.  Let $x= (x_1,...,x_n)$ be normal coordinates at the point $x$ and $u_1,u_2, ...$ a local framing of $\cS_{\nu} \otimes E$ obtained by parallel translating a framing at $x$ along the geodesics through $x$.  Then with respect to this data, the local connection and curvature forms are defined by the equations
$$
\nabla u_i \,\, = \,\,  \sum_{jk} \Gamma^i_{jk} dx_k \otimes u_j\; \text{ and }\;
\Omega  u_i \,\, = \,\,  \sum_{jkl}K^i_{jkl} dx_k \wedge dx_l \otimes u_j.
$$
The $\Gamma^i_{jk}$ and $K^i_{jkl}$ are smooth locally defined functions on $M$, which are related as follows.

\begin{proposition}[\cite{ABP}, Proposition 3.7]\label{ABPprop} 
Write $\what{\Gamma}$ and $\what{K}$ for the formal Taylor series at $x$ for the function indicated, and $\what{\Gamma}[n]$ and $\what{K}[n]$ for the term of homogeneity $n$ in this expansion.  Then
$$
(n+1) \what{\Gamma}^i_{jk}[n]  \,\, = \,\,  \sum_{l}2 x_l  \what{K}^i_{jlk}[n-1].
$$
In particular, the Taylor series for $\Gamma^i_{jk}$ at $x$ is given by 
$$
\what{\Gamma}^i_{jk} \,\, = \,\,  x_l K^i_{jlk}(x)  \,\, + \,\, 
a_m x_l   x_m\frac{\pa K^i_{jlk}}{\pa x_m}(x) \,\, + \,\, 
a_{mn}  x_l   x_m x_n \frac{\pa^2 K^i_{jlk}}{ \pa x_m \pa x_n} (x)\,\, + \,\,  \cdots
$$
where $a_{\bullet} \in \Q$, and the expression on the right is summed over repeated indices.
\end{proposition}

As corollaries of  this fundamental proposition, we quote  the following facts for later use. Their proofs are straightforward and are omitted.
\medskip
\begin{enumerate}

\item All the terms in $\what{\Gamma}^i_{jk} $ have grading at most $2$, and those of grading $2$ are given by the $(K_{\nu})^i_{jlk}$ and their derivatives, where   $K_{\nu}$  is defined by $\Omega_{\nu}  u_i \,\, = \,\,  \sum_{jkl}(K_{\nu})^i_{jkl} dx_k \wedge dx_l \otimes u_j.$

\medskip
\item  Suppose that  $X \in C^{\infty}(TM)$.  Then $\nabla_{X} u_i (x)=0$ and we have more explicitely
$$
\nabla_{X} u_i \,\, = \,\,  \theta^i_j(X) u_j   \,\, = \,\, \Gamma^i_{jk}dx_k(X) u_j \,\, = \,\,
\Big(x^{\ell} K^i_{jlk}(x) + a_m  x_l   x_m\frac{\pa K^i_{jlk}}{\pa x_m}(x) + \cdots \Big) dx_k(X) u_j.
$$

\item   Suppose $X,Y \in C^{\infty}(TM)$.  Then
$$
\nabla_Y\nabla_X u_i (x)\,\, = \,\,   \Big(K^i_{jlk} dx_l(Y)dx_k(X) u_j\Big)  (x)  \,\, = \,\,  \frac{1}{2} (\Omega(X,Y) u_i) (x).
$$

\item   
$\Omega(X,Y) \,\, = \,\,\frac{1}{2}\Omega_{\nu} (X,Y) \,\, + \,\,  \Omega_E(X,Y)$,  and as operators the first has grading $2$, while the second has grading zero.

\item 
If $Z \in C^{\infty}(TM)$, then $\nabla_{Z}\nabla_{Y}\nabla_{X} u_i(x)$ has the term $\dd a_m  dx_l (Z)  dx_m(Y)dx_k(X)\frac{\pa K^i_{jlk}}{\pa x_m}(x) u_j$, which is not necessarily zero.  Thus $\varsigma(\nabla_{Z}\nabla_{Y}\nabla_{X})(x,\xi,\sigma)$ has a multiple of the term $dx_m(Y)\pa \Omega(Z,X)/\pa x_m$, which has grading at most two.  However, the terms of highest grading in $\varsigma(\nabla_{Z}\nabla_{Y}\nabla_{X})$ will have grading $3$, e.\ g.\ $\langle i Z, \xi \rangle \langle i Y, \xi \rangle \langle i X, \xi \rangle$.  Similar remarks apply to higher compositions of covariant derivatives.  Thus, our calculations below of the highest graded terms of the symbols of compositions of covariant derivatives will contain no derivatives of the $K^i_{jlk}$, and we may assume that $\what{\Gamma}^i_{jk}  =  x_l (K_{\nu})^i_{jlk}(x)$.

\end{enumerate}

\begin{remark}\label{threefacts}
It is clear that the following three facts hold:
\begin{enumerate}

\item
The highest graded terms of $\varsigma(\nabla^m_{X} \nabla^{\ell}_Y )(x, \xi,\sigma)$ have grading $m + \ell$.  

\item
The highest graded term of $(X^{m - k}Y^{\ell - k'}e^{i\langle  \exp^{-1}_{x}(x'),(\xi,\sigma) \rangle})(x) $ is $\langle i X,  (\xi,\sigma)\rangle^{m-k}  \langle  iY, (\xi,\sigma)\rangle^{\ell-k'} $.  

\item
The highest graded term of $(\nabla^k_X\nabla^{k}_Y(u_j))(x)$ is  $4^{-k} \Omega_{\nu}(X,Y)^k $. 
\end{enumerate}
\end{remark}

For non-negative integers $k \leq m,  k' \leq \ell$,  set
$$
\Big(\hspace{-0.1cm}\begin{array}{c} m,  \ell \\ k,  k' \end{array}\hspace{-0.1cm}\Big) \,\, = \,\,
\Big(\hspace{-0.1cm}\begin{array}{c} m \\k\end{array}\hspace{-0.1cm}\Big) 
\Big(\hspace{-0.1cm}\begin{array}{c} \ell \\k'\end{array}\hspace{-0.1cm}\Big)\,\, = \,\, 
 \frac{m!\ell!}{k !(m-k)!{k'}!(\ell - k')!}.
$$

\begin{lemma}\label{sigexp1.5}  Suppose $X, Y \in C^{\infty}(TM)$ are given in normal coordinates at $x$ by $X = \sum_{i=1}^n c_i \pa/\pa x_i$ and  $Y= \sum_{i=q+1}^n d_i \pa/\pa x_i$ where the $c_i$ and $d_i$ are constants.  So $Y(x)  \in TF_x$, and that is part of the data.  Then
$$
\varsigma(\nabla^m_{X} \nabla^{\ell}_Y )(x, \xi,\sigma)\,\, = \,\,
$$
$$
\sum_{k=0}^{\min(m,\ell)} 
\hspace{-0.2cm}  
4^{-k} k!\Big(\hspace{-0.1cm}\begin{array}{c} m,\ell \\k,k\end{array}\hspace{-0.1cm}\Big)
\langle  iX, \xi \rangle^{m-k}  \langle  iY, \sigma\rangle^{\ell-k} 
\Omega_{\nu}(X,Y)^k \,\, + \,\,
c_{m,\ell}(X,\xi,Y, \sigma),
$$
where $c_{m,\ell}(X,\xi,Y, \sigma)$ is polynomial in $\xi$ and $\sigma$ and has grading less than $m + \ell$.
\end{lemma}

\begin{proof}   
Note,
$$
\nabla^m_{X} \nabla^{\ell}_Y \Big( e^{i\langle  \exp^{-1}_{x}(x'),\xi \rangle} u_j \Big)(x) \,\, = \,\,  
\sum_{k =0}^m \sum_{k' =0}^{\ell} 
\hspace{-0.0cm}\Big(\hspace{-0.1cm}\begin{array}{c} m,\ell \\k,k'\end{array}\hspace{-0.1cm}\Big)
(X^{m - k}Y^{\ell - k'}e^{i\langle  \exp^{-1}_{x}(x'),\xi \rangle})(x)  (\nabla^k_X\nabla^{k'}_Y(u_j))(x).
$$
As $\Omega(X,X)(x) = \Omega(Y,Y)(x) = 0$, in order to get a term of grading $m + \ell$, we must have $k' = k$.  Then the term of grading $2k$ in $(\nabla^k_X\nabla^{k}_Y(u_j))(x)$  is $4^{-k} k!  \Omega_{\nu}(X,Y)^k$, and the term of grading $m+\ell-2k$ in $(X^{m - k}Y^{\ell - k}e^{i\langle  \exp^{-1}_{x}(x'),\xi \rangle})(x)$ is $\langle  iX, \xi \rangle^{m-k}  \langle  iY, \sigma\rangle^{\ell-k} $.
\end{proof}

We say an operator has grading $k$ if its symbol has that grading.

\begin{lemma}\label{crucial}
Suppose $X$ and $Y$ are as in Lemma \ref{sigexp1.5}.  Then modulo operators of lower grading
$$
\theta^{\alpha}(\langle  iX, \xi \rangle^m  \langle  iY, \sigma\rangle^{\ell}) \,\, = \,\,     
\sum_{k=0}^{\min(m,\ell)} 
\hspace{-0.2cm}  
(-4)^{-k}  k!\Big(\hspace{-0.1cm}\begin{array}{c} m,\ell \\k,k\end{array}\hspace{-0.1cm}\Big)
\nabla_X^{m-k} \nabla_Y^{\ell-k} 
\Omega_{\nu}(X,Y)^k.
$$
\end{lemma}

\begin{proof}  
For $m=0$ or $\ell =0$, this is just Lemma \ref{sigexp1}.   So, we need only assume that it is true for $0 \leq r < m$ and $0 \leq s <\ell$, and then prove it for $m,\ell$.  In what follows, we ignore operators of grading lower than $m + \ell$.

Using the fact that for differential operators and polynomial symbols $\theta^\alpha \circ \varsigma = \I$,   Corollary \ref{smdifop}, and applying $\theta^{\alpha}$  to the formula in the previous lemma, we have
$$
\nabla^{m}_{X} \nabla^{\ell}_Y   \,\, = \,\,
\sum_{k=0}^{\min(m,\ell)} 
\hspace{-0.2cm}  
4^{-k} k!\Big(\hspace{-0.1cm}\begin{array}{c} m,\ell \\k,k\end{array}\hspace{-0.1cm}\Big)
\theta^{\alpha}\Big[\langle  iX, \xi \rangle^{m-k}  \langle  iY, \sigma\rangle^{\ell-k} \Big]
\Omega_{\nu}(X,Y)^k.
$$
Set $s = \min(m,\ell)$, and rewrite this as
$$
\theta^{\alpha}\Big[\langle  iX, \xi \rangle^{m}  \langle  iY, \sigma\rangle^{\ell} \Big]
 \,\, = \,\,
\nabla^{m}_{X} \nabla^{\ell}_Y   \,\, - \,\, \sum_{k=1}^{s} 
4^{-k} k!\Big(\hspace{-0.1cm}\begin{array}{c} m,\ell \\k,k\end{array}\hspace{-0.1cm}\Big)
\theta^{\alpha}\Big[\langle i X, \xi \rangle^{m-k}  \langle  iY, \sigma\rangle^{\ell-k} \Big]
\Omega_{\nu}(X,Y)^k.
$$
Using the induction hypothesis, the second term on the right hand side equals
$$
- \sum_{k=1}^{s}  \sum_{k'=0}^{s-k}    
4^{-k} k!\Big(\hspace{-0.1cm}\begin{array}{c} m,\ell \\k,k\end{array}\hspace{-0.1cm}\Big)
(-4)^{-k'}  k'!\Big(\hspace{-0.1cm}\begin{array}{c} m-k,\ell -k\\k',k'\end{array}\hspace{-0.1cm}\Big)
\nabla_X^{m-k-k'} \nabla_Y^{\ell-k-k'} 
\Omega_{\nu}(X,Y)^{k+k'} \,\, = \,\, 
$$
$$ 
\sum_{k=1}^{s}  
\sum_{k'=0}^{s-k}    
(-1)^{k+1} (-4)^{-(k+k')} \frac{ (k+k')!}{k! k'!}\Big(\hspace{-0.1cm}\begin{array}{c} m,\ell \\k+k',k+k'\end{array}\hspace{-0.1cm}\Big)(k+k')!
 \nabla_X^{m-(k+k')} \nabla_Y^{\ell-(k+k')} 
\Omega_{\nu}(X,Y)^{k+k'} \,\, = \,\,
$$
$$ 
\sum_{r=1}^{s} \Big[ \sum_{k'=0}^{r-1} 
(-1)^{r-k'+1} \frac{ r!}{(r-k')! k'!} \Big]  (-4)^{-r} r!\Big(\hspace{-0.1cm}\begin{array}{c} m,\ell \\ r,r\end{array}\hspace{-0.1cm}\Big)
 \nabla_X^{m-r} \nabla_Y^{\ell-r} 
\Omega_{\nu}(X,Y)^r \,\, = \,\,
$$
$$ 
\sum_{r=1}^{ \min(m,\ell)} (-4)^{-r} r! \Big(\hspace{-0.1cm}\begin{array}{c} m,\ell \\ r,r\end{array}\hspace{-0.1cm}\Big)
 \nabla_X^{m-r} \nabla_Y^{\ell-r} 
\Omega_{\nu}(X,Y)^r.
$$
\end{proof}

Now we extend Theorem 2.7, \cite{Getz},  on composing symbols in $SC^{\infty,\infty} (M, E)$ which are polynomial in $\xi$ and $\sigma$.   This is just another application of \cite{ABP}, Proposition 3.7.  First, we have some notation.  Let $p, q \in SC^{\infty,\infty} (M, E)$ be two such symbols, and set 
$$
p \circ q \,\, = \,\, \varsigma(\theta^{\alpha}p \circ \theta^{\alpha}q).
$$
Next write 
$p = \sum_j   \omega_{p,j}  \otimes  p_j$, where $\omega_{p,j} \in C^{\infty}(\wedge^* \nu^*)$ and $p_j \in S^{\infty,\infty}(M, E)$ and  similarly  write
$q = \sum_k \omega_{q,k}  \otimes  q_k$.  Then set
$$
p(x,\xi,\sigma) \wedge q(x,\xi' ,\sigma')  \,\, = \,\,  \sum_{j,k} \omega_{p,j} \wedge\omega_{q,k} \otimes  p_j(x,\xi ,\sigma)q_k(x,\xi' ,\sigma').
$$
Note  that we are taking the usual wedge product of the form part of the symbols here, not the Clifford product.

Let $e_1,...,e_n$ be a local orthonormal basis of $TM$ with dual orthonormal basis $f_1,...,f_n$ of $T^*M$, with $f_1,...,f_q$ a local basis for $\nu^*$.  Set, as usual,
$$
\Omega_{\nu}(e_i,e_j) \,\, = \,\, \sum_{i,j,=1}^n  \sum_{k,\ell=1}^q (\Omega_{\nu})^k_{\ell,i,j} e_k \otimes f_{\ell}, \; \text{ that is }  \;  (\Omega_{\nu})^k_{\ell,i,j} \,\, = \,\, \langle \Omega_{\nu}(e_i,e_j)(e_{\ell}), e_k \rangle,
$$
and note that $(\Omega_{\nu})^k_{\ell,i,j}$ is skew in the indices $i,j$ (since $\Omega_{\nu}$ is a 2-form) as well as the $k,\ell$,
(since $\Omega_{\nu}$ has coefficients in $so_q = spin_q)$.

Set 
$$
\Omega_{\nu}(\pa/\pa \xi,\pa/\pa \xi') p(x,\xi ,\sigma) \wedge q(x,\xi' ,\sigma') \,\, =  
\sum_{i,j=1}^n  \sum_{k,\ell=1}^q(\Omega_{\nu})^k_{\ell,i,j} f_k \wedge f_{\ell} \wedge
\frac{\pa p(x,\xi ,\sigma)}{\pa \xi_i } \wedge \frac{\pa q(x,\xi' ,\sigma')}{\pa \xi '_j },
$$
so $e^{-\frac{1}{4}\Omega_{\nu}(\pa/\pa \xi,\pa/\pa \xi') }$ is actually a finite sum of compositions of such operators, and the number of compositions is $\leq q/2$ because of the $f_k \wedge f_{\ell}$.  We also set 
$$
\Omega_{\nu}(\pa/\pa \xi,\pa/\pa \sigma) p(x,\xi ,\sigma) \wedge q(x,\xi' ,\sigma') \,\, =  
\sum_{k,\ell=1}^q\sum_{i=1}^n\sum_{j=q+1}^n (\Omega_{\nu})^k_{\ell,i,j} f_k \wedge f_{\ell} \wedge
\frac{\pa^2 p(x,\xi ,\sigma)}{\pa \xi_i \pa \sigma_j} \wedge q(x,\xi' ,\sigma'),
$$
and the similarly defined operators $\Omega_{\nu}(\pa/\pa \xi,\pa/\pa \sigma')$, $\Omega_{\nu}(\pa/\pa \sigma, \pa/\pa \xi')$, $\Omega_{\nu}(\pa/\pa \sigma,\pa/\pa \sigma')$, and $\Omega_{\nu}(\pa/\pa \xi',\pa/\pa \sigma')$.  
\begin{definition}\label{Omega}
Set
$$
\Omega_{\nu}(\pa/\pa (\xi,\sigma),\pa/\pa (\xi',\sigma')) \,\, = \,\,   
 \Omega_{\nu}(\pa/\pa \xi,\pa/\pa \xi') \, + \,
 \Omega_{\nu}(\pa/\pa \xi,\pa/\pa \sigma')  \, + \, 
 \Omega_{\nu}(\pa/\pa \sigma, \pa/\pa \xi')\, + \, \Omega_{\nu}(\pa/\pa \sigma,\pa/\pa \sigma').
 $$
\end{definition}

We are now in position to state our {{first main result which is a}} foliation version of Theorem 2.7 of \cite{Getz} (see also Theorem 2.1 of \cite{BF}):

\begin{theorem}\label{comp}
Let $p \in SC^{m,\ell} (M, E)$ and $q \in SC^{m', \ell'} (M, E)$ be polynomial in $\xi$ and $\sigma$.  There are differential operators $a_k$, $k \geq 0$ on the bundle $SC^{\infty,\infty} (M, E) \otimes_{C^{\infty}(M)}  SC^{\infty,\infty} (M, E)$ so that if we denote by $a_k(p,q)$ the image of $a_k(p \otimes q)$ under 
$$
SC^{\infty,\infty} (M, E) \otimes_{C^{\infty}(M)}  SC^{\infty,\infty} (M, E) \longrightarrow SC^{\infty,\infty} (M, E),
$$
 the fiberwise composition of endomorphisms, then

\begin{enumerate}

\item
$\dd a_k(p,q) \in \sum_{k_1+k_2 = k}SC^{m + m' - k_1,  \ell + \ell' - k_2} (M, E)$;

\item
$p \circ q \,\, = \,\, \sum_0^{\infty} a_k(p,q)$, which is actually a finite sum;

\medskip
\item
$a_0(p,q)(x, \xi, \sigma)  \,\, = \,\,  e^{-\frac{1}{4}\Omega_{\nu}(\pa/\pa (\xi,\sigma),\pa/\pa (\xi',\sigma'))}p(x,\xi,\sigma) \wedge q(x, \xi',\sigma') \, |_{(\xi',\sigma') =  (\xi,\sigma)}$.
\end{enumerate}

\end{theorem}

Before giving the proof, we remark on two special cases.  If the dimension of $F$ is zero, that is the foliation is by points, this is Theorem 2.7 of \cite{Getz}.  In this case, $\sigma$, $\sigma'$, $\ell$, $\ell'$, and $k_2$ just disappear.  At the other end of the spectrum, when the foliation has maximal dimension, so it has a single leaf, we are in the case considered by Widom in \cite{Wid1}.   In this case $\xi$, $\xi'$, $m$, $m'$, $k_1$,  and $-\frac{1}{4}\Omega_{\nu}(\pa/\pa (\xi,\sigma),\pa/\pa (\xi',\sigma'))$ disappear, and we have
$$
a_0(p,q) \,\, = \,\,  p(x,\sigma)  q(x, \sigma),
$$
the first term in the formula in Corollary 4.11 of \cite{Wid1}.   The interesting new situations now occur in intermediate dimensions and for non trivial foliations.

\begin{proof}  To prove the theorem,  we proceed just as in the proof of Lemma \ref{sigexp1.5}, and we use the same notation.   Careful bookkeeping of the terms which are ignored in the calculation we do will prove parts $(1)$ and $(2)$ of the theorem, so what we will prove is part $(3)$.

The first step is to note that any symbol which is polynomial in $\xi$ and $\sigma$ can be written as a sum of symbols {of} the form $p_{m,\ell,r} = h \langle iX, \xi \rangle^m \langle iY, \sigma \rangle^{\ell}$, where $X$ and $Y$ are local sections of $TM$ as in Lemma \ref{sigexp1.5}, $h \in C^{\infty}(\wedge^r \nu^* \otimes \End(E))$, and $Y(x)$ is in $TF_x$, and that is part of the data.  Since $(p,q) \to p \circ q$ is linear in both variables, we may assume that $p = h \langle iX, \xi \rangle^m \langle iY, \sigma \rangle^{\ell}$ and   $ q = g \langle iW, \xi \rangle^{m'} \langle iZ, \sigma \rangle^{\ell'}$, where $g \in C^{\infty}(\wedge^{r'} \nu^* \otimes \End(E))$.  

For the simplest cases, say for $m = m' = 1$ and $\ell,\ell',r$ and $r'$ are zero, we may use Lemma \ref{twovecs} to get 
 $$
\langle iX, \xi \rangle \circ \langle iW, \xi \rangle \,\, = \,\, {\varsigma}(\nabla_X\nabla_W)(x,\xi) \,\, = \,\, \langle  iX,\xi\rangle\langle   iW,\xi\rangle  - \frac{1}{4}\Omega_{\nu} (iX,iW)  +  iX_{x}W\langle  \exp^{-1}_{x}(x'),\xi \rangle +  \frac{1}{2}\Omega_E (X,W).
$$
Then note that $\Omega_{\nu} (iX,iW) = \Omega_{\nu}(\pa/\pa (\xi,\sigma),\pa/\pa (\xi',\sigma')) \langle  iX,\xi\rangle\langle   iW,\xi' \rangle  \, |_{\xi'=\xi}$, and the last two terms have grading less than two.
  
For the general case, we use the three facts from Remark \ref{threefacts}, and the fact that operators of the form $\Omega_{\nu}(X,Y)$ commute with operators of the form $\nabla_Z$, modulo operators of lower grading.  Then we have that, modulo operators of lower grading, that is less than $m + \ell + r + m' + \ell' + r' $, 
$$
p \circ  q(x,\xi,\sigma) \,\, = \,\,   \varsigma \Big(\theta^{\alpha}(h \langle iX, \xi \rangle^m \langle iY, \sigma \rangle^{\ell}) \circ \theta^{\alpha}(g \langle iW, \xi \rangle^{m'} \langle iZ, \sigma \rangle^{\ell'})\Big)   \,\, = \,\, 
$$  
$$
\varsigma \Big(
\sum_{k=0}^{\min(m,\ell)} 
\hspace{-0.2cm} 
h(x)(-4)^{-k}  k!\Big(\hspace{-0.1cm}\begin{array}{c} m,\ell \\k,k\end{array}\hspace{-0.1cm}\Big)
\nabla_X^{m-k} \nabla_Y^{\ell-k} 
\Omega_{\nu}(X,Y)^k    
\sum_{k'=0}^{\min(m',\ell')} 
\hspace{-0.2cm} 
g(x)(-4)^{-k'}  k'!\Big(\hspace{-0.1cm}\begin{array}{c} m',\ell' \\k',k'\end{array}\hspace{-0.1cm}\Big)
\nabla_W^{m'-k'} \nabla_Z^{\ell'-k'} 
\Omega_{\nu}(W,Z)^{k'}\Big)   \,\, = \,\, 
$$
$$
gh\sum_{k, k'} (-4)^{-(k+k')}  k!\Big(\hspace{-0.1cm}\begin{array}{c} m,\ell \\k,k\end{array}\hspace{-0.1cm}\Big)
 k'!\Big(\hspace{-0.1cm}\begin{array}{c} m',\ell' \\k',k'\end{array}\hspace{-0.1cm}\Big)
\Omega_{\nu}(X,Y)^k  \Omega_{\nu}(W,Z)^{k'} 
\varsigma \Big(\nabla_X^{m-k} \nabla_Y^{\ell-k} \nabla_W^{m'-k'} \nabla_Z^{\ell'-k'} \Big),
$$
since operators of the form $X^a Y^b (g)$ {{have grading at most $r' \leq r'+a+b$.}}  Thus we may ignore $g$ and $h$, and we only need the terms of highest grading of $\varsigma \Big(\nabla_X^{m-k} \nabla_Y^{\ell-k} \nabla_W^{m'-k'} \nabla_Z^{\ell'-k'} \Big)$, i.\ e.\ those terms of grading $m+ \ell + m' + \ell' -2(k+k')$. An argument as in the proof of Lemma  \ref{sigexp1.5}  then shows that these terms are given by
\begin{multline*}
\sum 4^{-\what{k}} a! b! c! d! e! f!
\Big(\hspace{-0.1cm}\begin{array}{c} m-k \\ a, b, c \end{array}\hspace{-0.1cm}\Big) 
\Big(\hspace{-0.1cm}\begin{array}{c} \ell-k \\  a, d, e \end{array}\hspace{-0.1cm}\Big) 
\Big(\hspace{-0.1cm}\begin{array}{c} m' -k' \\ b, d, f \end{array}\hspace{-0.1cm}\Big) 
\Big(\hspace{-0.1cm}\begin{array}{c} \ell' -k' \\ c, e, f\end{array}\hspace{-0.1cm}\Big) 
\\
\langle iX, \xi \rangle^{m-k_X}  \langle iY, \sigma \rangle^{\ell-k_Y}\langle iW, \xi\rangle^{m'-k_W} \langle iZ, \sigma \rangle^{\ell'-k_Z} 
\\
\Omega_{\nu}(X,Y)^a   \Omega_{\nu}(X,W)^b   \Omega_{\nu}(X,Z)^c   \Omega_{\nu}(Y,W)^d   \Omega_{\nu}(Y,Z)^e   \Omega_{\nu}(W,Z)^f,
\end{multline*}

\noindent
where $\what{k} = a+b+c+d+e+f$, $k_X =  k+a+b+c$, $k_Y = k+a+d+e$, $k_W = k'+b+d+f$, $k_Z = k'+c+e+f$, and the sum is taken over all  $a,b,c,d,e,f \geq 0$ so that $m-k_X$, $\ell-k_Y$, $m'-k_W$, and $\ell'-k_Z $ are all $\geq 0$.  We have used here the convenient notation
$$
\Big(\hspace{-0.1cm}\begin{array}{c}  m \\ a, b, c \end{array}\hspace{-0.1cm}\Big) \,\, = \,\, 
\frac{m!}{a!b!c!(m-(a+b+c))!}.
$$
Substituting in the expression above for $p \circ q$ (and ignoring $gh$), gives
$$
a_0(\langle iX, \xi \rangle^m \langle iY, \sigma \rangle^{\ell}, \ \langle iW, \xi \rangle^{m'} \langle iZ, \sigma \rangle^{\ell'})  \,\, = \,\, 
$$
\begin{multline*}
\sum 4^{\what{k}} (-4)^{-(k+k')} k!\Big(\hspace{-0.1cm}\begin{array}{c} m,\ell \\k,k\end{array}\hspace{-0.1cm}\Big)
 k'!\Big(\hspace{-0.1cm}\begin{array}{c} m',\ell' \\k',k'\end{array}\hspace{-0.1cm}\Big)
a! b! c! d! e! f!
\Big(\hspace{-0.1cm}\begin{array}{c} m-k \\ a, b, c \end{array}\hspace{-0.1cm}\Big) 
\Big(\hspace{-0.1cm}\begin{array}{c} \ell-k \\  a, d, e \end{array}\hspace{-0.1cm}\Big) 
\Big(\hspace{-0.1cm}\begin{array}{c} m'-k' \\ b, d, f \end{array}\hspace{-0.1cm}\Big) 
\Big(\hspace{-0.1cm}\begin{array}{c} \ell'-k' \\ c, e, f\end{array}\hspace{-0.1cm}\Big) 
\\
\langle iX, \xi \rangle^{m-k_X}  \langle iY, \sigma \rangle^{\ell-k_Y}\langle iW, \xi\rangle^{m'-k_W} \langle iZ, \sigma \rangle^{\ell'-k_Z} 
\\
\Omega_{\nu}(X,Y)^{a+k}   \Omega_{\nu}(X,W)^b   \Omega_{\nu}(X,Z)^c   \Omega_{\nu}(Y,W)^d   \Omega_{\nu}(Y,Z)^e   \Omega_{\nu}(W,Z)^{f+k'}.
\end{multline*}

The operators in $\Omega_{\nu}(\pa/\pa (\xi,\sigma),\pa/\pa (\xi',\sigma'))$, along with the operators  
$\Omega_{\nu}(\pa/\pa \xi,\pa/\pa \sigma)$ and $\Omega_{\nu}(\pa/\pa \xi',\pa/\pa \sigma')$, all commute.
In addition, 
$e^{-\frac{1}{4}\Omega_{\nu}(\pa/\pa \xi,\pa/\pa \sigma)}e^{\frac{1}{4}\Omega_{\nu}(\pa/\pa \xi,\pa/\pa \sigma)} = e^{-\frac{1}{4}\Omega_{\nu}(\pa/\pa \xi',\pa/\pa \sigma')}e^{\frac{1}{4}\Omega_{\nu}(\pa/\pa \xi',\pa/\pa \sigma')} = \I$.
Thus, 
$$
e^{-\frac{1}{4}\Omega_{\nu}(\pa/\pa (\xi,\sigma),\pa/\pa (\xi',\sigma')) }  \,\, = \,\, 
$$
$$
e^{-\frac{1}{4}\Omega_{\nu}(\pa/\pa \xi,\pa/\pa \xi')} 
e^{-\frac{1}{4}\Omega_{\nu}(\pa/\pa \xi,\pa/\pa \sigma') }
e^{-\frac{1}{4}\Omega_{\nu}(\pa/\pa \sigma, \pa/\pa \xi')}
e^{-\frac{1}{4}\Omega_{\nu}(\pa/\pa \sigma,\pa/\pa \sigma') }
$$
$$
\hspace{1.5cm}e^{-\frac{1}{4}\Omega_{\nu}(\pa/\pa \xi',\pa/\pa \sigma')}
e^{-\frac{1}{4}\Omega_{\nu}(\pa/\pa \xi,\pa/\pa \sigma)}
e^{\frac{1}{4}\Omega_{\nu}(\pa/\pa \xi',\pa/\pa \sigma')}
e^{\frac{1}{4}\Omega_{\nu}(\pa/\pa \xi,\pa/\pa \sigma)}.
$$
Now note, for example, that 
$$
e^{\frac{1}{4}\Omega_{\nu}(\pa/\pa \xi',\pa/\pa \sigma')}e^{\frac{1}{4} \Omega_{\nu}(\pa/\pa \xi,\pa/\pa \sigma)} \Big( \langle iX, \xi \rangle^m  \langle iY, \sigma \rangle^{\ell}\langle iW, \xi \rangle^{m'} \langle iZ, \sigma \rangle^{\ell'}\Big) 
\,\, = \,\,
$$
$$
\sum_{k,k'} (-4)^{-k} k! 
\Big(\hspace{-0.1cm}\begin{array}{c} m , \ell \\ k, k \end{array}\hspace{-0.1cm}\Big) 
\langle iX, \xi \rangle^{m-k}  \langle iY, \sigma \rangle^{\ell-k}\Omega_{\nu}(X,Y)^k 
(-4)^{-k'} k'! 
\Big(\hspace{-0.1cm}\begin{array}{c} m' , \ell' \\ k', k' \end{array}\hspace{-0.1cm}\Big) 
\langle iW, \xi \rangle^{m'-k'} \langle iZ, \sigma \rangle^{\ell'-k'}\Omega_{\nu}(W,Z)^{k'}.
$$

A similar computation gives that 
$$
e^{-\frac{1}{4}\Omega_{\nu}(\pa/\pa (\xi,\sigma),\pa/\pa (\xi',\sigma')) } e^{-\frac{1}{4}\Omega_{\nu}(\pa/\pa \xi',\pa/\pa \sigma')}
e^{-\frac{1}{4}\Omega_{\nu}(\pa/\pa \xi,\pa/\pa \sigma)}
\Big(\langle iX, \xi \rangle^{m-k}  \langle iY, \sigma \rangle^{\ell-k}\langle iW, \xi \rangle^{m'-k'} \langle iZ, \sigma \rangle^{\ell'-k'}\Big) \,\, = \,\, 
$$
\begin{multline*}
\sum 4^{-\what{k}} a! b! c! d! e! f!
\Big(\hspace{-0.1cm}\begin{array}{c} m-k \\ a, b, c \end{array}\hspace{-0.1cm}\Big) 
\Big(\hspace{-0.1cm}\begin{array}{c} \ell-k \\  a, d, e \end{array}\hspace{-0.1cm}\Big) 
\Big(\hspace{-0.1cm}\begin{array}{c} m' -k' \\ b, d, f \end{array}\hspace{-0.1cm}\Big) 
\Big(\hspace{-0.1cm}\begin{array}{c} \ell' -k' \\ c, e, f\end{array}\hspace{-0.1cm}\Big) 
\\
\langle iX, \xi \rangle^{m-k_X}  \langle iY, \sigma \rangle^{\ell-k_Y}\langle iW, \xi\rangle^{m'-k_W} \langle iZ, \sigma \rangle^{\ell'-k_Z} 
\\
\Omega_{\nu}(X,Y)^a   \Omega_{\nu}(X,W)^b   \Omega_{\nu}(X,Z)^c   \Omega_{\nu}(Y,W)^d   \Omega_{\nu}(Y,Z)^e   \Omega_{\nu}(W,Z)^f.
\end{multline*}
Combining these we finally get the desired equality
$$
a_0(\langle iX, \xi \rangle^m \langle iY, \sigma \rangle^{\ell}, \ \langle iW, \xi \rangle^{m'} \langle iZ, \sigma \rangle^{\ell'}) \,\, = \,\, 
e^{-\frac{1}{4}\Omega_{\nu}(\pa/\pa (\xi,\sigma),\pa/\pa (\xi',\sigma')) } \Big(\langle iX, \xi \rangle^m \langle iY, \sigma \rangle^{\ell}\langle iW, \xi \rangle^{m'} \langle iZ, \sigma \rangle^{\ell'}\Big).
$$ 
\end{proof}
 \medskip
In the case where $F$ is a Riemannian foliation, the formula for $a_0(p,q)(x, \xi, \sigma)$ has a particularly simple form.    Note that in this case, we use a Bott connection on $\nu^*$.  It is then immediate that  whenever $Y \in TF$, $\Omega_{\nu}(Y,X) = 0$, since the curvature is locally the pull back of the curvature on any transversal.  Write $\xi = (\eta,\zeta)$ where $\eta$ is the projection of $\xi$ to $\nu^*$.  Then all the terms of 
$$
\Omega_{\nu}(\pa/\pa (\xi,\sigma),\pa/\pa (\xi,\sigma'))  \,\, = \,\, 
\Omega_{\nu}(\pa/\pa (\eta, \zeta,\sigma),\pa/\pa (\zeta',\eta',\sigma'))
$$
are zero except  $\Omega_{\nu}(\pa/\pa \eta,\pa/\pa \eta')$, and we have
$$
a_0(p,q)(x, \eta, \zeta, \sigma) \,\, = \,\,  e^{-\frac{1}{4}\Omega_{\nu}(\pa/\pa \eta,\pa/\pa \eta')}
p_0(x,\eta, \zeta, \sigma) \wedge q_0(x, \eta', \zeta, \sigma) \, |_{\eta'=\eta},
$$
and the operator $\Omega_{\nu}(\pa/\pa \eta,\pa/\pa \eta')$ is identical to the one in \cite{Getz} and \cite{BF}.

As examples, we show how Theorem \ref{comp} applies to  pure form operators and to the operator $D^2$.  Suppose that  $ p = \omega_{\alpha} \otimes 1$ and $q = \omega_{\beta} \otimes 1$ are pure form operators of degree ${\alpha}$ and ${\beta}$.  By Equation \ref{wedge},
$$
\varsigma(  \omega_{\alpha} \otimes 1 \circ  \omega_{\beta} \otimes 1)  \,\, = \,\,
 \omega_{\alpha}  \cdot  \omega_{\beta} \otimes 1  \,\, = \,\,
$$
$$
 \omega_{\alpha} \wedge\omega_{\beta} \otimes  1
  \,\, \pm \,\,   
  \sum_k (i_{e_k}  \omega_{\alpha} ) \wedge (i_{e_k}  \omega_{\beta} ) \otimes  1   \,\, \pm \,\,   
\sum_{k,\ell} (i_{e_k} i_{e_{\ell}} \omega_{\alpha} ) \wedge (i_{e_k} i_{e_{\ell}}  \omega_{\beta} )  \otimes  1  \,\, + \,\,
\cdots
$$
All of the terms on the right, except the first one (which has grading  $\alpha + \beta$), have grading less than $\alpha + \beta$.  This is precisely why we change from the Clifford product to the usual wedge product in this theorem.  Note that 
$$
a_0(p,q)   =   \omega_{\alpha} \wedge \omega_{\beta}, \;\;  a_1(p,q)  = 
 \pm    \sum_k (i_{e_k}  \omega_{\alpha} ) \wedge (i_{e_k}  \omega_{\beta} ), \;\;
 a_2(p,q)  =   \pm   
\sum_{k,\ell} (i_{e_k} i_{e_{\ell}} \omega_{\alpha} ) \wedge (i_{e_k} i_{e_{\ell}}  \omega_{\beta} ),  \;   \ldots
$$

Now consider $D^2$, and assume for simplicity that $F$ is Riemannian and that $\vartheta_{\nu} = 0$.  Then we have $\varsigma(D)  \in SC^{2,0}(M,E)$, and modulo terms with lower grading,
$$
\varsigma(D)\circ \varsigma(D)(x,\xi,\sigma)  \,\, = \,\, 
e^{-\frac{1}{4}\Omega_{\nu}(\pa/\pa (\xi,\sigma),\pa/\pa (\xi',\sigma')) }
 \varsigma(D)(x,\xi,\sigma) \wedge \varsigma(D)(x, \xi',\sigma') \, |_{(\xi',\sigma') = (\xi,\sigma)}.
$$
We need to be careful about what ``terms with lower grading" means.  As $\varsigma(D)  \in SC^{2,0}(M,E)$,   we have $a_0(\varsigma(D),\varsigma(D)) \in SC^{4,0}(M,E)$.   In this case, $a_0(\varsigma(D),\varsigma(D))$ will not contain the terms of $\varsigma(D)\circ \varsigma(D)$ of maximal grading, so the terms of maximal grading will come from $a_n(\varsigma(D),\varsigma(D))$ for $n > 0$.  Now, $a_0(\varsigma(D), \varsigma(D))$ is given by
$$
 \varsigma(D)(x,\xi,\sigma) \wedge \varsigma(D)(x, \xi',\sigma') \,\, - \,\,\frac{1}{4}\Omega_{\nu}(\pa/\pa (\xi,\sigma),\pa/\pa (\xi',\sigma')) \varsigma(D)(x,\xi,\sigma) \wedge \varsigma(D)(x, \xi',\sigma') \,  |_{(\xi',\sigma') = (\xi,\sigma)} \,\, = \,\, 
$$ 
$$
(i \sum_{j=1}^q   f_j   \otimes \,  \xi_j  - \frac{1}{2} \mu ) \wedge (i \sum_{k=1}^q  f_k  \otimes \,  \xi'_k  - \frac{1}{2} \mu ) \,\, - \,\,  \frac{1}{4}\sum^q_{i,j,k,\ell=1}(\Omega_{\nu})^k_{\ell,i,j} f_k \wedge f_{\ell}  \wedge if_i \wedge i f_j \,  |_{(\xi',\sigma') = (\xi,\sigma)} \,\, = \,\,
$$
$$
- \sum_{j,k=1}^q  f_j \wedge  f_k \otimes \,  \xi_j \xi_k  \,\, - \,\,
\frac{i}{2}\Big( \sum_{k=1}^q  \mu \wedge f_k  \otimes \,  \xi_k + \sum_{j=1}^q  f_j   \wedge \mu \otimes \,  \xi_k \Big) \,\, - \,\, 
 \frac{1}{4}\sum^q_{i,j,k,\ell=1}(\Omega_{\nu})^k_{\ell,i,j} f_k \wedge f_{\ell}  \wedge if_i \wedge i f_j   \,\, = \,\,
$$
$$
 \frac{1}{4}\sum^q_{i,j,k,\ell=1}(\Omega_{\nu})^k_{\ell,i,j} f_k \wedge f_{\ell}  \wedge f_i \wedge  f_j   \,\, = \,\,
\frac{1}{2}  \sum_{j<k} f_j \wedge f_k \wedge \Omega_{\nu}(e_j,e_k) \,\, = \,\, \frac{1}{8}\kappa,  
$$
since $F$ is Riemannian.
So, in this case, $a_0(\varsigma(D),\varsigma(D)) $ gives the term which has a chance of having grading four, but actually has  grading zero.  The reason this happens, that is $a_0(\varsigma(D), \varsigma(D))$ gives us no information about  the terms of maximal grading of $\varsigma(D^2)$, will be clarified in the next section and it is, as in the classical case, that $tD$ is not an A$\Psi$DO, while $t^2D^2$ is an A$\Psi$DO (because $\vartheta_{\nu} = 0$).  See the remarks about this on page 26 of \cite{BF}.

\section{Asymptotic pseudodifferential operators and their symbol calculus}\label{Asymptotic}

In this section we develop a symbol calculus for asymptotic pseudodifferential operators adapted to the foliation $F$.  To do this, we extend the operator $\theta^{\alpha}(p)$ defined over $M$ to an operator defined over $TF$.  This must be done with some care so that the crucial relationship given in Proposition \ref{ext} holds.  We then extend some of the material in \cite{BF} to our case, and for the sake of brevity, quote several results from that paper and refer the reader to it and its references for the proofs.  

Our basic problem is to compute the symbol in $SC^{\infty,\infty}(M,E)$ of the composition of two operators constructed out of two symbols in $SC^{\infty,\infty}(M,E)$,  and we want the formula to depend only on the two symbols, just as in Theorem \ref{comp}.  In addition, we want a way to be able to recover the symbol of an operator which comes from a symbol in $SC^{\infty,\infty}(M,E)$.   As the calculations in Section \ref{Sodo} make clear, the procedure we used there for polynomial symbols will not work in general.  The solution to this problem is to make the variable $\sigma$ correspond to a space variable, and it is based on a simple idea.   Consider the symbol $p(x,\xi,\sigma) = i^{|\alpha| + |\beta|}\xi^{\alpha} \sigma^{\beta}$ defined over $\R^n$,  which acts on functions on $\R^n$ as the differential operator $\pa/\pa  x^{\alpha +\beta}$, whose symbol is $i^{|\alpha| + |\beta|}\xi^{\alpha +\beta}$.  This is not what we want.   Replace $\R^n$ with $\R^n \times \R^n$ with coordinates $(x,y)$.  Now let $f(x)$ be a function on $\R^n$ and define the function $\what{f}(x,y) = f(x+y)$.  Then $p(x, \xi,\sigma)$ acts on $\what{f}$ as the differential operator $\pa/\pa  x^{\alpha}  \pa/\pa  y^{\beta}$, whose symbol is $p(x, \xi,\sigma)$, just what we want.  In addition,   $\pa/\pa  x^{\alpha}  \pa/\pa  y^{\beta}(\what{f})$  restricted to the first $\R^n$ (that is, set $y=0$) is  $\pa/ \pa  x^{\alpha +\beta}(f)$, the action of $p(x,\xi,\sigma)$ on functions on $\R^n$.

To proceed, we replace $M$ by the manifold $TF$, and we note that there are equivalences of bundles
$$
T(TF)   \simeq    \pi^*(TM \oplus TF ) \simeq  \pi^*(TF \oplus \nu \oplus TF) 
$$
and 
$$  
T^*(TF) \simeq  \pi^*(T^*M  \oplus T^*F) \simeq   \pi^*(T^*F \oplus \nu^* \oplus T^*F),
$$
where $\pi: TF \to M$ is the projection.  These depend on the choice of a transverse bundle to the bundle along the fibers of $TF$.

Set $\what{E} = \pi^*(E)$  and $\what{E}_{\cS} = \pi^*(\cS_{\nu} \otimes E)$, and denote by $\pi_T:T^*(TF) \to TF$ the projection.  
\begin{definition}
The symbol space $S^{m,\ell}(TF, \what{E}_{\cS})$ consists of all $p \in  C^{\infty}( T^*(TF), \pi^*_T(\End  \what{E}_{\cS_{\nu}}))$  so that for any multi-indices $\alpha$, $\beta$ and $\lambda$,  there is a constant $C_{\alpha,\beta,\lambda} > 0$ so that 
$$
||  \,   \pa^{\alpha}_{\xi} \pa^{\beta}_{\sigma}  \pa^{\lambda}_{x,X}p(x,X,\xi,\sigma) \, || \,\, \leq \,\, 
C_{\alpha,\beta,\lambda}(1 +  |\xi|)^{m- |\alpha|}(1 + |\sigma|)^{\ell-|\beta|}.
$$

The symbol space $SC^{m,\ell} (TF, \what{E}_{\cS})$ is 
$$
SC^{m,\ell} (TF, \what{E}_{\cS}) \,\, = \,\,   \sum_{k=0}^q  S^{m-k,\ell} (TF, \what{E}_{\cS}) \cap 
C^{\infty}(T^*(TF),  \pi_T^*(\wedge^k \pi^*(\nu^*) \otimes  \End(\what{E}))).
$$
Elements of $SC^{m,\ell} (TF, \what{E}_{\cS}) $ have grading $m + \ell$.

Set  $SC^{\infty,\infty}(TF, \what{E}_{\cS})=\bigcup_{m,\ell} SC^{m,\ell} (TF, \what{E}_{\cS})$ and 
$SC^{-\infty,-\infty}(TF, \what{E}_{\cS})=\bigcap_{m,\ell} SC^{m,\ell} (TF, \what{E}_{\cS})$.
\end{definition}

{{The topology on $SC^{m,\ell} (TF, \what{E}_{\cS})$, which is induced from  $S^{m,\ell}(TF, \what{E}_{\cS})$, is just the analog of the topology on $SC^{m,\ell} (M, E)$. }}  

For $ Z = (Z_1,Z_{\nu},Z_2) \in T(TF)_{(x,X)}$,  set $x' = \exp_{x}(Z_1+ Z_2,Z_{\nu})$ and
$X' = \mathcal{T}_{x,x'}(X + Z_2)$, where $\mathcal{T}_{x, x'}:TF_x \to TF_{x'}$ is the parallel translation of the bundle $TF$ along the geodesic $t \to  \exp_{x}(tZ_1 +tZ_2,tZ_{\nu})$ in $M$ from $x$ to $x'$.  Define  $\exp_{(x,X)}:T(TF)_{(x,X)} \to TF$ as 
$$
\exp_{(x,X)}(Z) \,\,=\,\,  (x', X') \,\,=\,\,( \exp_{x}(Z_1+ Z_2,Z_{\nu}),\mathcal{T}_{x,\exp_{x}(Z_1+ Z_2,Z_{\nu})}(X + Z_2) ).
$$
So $\exp_{(x,X)}(0) = (x,X)$, as it should.  

Next, define 
 $$
 \wtit{\alpha}((x,X), (x',X')) = \alpha(x,x')\alpha(x, \exp_{x}( \mathcal{T}^{-1}_{x,x'}(X') - X)),
 $$ for 
$(0,\mathcal{T}^{-1}_{x,x'}(X')-X)$ in the component of $(\pi,\exp)^{-1}(\Supp \, \alpha)$ which contains the zero section. Otherwise, $\wtit{\alpha}((x,X), (x',X')) = 0$.  So for $x'$ not close to $x$ or $\mathcal{T}^{-1}_{x,x'}(X')$  not close to $X$,  $ \wtit{\alpha}((x,X), (x',X'))=0$.  Note that 
$$
\wtit{\alpha}((x, X),\exp_{(x,X)}(Z)) = \alpha(x, \exp_{x}(Z_1+Z_2,Z_{\nu})) \alpha(x,\exp_{x}(Z_2,0)),
$$
which is non-zero only for $Z$ close to zero, and does not depend on $X$. 

Let $u$ be a section of \  $\what{E}_{\cS}$, and set
$$
\overline{u}_{(x,X)}(Z) \,\, = \,\,  \wtit{\alpha}((x, X),\exp_{(x,X)}(Z)) \ 
\mathcal{T}^{-1}_{x, \exp_{x}(Z_1+Z_2,Z_{\nu})} u(\exp_{(x,X)}(Z)),
$$
which is an element of $(\what{E}_{\cS})_{(x,X)} = (\cS_{\nu} \otimes E)_x$.  Given $p \in SC^{\infty,\infty}(TF, \what{E}_{\cS})$, define the operator
$\what{\theta}^{\alpha}(p)$ on $u$  to be
$$
\what{\theta}^{\alpha}(p)(u)(x,X) \,\, = \,\, (2 \pi )^{-n-p}   \int_{T^*(TF)_{(x,X)} \times 
T(TF)_{(x,X)}}                     \hspace{-2.5cm}
e^{-i\langle Z,(\xi,\sigma) \rangle}p(x,X,\xi ,\sigma) \overline{u}_{(x,X)}(Z)\,  dZ d\sigma d\xi. 
$$
Note that any element $p \in SC^{\infty,\infty}(M,E)$ determines an element $\what{p} \in SC^{\infty,\infty}(TF,\what{E}_{\cS})$ by
$$
\what{p}(x,X, \xi ,\sigma) \,\, = \,\, p(x, \xi ,\sigma).
$$
In addition, a section $u$ of  $\cS_{\nu} \otimes E$ determines the section $\what{u}$ of  $\what{E}_{\cS}$ by setting $\what{u}(x, X) = u(x)$.  Note that 
$$
\overline{\what{u}}_{(x,X)}(Z_1 - Z_2, Z_{\nu},Z_2) \,\, = \,\, \alpha(x, \exp_{x}(Z_2,0)) 
\overline{u}_{x}(Z_1 , Z_{\nu}).
$$

\begin{proposition}  \label{ext}
Suppose that $u$ is a section of $\cS_{\nu} \otimes E$,  and $p, q \in SC^{\infty,\infty}(M,E)$.  Then for all $(x,X) \in TF$,
$$
\theta^{\alpha}(p)(u)(x)   \,\, = \,\,  \what{\theta}^{\alpha}(\what{p})(\what{u})(x,X),
$$
and 
$$
\theta^{\alpha}(p)\theta^{\alpha}(q)(u)(x)   \,\, = \,\,  \what{\theta}^{\alpha}(\what{p}) \what{\theta}^{\alpha}(\what{q})(\what{u})(x,X).
$$
\end{proposition}

\begin{proof}
First rewrite (we ignore the constant $(2 \pi )^{-n-p} $ throughout) $\theta^{\alpha}(p)(u)(x)$ as 
$$
\int_{TM_x \times T^*M_x \times TF_x \times T^*F_x} 
\hspace{-2.5cm}
e^{-i\langle  (Z_1,Z_{\nu}),\xi\rangle} e^{-i\langle Z_2, \sigma- \zeta \rangle}p(x,\xi ,\sigma)
\alpha(x, \exp_{x}(Z_2,0)) \alpha(x, \exp_{x}(Z_1,Z_{\nu}))
$$
$$ 
\mathcal{T}^{-1}_{x,\exp_{x}(Z_1,Z_{\nu})} u(\exp_{x}(Z_1,Z_{\nu}))\,   dZ d\sigma d\xi,
$$
where recall $\xi = (\eta,\zeta)$.
Now
$$
\what{\theta}^{\alpha}(\what{p})(\what{u})(x,X) \,\, = \,\,   \int_{TM_x \times T^*M_x \times 
TF_x \times T^*F_x} \hspace{-2.5cm}
e^{-i\langle  (Z_1,Z_{\nu}),\xi\rangle} e^{-i\langle Z_2, \sigma \rangle}p(x,\xi ,\sigma) \overline{\what{u}}_{(x,X)}(Z)\, dZ  d\sigma d\xi
\,\, = \,\,
$$
$$
\int_{TM_x \times T^*M_x \times 
TF_x \times T^*F_x} \hspace{-2.5cm}
e^{-i\langle  (Z_1,Z_{\nu}),\xi\rangle} e^{-i\langle Z_2, \sigma \rangle}p(x,\xi ,\sigma) 
\alpha(x,\exp_{x}(Z_2,0))   \alpha(x,\exp_{x}(Z_1 + Z_2,Z_{\nu}))
$$
$$ 
\mathcal{T}^{-1}_{x,\exp_{x}(Z_1 + Z_2,Z_{\nu})} u(\exp_{x}(Z_1+Z_2,Z_{\nu}))\,  dZ d\sigma d\xi.
$$
The first result then follows by making the change of coordinates $Z_1 \to Z_1 - Z_2$.

For the second, we use the formula obtained from the change of coordinates $Z_1 \to Z_1 - Z_2$ to get
$$
\what{\theta}^{\alpha}(\what{p})\what{\theta}^{\alpha}(\what{q})(\what{u})(x,X) \,\, = \,\,
\int_{TM_x \times T^*M_x \times 
TF_x \times T^*F_x} \hspace{-2.5cm}
e^{-i\langle  (Z_1,Z_{\nu}),\xi\rangle} e^{-i\langle Z_2, \sigma- \zeta \rangle}p(x,\xi ,\sigma) 
\alpha(x,\exp_{x}(Z_2,0))   \alpha(x,\exp_{x}(Z_1,Z_{\nu}))
$$
$$ 
\mathcal{T}^{-1}_{x,\exp_{x}(Z_1,Z_{\nu})}\what{\theta}^{\alpha}(\what{q})(\what{u})(x',X')\,   dZ  d\sigma d\xi,
$$
where $(x',X') = (\exp_{x}(Z_1,Z_{\nu}), \mathcal{T}_{x,\exp_{x}(Z_1,Z_{\nu})}(X + Z_2))$. Now
$$
\what{\theta}^{\alpha}(\what{q})(\what{u})(x',X') \,\, = \,\,
  \int_{TM_{x'} \times T^*M_{x'} \times 
TF_{x'} \times T^*F_{x'}} \hspace{-3.5cm}
e^{-i\langle  (Y_1,Y_{\nu}),\xi' \rangle} e^{-i\langle Y_2, \sigma'- \zeta' \rangle}q(x',\xi' ,\sigma') 
\alpha(x',\exp_{x'}(Y_2,0))   \alpha(x',\exp_{x'}(Y_1,Y_{\nu}))
$$
$$ 
\mathcal{T}^{-1}_{x',\exp_{x'}(Y_1,Y_{\nu})} {u}(\exp_{x'}(Y_1,Y_{\nu}))\,    dY d\sigma' d\xi'.
$$
Substituting this in the expression for $\what{\theta}^{\alpha}(\what{p})\what{\theta}^{\alpha}(\what{q})(\what{u})(x,X)$ and comparing the result with $\theta^{\alpha}(p)\theta^{\alpha}(q)(u)(x)$ immediately gives the second result.
\end{proof}

\begin{definition} 
A family {{$p(t) \in SC^{m,\ell}(M,E)$, $t \in \R$,}} is an asymptotic symbol if there are symbols $p_k$, of grading $m + \ell -k$ {{and independent of $t$}}, so that the following asymptotic expansion holds as $t \to 0$,
$$
p(t) \,\, \sim \,\, \sum_{k=0}^{\infty} t^{k} p_k.
$$
The leading symbol of $p(t)$ is $p_0$.  
There is an obvious extension of this definition to $p(t) \in SC^{m,\ell}(TF,\what{E}_{\cS})$.
\end{definition}

Note that  $p(t) \,\, \sim \,\, \sum_{k=0}^{\infty} t^{k} p_k$ means that given any $N > 0$, 
$$
\lim_{t \to 0} t^{-N}\Big( p(x,\xi,\sigma,t) \,\, - \,\, \sum_{k=0}^{N} t^{k} p_k(x,\xi,\sigma) \Big) \,\, = \,\, 0 
$$
in the space of symbols of grading $m+\ell - N-1$.   That is, we write $p(t) - \sum_{k=0}^{N} t^{k} p_k = \sum p_{k_1,k_2}$, where $k_1 + k_2 = N+1$,  $p_{k_1,k_2}  \in SC^{m-k_1,\ell-k_2}(M,E)$, and $\lim_{t \to 0} t^{-N}p_{k_1,k_2}=0 $ in $SC^{m-k_1,\ell-k_2}(M,E)$.  It does not imply that $\sum_{k=0}^{\infty} t^{k} p_k(x,\xi,\sigma)$ converges.

We identify two asymptotic symbols $p$ and $q$ which have the same asymptotic expansion, and write $p \sim q$.

\medskip
Note that if $p_0 \in SC^{m,\ell}(M,E)$ is any symbol, then $p(t) = p_0$ is an asymptotic symbol.  Similarly for $p_0 \in SC^{m,\ell}(TF,\what{E}_{\cS})$.  The following lemma is standard.

\begin{lemma}
Let $p_{n_1,n_2} \in SC^{m-n_1,\ell - n_2}(M,E)$, $n_1,n_2 = 0,1,2,...$.  Then there is an asymptotic symbol
$p(t) \in  SC^{m',\ell'}(M,E)$, for all $m',\ell' > m,\ell$,  such that  $\dd p(t) \sim \sum^{\infty}_{n_1+n_2=0}  \hspace{-0.3cm}t^{n_1 + n_2} p_{n_1,n_2}$.  
\end{lemma} 

\begin{proof}  This is a reasonably standard result, but it does require some care.    In particular,  choose a smooth non-decreasing function $\varphi$ on $\R$ with $\varphi(x) = 0$ for $x <1$, and $\varphi(x)=1$ for $x>2$.    Choose a decreasing sequence  $\ep_j $, with limit $0$.   Set
$$
p(x,\xi,\sigma,t) = \sum^{\infty}_{n=0} \hspace{0.1cm}\sum_{n_1+n_2=n}   t^{n} \varphi(\ep_{n}(t^{-2} + |\xi|^2 + |\sigma|^2))
 p_{n_1,n_2}(x, \xi, \sigma).
$$
For fixed $\xi$, $\sigma$ and $t$, this is actually a finite sum, so it converges. 
 
Next we show that $p(t) \in  SC^{m',\ell'}(M,E)$, for all $m',\ell' > m,\ell$.  So $t$ is now fixed, but not $x$, $\xi$ or $\sigma$.  Fix multi-indices $\alpha$, $\beta$ and $\lambda$.   Then, since $\varphi = 1$ for $| \xi |^2$ or $| \sigma |^2$ sufficiently large, 
$$
||  \,   \pa^{\alpha}_{\xi} \pa^{\beta}_{\sigma}  \pa^{\lambda}_{x}\varphi(\ep_{n_1 + n_2} (t^{-2} +|\xi|^2  +|\sigma|^2))                        p_{n_1,n_2}(x, \xi, \sigma)\, || \,\, \leq \,\, 
C_{n_1,n_2,\alpha, \beta,\lambda}(1 + |\xi|)^{m-n_1- |\alpha|}(1 + |\sigma|)^{\ell-n_2 - |\beta|}
 \,\, = \,\,
$$
$$
C_{n_1,n_2,\alpha, \beta,\lambda}(1 + |\xi|)^{m'-n_1 - |\alpha|}(1 + |\sigma|)^{\ell'- n_2 -|\beta|}
(1 + |\xi|)^{m - m'}(1 + |\sigma|)^{\ell - \ell'}.
$$
As we identify two asymptotic symbols if they have the same asymptotic expansion, and since symbols which are of {{uniformly fiberwise}}  compact $\xi,\sigma$ support are in $SC^{-\infty,-\infty}(M,E)$, we may assume that $p_{n_1, n_2} = 0$ on the compact set where
$$
C_{n_1,n_2,\alpha, \beta,\lambda}(1 + |\xi|)^{m - m'}(1 + |\sigma|)^{\ell - \ell'} \geq \frac{1}{(n_1 + n_2+1)!}.
$$
Then we have
$$
||  \,   \pa^{\alpha}_{\xi} \pa^{\beta}_{\sigma}  \pa^{\lambda}_{x}p(x,\xi,\sigma,t) \, || \,\, \leq \,\, 
 \sum_{n=0}^{\infty}  \frac{t^n}{n!}  (1 + |\xi|)^{m'- |\alpha|}(1 + |\sigma|)^{\ell'-|\beta|}  \,\, = \,\, 
e^{t}(1 + |\xi|)^{m' - |\alpha|}(1 + |\sigma|)^{\ell'-|\beta|},
$$
so $p(t) \in SC^{m',\ell'}(M,E)$, for all $m',\ell' > m,\ell$.

Now $\dd p(t) \sim  \sum^{\infty}_{n=0} \hspace{0.1cm}\sum_{n_1+n_2=n}t^n p_{n_1,n_2}$, since if $t < \sqrt{\ep_N/2}$, then $\varphi(\ep_n(t^{-2} + |\xi|^2  + |\sigma|^2) = 1$ for $n = 0,..., N$, so
$$
t^{-N}\Big( p(x,\xi,\sigma,t) \,\, - \sum^{N}_{n=0} \hspace{0.1cm}\sum_{n_1+n_2=n} \hspace{-0.3cm}  t^n p_{n_1,n_2}(x,\xi,\sigma ) \Big) \,\, = \,\,
$$
$$ 
t \Big[ \sum^{\infty}_{n=N+1} \hspace{0.1cm}\sum_{n_1+n_2=n} t^{n -(N+1)} \varphi(\ep_n(t^{-2} + |\xi|^2  + |\sigma|^2))p_{n_1,n_2}(x, \xi, \sigma) \Big].
$$ 
\end{proof}

Note that  $SC^{m,\ell}(M,E) \neq \cap_{m',\ell'>m,\ell} SC^{m',\ell'}(M,E)$.   A simple counterexample is given by the function $\sigma^2 \ln(1 + \sigma^2)$ on $\R^2$ (with coordinates $(\xi,\sigma)$), which  is in $S^{0,\ell'}$ for all $\ell' > 2$, but is not in $S^{0,2}$.

\medskip
The notions of asymptotically zero and equivalence of families of operators (Definitions 3.4 and 3.5 of \cite{BF}) translate directly to our situation.

{{{{Recall that smoothing operators are operators with smooth $C^\infty$-bounded uniformly supported Schwartz kernels.   These are called uniform smoothing operators in Appendix \ref{Bifiltered} and the support condition fits with Roe's definition of locally compact operators \cite{RoeBook}, see also \cite{BR}. Notice that smoothing operators are sometimes defined in the literature as those operators which  extend to  bounded operators between any  Sobolev spaces, without condition on the support, so as to include more general Schwartz functional calculus on elliptic operators. The reason we insist on uniform support is because we want them to furnish ideals in our uniformly supported bifiltered pseudodifferential  calculus. }}
 }}
\begin{definition}
A family of smoothing  operators $P_t$ on sections of $\cS_{\nu} \otimes E$ or $\what{E}_{\cS}$ is asymptotically zero if given any $N \geq 0$,  for all $s,k$,
$$
\lim_{t \to 0} t^{-N}|| P_t ||_{s,k} \,\, = \,\, 0.
$$ 
Two families of operators $P_t, Q_t$  are equivalent, written $P_t \sim Q_t$, if their difference is asymptotically zero. 
\end{definition}
Here, $|| P_t ||_{s,k}$ is the norm of $P_t$ as an operator from the usual $s$ Sobolev space associated to  $\cS_{\nu} \otimes E$ or $\what{E}_{\cS}$ to the usual $k$ Sobolev space. Recall that our manifold, as well as all the bundles we use, has bounded geometry, hence the Sobolev spaces are perfectly well defined on $M$ and on the total space $TF$, see for instance \cite{Shubin}.

\begin{definition}
{{Suppose that the family $p(t) \in SC^{\infty,\infty} (M, E)$ is smooth in $t$.  For $t > 0$, the  rescaling of $p(t)$, (denoted $p_t$, $p(t)_t$, or $p(x,\xi ,\sigma,t)_t$), is defined as follows.   It is the linear operator which, for 
$p(t) \in S^{m-k,\ell} (M, E) \cap 
C^{\infty}(T^*M   \oplus  T^*F,  \pi^*(\wedge^k \nu^* \otimes  \End(E)))$,  is given by
$$
p(x,\xi ,\sigma,t)_t  \,\, = \,\, t^{k} p(x,t\xi ,t\sigma,t).
$$
}}
\end{definition}

Again, there is an obvious extension of this to $SC^{\infty,\infty} (TF, \what{E}_{\cS})$.

\begin{definition}
An asymptotic pseudodifferential operator (A$\Psi$DO) is a family of operators $P_t$ on sections of $\cS_{\nu} \otimes E$ so that there is an asymptotic symbol $p(t) \in SC^{\infty,\infty}(M,E)$, with $P_t \, \sim \, \theta^{\alpha}(p(t)_t)$.

{{If $p(t) \,\, \sim \,\, \sum_{k=0}^{\infty} t^{k} p_k$, 
the leading symbol of $P_t$ is the symbol $p_0$.}}

Similarly for operators on $\what{E}_{\cS}$.
\end{definition}
We make no distinction between $P_t$ and its equivalence class.

\begin{definition}\label{gensym}
Given an operator $P$ on sections of $\cS_{\nu} \otimes E$ or $\what{E}_{\cS}$, its symbol $\varsigma (P)$ is defined as follows. Let $(x, X) \in TF$  and $(\xi, \sigma)  \in   T^*(TF)_{(x, X)} = T^*M_x \times T^*F_x$, and $u_{(x, X)} \in (\what{E}_{\cS})_{(x, X)}  =    (\cS_{\nu} \otimes E)_x$.    Set 
$$
\varsigma(P)(x, X, \xi,\sigma)(u_{(x, X)}) \,\, = \,\,  
P \Big((x',X') \mapsto e^{i\langle  \exp^{-1}_{(x,X)}(x',X'),(\xi,\sigma)\rangle}
\wtit{\alpha}((x,X),(x',X'))\mathcal{T}_{x,x'}(u_{(x, X)})\Big) \, |_{(x', X') =(x, X)}.
$$
\end{definition}

\begin{lemma}\label{simOps}
{{If  a family of smoothing operators $P_t$ on sections of $\cS_{\nu} \otimes E$ or $\what{E}_{\cS}$ satisfies $P_t \sim 0$, }} then  $\varsigma(P_t)\sim 0$.  
\end{lemma}

\begin{proof}
We do the proof for operators on sections of $\what{E}_{\cS}$, as the proof for  $\cS_{\nu} \otimes E$ is identical.

Suppose that  $P_t \sim 0$.  Given $u \in (\what{E}_{\cS})_{(x,X)}$, set 
$$
\wtit{u} (x',X')
\,\, = \,\,  
e^{i\langle  \exp^{-1}_{(x,X)}(x',X'),(\xi, \sigma) \rangle}
\wtit{\alpha}((x,X),(x',X')) \mathcal{T}_{x,x'}(u).
$$
Denote by $u_1,u_2, \ldots$ an orthonormal basis of $(\what{E}_{\cS})_{(x,X)}$.
Then
$$
\lim_{t \to 0}t^{-N} ||\varsigma(P_t)(x, X, \xi, \sigma)|| \,\, = \,\, 
\lim_{t \to 0}t^{-N}\sup_{||u|| = 1} ||\varsigma(P_t)(x, X,\xi, \sigma)(u)|| \,\, = \,\, 
\lim_{t \to 0}t^{-N}\sup_{||u|| = 1} ||P_t(\wtit{u} )(x,X)|| \,\, = \,\, 
$$
$$ 
\lim_{t \to 0}t^{-N}\sup_{||u|| = 1}\Big[\sum_i |\langle  P_t(\wtit{u} )(x,X), u_i\rangle|^2 \Big]^{1/2} \,\, = \,\, 
\lim_{t \to 0}t^{-N}\sup_{||u|| = 1}\Big[\sum_i |\langle  P_t(\wtit{u} ), \delta^{x,X}_{u_i}\rangle|^2 \Big]^{1/2}, 
$$
where $ \delta^{x,X}_{u_i}$ is the Dirac delta section with value $u_i$ at $(x,X)$.  Because of the bounded geometry of our situation, the $-k$ Sobolev norm $|| \delta^{x,X}_{u_i}||_{-k}$ is uniformly bounded,  provided $k$ is sufficiently large.  In addition, the sections $\wtit{u}$, where $||u|| = 1$, have $||\wtit{u}||_{0}$ uniformly bounded.  Then we have,
$$
\lim_{t \to 0}t^{-N}\sup_{||u|| = 1}\Big[\sum_i |\langle  P_t(\wtit{u} ), \delta^{x,X}_{u_i}\rangle|^2 \Big]^{1/2} \,\, \leq \,\,
\lim_{t \to 0}t^{-N}\sup_{||u|| = 1}\Big[\sum_i \Big( ||P_t||_{0,k} ||\wtit{u}||_0  || \delta^{x,X}_{u_i}||_{-k}\Big)^2 \Big]^{1/2} \,\,=\,\, 0.
$$

To estimate the norms of the derivatives of $\varsigma(P_t)$, we may proceed in a similar fashion, utilizing the Schwartz kernel $K_t$ of $P_t$.  In particular, 
$$
\lim_{t \to 0}t^{-N} || \pa^{\alpha}_{\xi} \pa^{\beta}_{\sigma} \pa^{\lambda}_{x,X}\varsigma(P_t)(x, X, \xi, \sigma)  || \,\, = \,\, 
\lim_{t \to 0}t^{-N} \sup_{||u|| = 1} || \pa^{\alpha}_{\xi} \pa^{\beta}_{\sigma} \pa^{\lambda}_{x,X}P_t(\wtit{u})(x,X)|| \,\, = \,\,
$$$$
\lim_{t \to 0}t^{-N} \sup_{||u|| = 1} || \pa^{\alpha}_{\xi} \pa^{\beta}_{\sigma} \pa^{\lambda}_{x,X} \int  K_t((x,X),(x',X'))\wtit{u} (x',X') dX' dx'||\,\, = \,\,
$$$$
\lim_{t \to 0}t^{-N} \sup_{||u|| = 1} ||  \int  \pa^{\lambda}_{x,X}K_t((x,X),(x',X')) (\pa^{\alpha}_{\xi} \pa^{\beta}_{\sigma} \wtit{u}) (x',X') dX' dx'||\,\, = \,\,
$$$$
\lim_{t \to 0}t^{-N} \sup_{||u|| = 1} ||   \pa^{\lambda}_{x,X}P_t(\pa^{\alpha}_{\xi} \pa^{\beta}_{\sigma} \wtit{u}) (x,X) ||  \,\, = \,\, 
\lim_{t \to 0}t^{-N}\sup_{||u|| = 1}\Big[\sum_i |\langle  \pa^{\lambda}_{x,X}P_t(\pa^{\alpha}_{\xi} \pa^{\beta}_{\sigma}  \wtit{u}), \delta^{x,X}_{u_i}\rangle|^2 \Big]^{1/2}    \,\, = \,\, 
$$$$
\lim_{t \to 0}t^{-N}\sup_{||u|| = 1}\Big[\sum_i |\langle  P_t(\pa^{\alpha}_{\xi} \pa^{\beta}_{\sigma} \wtit{u}), \pa^{\lambda}_{x,X}\delta^{x,X}_{u_i}\rangle|^2 \Big]^{1/2}.  
$$
As above, for large enough $k$,  the Dirac delta sections $\pa^{\lambda}_{x,X}\delta^{x,X}_{u_i}$ have $|| \pa^{\lambda}_{x,X}\delta^{x,X}_{u_i} ||_{-k}$ uniformly bounded, and $|| \pa^{\alpha}_{\xi} \pa^{\beta}_{\sigma} \wtit{u}) ||_0$ is uniformly bounded.    It follows immediately that
$$
\lim_{t \to 0}t^{-N} || \pa^{\alpha}_{\xi} \pa^{\beta}_{\sigma}\pa^{\lambda}_{x,X}\varsigma(P_t)(x, X, \xi, \sigma)  || \,\, = \,\, 0.
$$
\end{proof}

The proof of the following lemma  depends on Lemma \ref{BFresults2} and Theorem \ref{mainlemma}, and so is deferred to the appendix.
\begin{lemma}\label{simOps2}
 If a family of symbols $p(t) \in SC^{-\infty,-\infty}(TF,\what{E}_{\cS})$  and  $p(t) \sim 0$, then $P_t =\what{\theta}^{\alpha}(p(t)_t)\sim 0$.  The same result holds for $p(t) \in SC^{-\infty,-\infty}(M,E)$.
\end{lemma}

By $p(t) \sim 0$ we mean that for all $N$, $\lim_{t \to 0} t^{-N}p(t) \,\, = \,\, 0 $ in $SC^{-\infty,-\infty} (TF, \what{E}_{\cS})$.  Note that we are not assuming that $p(t)$ is an asymptotic symbol.

\medskip

We now translate some results from \cite{BF}, which are extensions of results in \cite{Getz},  to our situation.  For the following, note that there are two lemmas in \cite{BF} labeled 3.9.  We are interested in the second one on page 20.  This lemma will be important for the proof of our main theorem.

\begin{lemma}[Lemma 3.9, p.\ 20 of \cite{BF}] \label{BFresults3}
{{Suppose that 
$p(t) \in SC^{m,\ell}(TF,\what{E}_{\cS})$ is a bounded family of symbols (i.\ e.\ the symbol estimates are independent of $t$). }} Then
$$
[\varsigma(\what{\theta}^{\alpha}(p_t))(x, X,\xi,\sigma,t)]_{t^{-1}} \,\, \sim \,\, p(x,X,\xi,\sigma,t).
$$

Setting $X=0$, gives the same result  for $p(t) \in SC^{m,\ell}(M,E)$.
\end{lemma}

Note  that we are not assuming that $p$ is an asymptotic symbol.  Also, note that what we have to prove is that for all $N$,  $[\varsigma(\what{\theta}^{\alpha}(p_t))(x, X,\xi,\sigma,t)]_{t^{-1}}  -  p(x,X,\xi,\sigma,t)$ is of grading    $m+\ell-N-1$, and that 
$$
\lim_{t \to 0}t^{-N}\Big([\varsigma(\what{\theta}^{\alpha}(p_t))(x, X,\xi,\sigma,t)]_{t^{-1}}  -  p(x,X,\xi,\sigma,t)\Big) \,\, = \,\, 0,
$$
as described above.

\begin{remark}\label{inverses2}  This lemma says that $\varsigma$ is just what we want for our symbol operator as it allows us to recover (up to equivalence of symbols) the original symbol from its associated operator.   In particular, if $p(t)$ is polynomial in $\xi$ and $\sigma$, then the proof below actually shows that for all $t$, 
$$
[\varsigma(\theta^{\alpha}(p_t))(x, \xi,\sigma,t)]_{t^{-1}} \,\, = \,\, p(x, \xi,\sigma,t).
$$
This is because for $N$ large enough, the error term is zero.   In addition, just as in the previous section, we have that for any smooth differential operator $D$, $\theta^{\alpha} (\varsigma(D )) =  D$.  So we have that for differential operators and polynomial symbols, $\theta^{\alpha}$ and $\varsigma$ are inverses of each other.
\end{remark}

\begin{proof}  
We may use the proof of \cite{BF} mutatis mutandis.  Some typos in that proof are given in the remark immediately after this proof. 

First note that we may assume without loss of generality that $p$ has no form components.  For, if $p = \omega \otimes \wtit{p}$, where $\wtit{p}$ contains no form component, then   it follows easily that $\varsigma(\what{\theta}^{\alpha}(\omega \otimes \wtit{p})) =   \omega \otimes \varsigma(\what{\theta}^{\alpha} (\wtit{p}))$.  A straight forward computation gives  (we suppress constants)
$$
[\varsigma(\what{\theta}^{\alpha}(p_t))(x,X,\xi,\sigma,t)]_{t^{-1}} \,\, = \,\,\varsigma(\what{\theta}^{\alpha}(p_t))(x,X, \xi/t,\sigma/t,t)  \,\, = \,\, 
$$
$$
\int e^{-i\langle Z,(\what{\xi},\what{\sigma})\rangle} \wtit{\alpha}((x,X),\exp_{(x,X)}(Z))^2 
p(x,X, t\what{\xi} + \xi,t\what{\sigma} + \sigma,t)  dZ d\what{\sigma} d\what{\xi}.
$$

Now apply Taylor's formula to the variables $\what{\xi}$ and $\what{\sigma}$ in $p$ to obtain the formula 
$$
p(x,X,t\what{\xi}+\xi , t\what{\sigma} + \sigma,t) \,\, = 
\sum_{|(\alpha,\beta)| \leq N} \frac{t^{{|(\alpha,\beta)|}} \what{\xi}^{\alpha} \what{\sigma}^{\beta}}{\alpha ! \beta ! } 
(\pa^{\alpha}_{\what{\xi}}\pa^{\beta}_{\what{\sigma}}p)(x,X,\xi , \sigma,t)  \,\, + \,\, 
$$
$$\sum_{|(\alpha,\beta)| = N+1} \frac{N+1}{\alpha!\beta!} t^{N+1}\what{\xi}^{\alpha} \what{\sigma}^{\beta}   
\int_0^1 (\pa^{\alpha}_{\what{\xi}}\pa^{\beta}_{\what{\sigma}}p)(x, X, st\what{\xi} + \xi,  st\what{\sigma} + \sigma, t)ds.
$$
Just as in \cite{BF}, when we integrate with respect to $\what{\xi}$ and $ \what{\sigma}$, all the terms in the first sum disappear, except for the first one, namely $p(x,X,\xi , \sigma,t)$, so we have that 
$$
[\varsigma(\what{\theta}^{\alpha}(p_t))(x,X,\xi,\sigma)]_{t^{-1}} \,\, = \,\,
p(x,X,\xi , \sigma,t)  \,\, + \,\, \mathcal{E},
$$
where  the error term $\mathcal{E}$ is
$$
\int e^{-i\langle Z,(\what{\xi},\what{\sigma})\rangle} \wtit{\alpha}((x,X),\exp_{(x,X)}(Z))^2 
\hspace{-0.3cm}
\sum_{|(\alpha,\beta)| = N+1} \hspace{-0.2cm}\frac{N+1}{\alpha!\beta!} t^{N+1}\what{\xi}^{\alpha} \what{\sigma}^{\beta}   
\int_0^1 (\pa^{\alpha}_{\what{\xi}}\pa^{\beta}_{\what{\sigma}}p)(x,X, st\what{\xi} + \xi,  st\what{\sigma} + \sigma, t)ds
 dZ d\what{\sigma} d\what{\xi}.
$$
If  $|(\alpha,\beta)| = N+1$, then $\pa^{\alpha}_{\what{\xi}}\pa^{\beta}_{\what{\sigma}}p$ has grading $m + \ell -N-1$, and we can finish the proof just as in \cite{BF}.
\end{proof}

\begin{remark}  
Typos in the proof of Lemma 3.9, p.\ 20 of \cite{BF}.  The exponentials have the wrong sign.  In the last line of $(3.19)$, the $d\xi_m$ should be $d\eta_m$. In the first line of $(3.20)$,  the sum is missing the factor $\eta^{\alpha}_m$.   In the second line of $(3.20)$,  $(t\alpha)^m$ should be $t^{|\alpha|}\eta^{\alpha}_m$, and the function $r$ is missing the variable $t$.    In  $(3.21)$,  $\pa_{X_m}$ should be  $\pa ^{\alpha}_{X_m}$.  In the first line of $(3.21)$, the term $\phi(m,X_m)$ is missing.  In $(3.23)$,  the middle line is missing the term $\eta^{\alpha}_m$ under the integral sign.
\end{remark}

Lemmas \ref{simOps} and  \ref{BFresults3} immediately give the following.
\begin{corollary}
{{Suppose that $P_t$ is an A$\Psi$DO, with $P_t \sim \what{\theta}^{\alpha}(p(t)_t)$ or   $P_t \sim \theta^{\alpha}(p(t)_t)$.  Then
$\varsigma(P_t )_{t^{-1}} \sim p(t)$.}}
\end{corollary}

\begin{proposition}[Lemma 3.6 of \cite{BF}] \label{BFresults4}
Suppose that $\phi$ is a smooth function on $T(TF)$ which  is zero on a neighborhood of the zero section.  Assume further that for each $(x,X)$, $\phi$ is supported in a neighborhood of  zero in $T(TF)_{(x,X)}$ where  $\exp_{(x,X)}$ is a diffeomorphism.  Given an asymptotic symbol $p(t) \in SC^{\infty,\infty}(TF,\what{E}_{\cS})$, and a section $u$ of $\what{E}_{\cS}$, set
$$\what{P_t}(u)(x,X) \,\, = \,\, 
(2 \pi )^{-n-p}   \int_{T^*(TF)_{(x,X)} \times  T(TF)_{(x,X)}}   \hspace{-2.5cm}
e^{-i\langle Z,(\xi,\sigma) \rangle}p(x,X,\xi ,\sigma,t)_t \phi(Z)\mathcal{T}^{-1}_{x,\exp_{x}(Z_1+Z_2,Z_{\nu})} u(\exp_{(x,X)}(Z))\,  dZ d\sigma d\xi.
$$
Then $\what{P_t}$ is an asymptotically zero operator .

Similarly for $p(t) \in SC^{\infty,\infty}(M,E)$.
\end{proposition}

\begin{proof}
We may use the proof of \cite{BF}, with the following changes.  The Schwartz kernel of $\what{P}_t$ is given by 
$$
K_t((x,X),(x',X')) \,\, = \,\, 
$$
$$(2\pi)^{-n-p}\int_{T^*(TF)_{(x,X)}}  \hspace{-1.3cm}
e^{-i\langle\exp^{-1}_{(x,X)}(x',X'), (\xi,\sigma)\rangle}
p(x,X,\xi ,\sigma,t)_t \phi(\exp^{-1}_{(x,X)}(x',X')) 
J(\frac{dz}{d\vol}) \mathcal{T}^{-1}_{x,x'} \,  d\sigma d\xi,
$$
where $J$ is the Jacobian of the change in volume forms, and $d\vol$ is the volume form on $TF$.  The kernel $K_t$ is an element of $C^{\infty}(TF \times TF)$ which is supported in a bounded neighborhood of the diagonal and is zero on a smaller neighborhood of the diagonal.  As such, standard arguments for uniformly supported operators on manifolds with bounded geometry show that it is infinitely smoothing, see for instance \cite{Shubin}.   

To see that it is asymptotically zero, set $Z = \exp^{-1}_{(x,X)}(x',X')$, and note that  
$$
e^{-i\langle Z, (\xi,\sigma)\rangle} = ||Z||^{-2k}\Delta^k_{\xi,\sigma}e^{-i\langle Z, (\xi,\sigma)\rangle} \quad  \text{ and }
\quad \Delta^k_{\xi,\sigma}\Big(p(x,X,t\xi ,t\sigma,t)\Big) = t^{2k}(\Delta^k_{\xi,\sigma}p)(x,X,t\xi ,t\sigma,t),
$$
where $\Delta_{\xi,\sigma} = -(\sum \pa^2/ \pa \xi_j^2  +  \sum \pa^2/ \pa \sigma_k^2)$.  Then, using integration by parts repeatedly, we have
$$
\int_{T^*(TF)_{(x,X)}}  \hspace{-1.2cm}
e^{-i\langle Z, (\xi,\sigma)\rangle} 
p(x,X,\xi ,\sigma,t)_t \phi(Z) J(\frac{dz}{d\vol}) \mathcal{T}^{-1}_{x,x'}
\, d\sigma d\xi  \,\, = \,\, 
$$ 
$$
t^{2k} \int_{T^*(TF)_{(x,X)}}  \hspace{-1.2cm}
 \phi(Z)  ||Z||^{-2k}e^{-i\langle Z, (\xi,\sigma)\rangle}
(\Delta^k_{\xi,\sigma}p)(x,X,t\xi ,t\sigma,t)
J(\frac{dz}{d\vol}) \mathcal{T}^{-1}_{x,x'}
  \,d\sigma d\xi,
$$ 
where $k$ is as large as we please.  To finish the proof, proceed just as in \cite{BF}.
\end{proof}

We have immediately,
\begin{corollary}\label{ndalpha}
Suppose $p(t) \in SC^{\infty,\infty}(TF,\what{E}_{\cS})$ is an asymptotic symbol.  Then the equivalence class of the  A$\Psi$DO  $P_t \, = \, \what{\theta}^{\alpha}(p(t)_t)$  does not depend on the choice of $\alpha$.

Similarly, if $p(t) \in SC^{\infty,\infty}(M,E)$ is an asymptotic symbol,  the  the equivalence classes of the A$\Psi$DOs  $\theta^{\alpha}(p(t)_t)$ and $ \what{\theta}^{\alpha}(p(t)_t)$ do not depend on the choice of $\alpha$.
\end{corollary}

Note that Proposition \ref{BFresults4} actually implies more than this.  It implies that in the definition of $\theta$ or $\what{\theta}$, we may use different bump functions in the places where $\alpha$ occurs.  
 
\medskip  
Let $\pi_{\oplus}:T(TF) \oplus T^*(TF) \to TF$ be the projection.   The next technical result is essential to the proof of our main result, Theorem \ref{mainlemma}. 

\begin{lemma}[Lemma 3.8 of \cite{BF}]\label{BFresults2}\
Suppose that 
$
r(t) \in C^{\infty}( T(TF) \oplus T^*(TF), \pi_{\oplus}^*\End(\what{E}_{\cS})),
$
and that for any multi-indices $\alpha$, $\beta$,  and $\delta$,  there is a constant $C_{\alpha,\beta, \delta} > 0$ (independent of $t$) with
$$
||  \,   \pa^{\alpha}_{\xi} \pa^{\beta}_{\sigma} \pa^{\delta}_{x,X,Z}        r(x,X,Z,\xi,\sigma,t) \, || \,\, \leq \,\, 
C_{\alpha,\beta,\delta}(1 + |\xi|)^{m- |\alpha|}(1 + |\sigma|)^{\ell-|\beta|},
$$
{{with respect to a fixed atlas of normal coordinates.}}

Assume moreover that $r$ has an asymptotic expansion $r(t) \sim \sum_{k = 0}^{\infty} t^k r_k$.

For a section $u$ of $\what{E}_{\cS}$, set
$$
R_t(u)(x,X) \,\, = \,\,  (2 \pi )^{-n-p}    \int_{T^*(TF)_{(x,X)} \times T(TF)_{(x,X)}}                     \hspace{-2.0cm} 
e^{-i\langle  Z, (\xi,\sigma)\rangle} r(x,X,Z,\xi,\sigma,t)_t \overline{u}_{(x,X)}(Z)\, dZ d\sigma d\xi.
$$
Then $R_t$ is an A$\Psi$DO, in particular $R_t = \what{\theta}^{\alpha}(\what{r}_t)$, where
$$
 \what{r}(x,X,\xi,\sigma,t) \,\, = \,\, (2 \pi )^{-n-p}   \int_{T^*(TF)_{(x,X)} \times T(TF)_{(x,X)}}   \hspace{-2.0cm} 
 e^{-i\langle  Z,(\what{\xi},\what{\sigma})\rangle} r(x,X, Z, t \what{\xi}+\xi, t \what{\sigma}+\sigma,t)\, 
 dZ d\what{\sigma} d\what{\xi},
$$
and 
$$
\what{r}(x,X,\xi,\sigma,t) \sim \sum_{\alpha,\beta \geq 0} \frac{t^{|(\alpha,\beta)|}}{\alpha!\beta!}
\pa^{\alpha}_{(Z_1, Z_{\nu})}\pa^{\alpha}_{\xi} \pa^{\beta}_{Z_2} \pa^{\beta}_{\sigma} 
r(x,X,Z,\xi,\sigma,t) \, | \, _{Z=0}.
$$
Furthermore, the leading symbol of $\what{r}$ is $\what{r}_0(x,X,\xi,\sigma) = r_0(x,X,0,\xi,\sigma)$.
\end{lemma} 

\begin{proof}
Again, we may use the proof of \cite{BF} mutatis mutandis, with the following changes.

First note that the formula for $p(m,\xi_m,t)$ is missing the factor $t^{|\alpha|}$, and the definition of $q(m,\xi_m,\eta_m,t)$ should be
$$
q(m,\xi_m,\eta_m,t) \,\, = \,\,  \int e^{-i\langle  X_m,\eta_m\rangle} r(m,\xi_m,X_m,t) \, dX_m.
$$
Equation $(3.10)$ should be 
$$
R_t(s)(m) \,\, = \,\, \int e^{-i\langle  X_m,\eta_m\rangle}  p(m,t\eta_m,t) \hat{s}(X_m) \, dX_m \,d\eta_m.
$$
In the third line of $(3.11)$, $\xi^{-x}_m$ should be $|\xi_m|^{-x}$, where $x$ is an even positive integer, and $\pa^{x}_{X_m}$ should be $\Delta^{x/2}_{X_m}$, where $\Delta$ is the Laplacian for $X_m$.   In addition, the last $\xi$ in that line should be an $X$.   In the fifth line, the first $d- |\beta|$ should be $|d- |\beta||$.  Next, note that $(3.11)$ is really two inequalities, one for $|\xi_m| \geq 1$ and another for $|\xi_m| \leq 1$.   The first inequality is the one given, but $K$ should be replaced by $2^{x}K$.   The second is proven by deleting the third and fourth lines of $(3.11)$ and replacing $K C_{\alpha,\beta,x}$ by $2^{x}KC_{\alpha,\beta}$.  One then makes the appropriate changes in $(3.12)$, i.e.
$$
\rho_{\alpha,\beta}(T(r)) \,\, \leq \,\, 2^{x}K\Big( \rho_{\alpha,\beta,x}(r) + \rho_{\alpha,\beta}(r)   \Big).
$$

Equation  $(3.13)$ should be
$$
r(m, t\xi_m + \eta_m,X_m, t)  \,\, \sim \,\,    \sum_{\alpha\geq 0} \frac{t^{|\alpha|}}{\alpha!}(\pa^{\alpha}_{\xi_m}
r)(m,\eta_m,X_m,t) \xi^{\alpha}_m,
$$
whence comes the missing $t^{|\alpha|}$ in statement of Lemma 3.8 of \cite{BF}.  Note that $\pa^{\alpha}_{\xi_m} r$ can be replaced by  $\pa^{\alpha}_{\eta_m} r$ if we think of $r$ as being a function of $\eta$ instead of $\xi$.  Then in $(3.14)$, $1/\alpha!$ should be replaced by $t^{|\alpha|}/\alpha!$, and $\pa^{\alpha}_{\xi_m} r$ should be replaced by $\pa^{\alpha}_{\eta_m} r$.  This gives
$$
p(m,\eta_m,t)  \,\, \sim \,\,
\sum_{\alpha\geq 0} \frac{t^{|\alpha|}}{\alpha!} \pa^{\alpha}_{X_m}\pa^{\alpha}_{\eta_m}
 r(m,\eta_m,X_m,t) \, | \, _{X_m=0}.
$$
\end{proof}

\begin{proposition}\label{L3.9BF}[Lemma 3.9, p.\    19, \cite{BF}]
If $P_t$ is an A$\Psi$DO,  then its {{formal}}  adjoint is also an A$\Psi$DO.
\end{proposition}

\begin{proof}  
Suppose that $P_t \sim \theta^{\alpha}(p(t)_t)$, and that $u$ and $v$ are {{smooth compactly supported}} sections of $\cS_{\nu} \otimes E$.  Then (ignoring the $(2 \pi )^{-n-p}$)
\begin{multline*}
\blangle P_t u, v \brangle \,\, = \,\,
 \int_{M}    \int_{TM_x \times T^*M_x \times 
TF_x \times T^*F_x}  \hspace{-2.5cm}
 e^{-i\langle  X,\xi\rangle} e^{-i\langle  Y, \sigma -\zeta \rangle} \alpha(x,\exp_{x}(X)) \alpha(x,\exp_{x}(Y)) 
\\
\blangle p(x,\xi,\sigma,t)_t \mathcal{T}_{x,\exp_{x}(X)}^{-1}  u(\exp_{x}(X)), v(x) {\brangle} \,   dY d\sigma dX d\xi dx.
\end{multline*}
Set $x' = \exp_{x}(X)$ and  $X' = \exp_{x'}^{-1}(x)$.  
Since $\mathcal{T}$ (which is parallel translation along the geodesic $\exp_{x}(tX)$) is an isometry, and we may assume that $\alpha$ is symmetric, we have 
$$
\alpha(x, \exp_{x}(X)){\blangle} p(x,\xi,\sigma,t)_t \mathcal{T}_{x,\exp_{x}(X)}^{-1}  u(\exp_{x}(X)), v(x) {\brangle}\,\, = \,\,
$$
$$
\alpha(x', \exp_{x'}(X')){\blangle} u(x'), \mathcal{T}_{x,x'} p^*(x,\xi,\sigma,t)_t 
\mathcal{T}_{x',x} 
\mathcal{T}_{x',\exp_{x'}(X')}^{-1} 
v(\exp_{x'}(X')) {\brangle} \,\, = \,\,
$$
$$
{\blangle} u(x'), \mathcal{T}_{x,x'} p^*(x,\xi,\sigma,t)_t 
\mathcal{T}_{x',x} \alpha( x', \exp_{x'}(X'))
\mathcal{T}_{x',\exp_{x'}(X')}^{-1} 
v(\exp_{x'}(X')) {\brangle} \,\, = \,\,
$$
$$
{\blangle} u(x'), \mathcal{T}_{x,x'} p^*(x,\xi,\sigma,t)_t 
\mathcal{T}_{x',x} 
\overline{v}_{x'}(X') {\brangle}.
$$

On the support of $\alpha$, we identify $TM_x$ with $M$ and $M$ with $TM_{x'}$, using $\exp$.  Then  $dX dx = \mathcal{J}(x,X,x',X')  dx' dX'$, where the Jacobian $\mathcal{J}$ is a smooth function with all of its derivatives bounded on the support of $\alpha$.  Multiplying by a bump function which is $1$ on the support of $\alpha$ {{(with respect to a fixed atlas of normal coordinates)}}, we may assume that  $\mathcal{J}$ is zero off  a neighborhood of the support of $\alpha$.  We will incorporate any future Jacobians in $\mathcal{J}$.   Then 
$$
\blangle P_t u, v \brangle \,\, = \,\,   \int_{M}    \int_{TM_{x'} \times T^*M_x \times 
TF_x \times T^*F_x}  \hspace{-2.8cm}
e^{-i\langle X,\xi \rangle}e^{-i\langle  Y, \sigma -\zeta \rangle}   \alpha(x,\exp_{x}(Y))\mathcal{J}
{\blangle} u(x'), \mathcal{T}_{x,x'} p^*(x,\xi,\sigma,t)_t
\mathcal{T}_{x',x} \overline{v}_{x'}(X')  {\brangle}  \, dY d\sigma dX' d\xi dx'.
$$

Next we make the change of coordinates from $T^*M_x$ to $T^*M_{x'}$ using $\mathcal{T}$, and the change of coordinates from $TF_x \times T^*F_x$ to $TF_{x'} \times T^*F_{x'}$ using $\mathcal{T}$.   Note that this second $\mathcal{T}$ is parallel translation along the geodesic $\exp_{x}(tX)$ for the given bundles, so in general is not the same as the first $\mathcal{T}$ followed by projection.  Thus we have,
\begin{multline*}
\blangle P_t u, v \brangle  \,\, = \,\,    \int_{M}    \int_{TM_{x'} \times T^*M_x \times 
TF_x \times T^*F_x}  \hspace{-2.8cm}
e^{-i\langle  \mathcal{T}_{x,x'}(X),  \mathcal{T}_{x,x'}(\xi) \rangle} 
e^{-i\langle  \mathcal{T}_{x,x'}(Y), \mathcal{T}_{x,x'}(\sigma -\zeta) \rangle}  \alpha(x,\exp_{x}(Y))\mathcal{J}
\\
{\blangle} u(x'), \mathcal{T}_{x,x'} p^*(x,\xi,\sigma,t)_t 
\mathcal{T}_{x',x}\overline{v}_{x'}(X')  {\brangle} \,  dY d\sigma dX' d\xi  d x'.
\end{multline*}
Note that $ \mathcal{T}_{x,x'}(X) = -X'$.   Setting $\xi' = -\mathcal{T}_{x,x'}(\xi)$,  $\sigma' = \mathcal{T}_{x,x'}(\sigma - \zeta) + \zeta'$, and $Y' = \mathcal{T}_{x,x'}(Y)$, this becomes
\begin{multline*}
\int_{M}    \int_{TM_{x'} \times T^*M_{x'} \times 
TF_{T,x'} \times T^*F_{T,x'}}  \hspace{-3.8cm}  (-1)^n
e^{-i\langle  X',  \xi'  \rangle}   
e^{-i\langle Y', \sigma' -\zeta' \rangle}
\alpha(x,\exp_{x}(Y))\mathcal{J}\beta(x',\exp_{x'}(Y'))
\\
{\blangle} u(x'), \mathcal{T}_{x,x'} p^*(x,\xi,\sigma,t)_t 
\mathcal{T}_{x',x} \overline{v}_{x'}(X')  {\brangle} \,   dY' d\sigma'  d X' d\xi'  d x'.
\end{multline*}
In addition, we have introduced the function $\beta(x',\exp_{x'}(Y')$, which is a bump function which has value one whenever $\alpha(x,\exp_{x}(Y)) \neq 0$.   We are assured that such a $\beta$ exists because we can make the support of $\alpha$ as close to the diagonal as we please.  Because of Corollary \ref{ndalpha}, we may replace $\alpha$ in $\overline{v}_{x'}$ by $\beta$. 

Set
$$
r(x',\xi',\sigma',X',Y',t) \,\, = \,\,  [(-1)^n
\alpha(x,\exp_{x}(Y))\mathcal{J}
\mathcal{T}_{x,x'} p^*(x,\xi,\sigma,t)_t 
\mathcal{T}_{x',x}]_{t^{-1}},
$$
where $x = \exp_{x'}(X')$,  $X = - \mathcal{T}^{-1}_{x,x'}X'$, $Y = \mathcal{T}^{-1}_{x,x'}(Y')$, $\xi = -\mathcal{T}^{-1}_{x,x'}(\xi')$, $\sigma = \mathcal{T}^{-1}_{x,x'}(\sigma' - \zeta')  -  \pi_{TF} \mathcal{T}^{-1}_{x,x'}(\xi')$, and $\pi_{TF}:TM \to TF$ is the projection.  Then
$$
P_t^*(v)(x') \,\, = \,\,  (2\pi)^{-n-p}
 \int_{TM_{x'} \times T^*M_{x'} \times 
TF_{x'} \times T^*F_{x'}}  \hspace{-3.1cm}
 e^{-i\langle  X',\xi' \rangle} e^{-i\langle  Y', \sigma' -\zeta' \rangle}   r(x',\xi',\sigma',X',Y',t)_t \beta(x',\exp_{x'}(Y')) \overline{v}_{x'}(X')   \,  dY' d\sigma' d X' d\xi' . 
$$
Using the fact that $\mathcal{J} = 0$ off a neighborhood of the support of $\alpha$ and that $p$ has an asymptotic expansion, it follows immediately that  $r$ satisfies the conditions of  Lemma \ref{BFresults2} (where there is no $X$, and the role of $Z$ is played by $(X',Y'$)), and we are done.
\end{proof}

{{Recall that given symbols $p, q  \in SC^{\infty,\infty}( M,E)$, the symbol $a_0(p,q)$ is given by
$$
a_0(p,q) (x,\xi,\sigma)\,\, = \,\,   e^{-\frac{1}{4}\Omega_{\nu}(\pa/\pa (\xi,\sigma),\pa/\pa (\xi',\sigma'))}
p_0(x,\xi,\sigma) \wedge q_0(x, \xi',\sigma') \, |_{(\xi',\sigma')=(\xi,\sigma)}.
$$
See Definition \ref{Omega} for the definition of $e^{-\frac{1}{4}\Omega_{\nu}(\pa/\pa (\xi,\sigma),\pa/\pa (\xi',\sigma'))}
$.}}

{{A similar formula holds for $p,q \in SC^{\infty,\infty}(TF,\what{E}_{\maS})$.}}

\begin{definition}
Suppose that $P_t$ is an A$\Psi$DO with  leading symbol $p_0$. Then $P_t$ is asymptotically elliptic if the map $q \mapsto a_0(p_0,q)$ is invertible.
\end{definition}

{{Note that if $p  \in SC^{m,\ell}( M,E)$ and $q \in SC^{m',\ell'}( M,E)$ (respectively $p \in SC^{m,\ell}(TF,\what{E}_{\maS})$ and $q \in SC^{m',\ell'}(TF,\what{E}_{\maS})$),
then Theorem \ref{mainlemma} below (which does not depend on {{Proposition \ref{inverses})  }} implies that the symbol $a_0(p,\cdot)^{-1}(q)$ is an element of $SC^{m'-m,\ell'-\ell}( M,E)$  (respectively $SC^{m'-m,\ell'-\ell}(TF,\what{E}_{\maS})$).}}
It follows easily that we may use the proof of Theorem 3.1 in \cite{BF}, mutatis mutandis, to prove the following.
\begin{proposition}\label{inverses}
If $P_t$ is an asymptotically elliptic operator, then there is an {{ A$\Psi$DO}}  $Q_t$ such that $P_t \circ Q_t \sim I$.
\end{proposition}

\section{The main theorem}

In this section we prove our second {{main}}  result, which is the extension of  Theorem \ref{comp} to  $SC^{\infty,\infty} (M, E)$. This theorem is originally due to {{Getzler and was extended by Block-Fox}}.  
 
\begin{theorem}[Theorem 3.5 of \cite{Getz}, Lemma 3.10 of \cite{BF}]\label{mainlemma}

Let $p(t)$ and $q(t)$ be asymptotic symbols in $SC^{\infty,\infty}(M,E)$, with associated  A$\Psi$DOs $P_t = \theta^{\alpha}(p(t)_t)$ and $Q_t = \theta^{\alpha}(q(t)_t)$.  
Then
\begin{enumerate}
\item  $P_t \circ Q_t$ is an A$\Psi$DO.
\medskip
\item  The leading symbol of $P_t \circ Q_t$ is $a_0(p_0,q_0)$, where $p_0$ and $q_0$ are the leading symbols of $p$ and  $q$.  In particular,
$$
a_0(p_0,q_0)(x, \xi, \sigma)  \,\, = \,\,  e^{-\frac{1}{4}\Omega_{\nu}(\pa/\pa (\xi,\sigma),\pa/\pa (\xi',\sigma')) }p_0(x,\xi,\sigma) \wedge q_0(x, \xi',\sigma') \, |_{(\xi',\sigma') =  (\xi,\sigma)}.
$$
\end{enumerate}
{{A similar statement holds for asymptotic symbols $p(t), q(t) \in SC^{\infty,\infty}(TF,\what{E}_{\maS})$.}}
\end{theorem}

{{
\begin{remark}
It follows immediately that
$$
a_0(p_0,q_0)= \lim_{t \to 0} \varsigma(P_t \circ Q_t)_{t^{-1}}.
$$
\end{remark}
}}
\begin{remark}
In Theorem \ref{comp}, $p$ is the leading symbol of $P_t = \theta^{\alpha}(p_t)$, that is $p_0 = p$ in that theorem.
\end{remark}

Our proof follows that of \cite{BF}.  However, since our situation is more complicated, the proof is also more complicated.

\begin{proof}   We will be working with  $P_t = \what{\theta}^{\alpha}(p_t)$ and  $Q_t = \what{\theta}^{\alpha}(q_t)$, and then we will apply  Proposition \ref{ext} to get the result.  Essentially the same proof works for asymptotic symbols  $p(t), q(t) \in SC^{\infty,\infty}(TF,\what{E}_{\maS})$.  For simplicity, we will ignore the constants.  

Let $(x,X) \in TF$, and $u$ be a section of $\what{E}_{\cS}$.  Then {{we have the following equation}}
\begin{Equation}\label{comp1} 
 \hspace{0.25cm}
$\dd P_t \circ Q_t(u)(x,X) \,\, = \,\, 
\int_{T^*(TF)_{(x,X)}  \times T(TF)_{(x,X)}} \hspace{-2.7cm}
e^{-i\langle  Z,(\xi,\sigma)\rangle} p(x,\xi,\sigma ,t)_t \, \wtit{\alpha}((x,X),(x_1,X_1))
\mathcal{T}^{-1}_{x,x_1}$
\end{Equation} 
\hspace{1.5cm}
$\dd\Big{(}
\int_{T^*(TF)_{(x_1,X_1)}  \times T(TF)_{(x_1,X_1)}} \hspace{-3.2cm}
e^{-i\langle  Y,(\kappa,\varrho) \rangle} q(x_1,\kappa,\varrho,t)_t  \,
\wtit{\alpha}((x_1,X_1),(x_2,X_2))  \mathcal{T}^{-1}_{x_1,x_2}
u(\exp_{(x_1,X_1)}(Y))  dY d\varrho  d\kappa \Big{)}
 dZ d\sigma d\xi,
$

\medskip\noindent
where $(x_1,X_1) = \exp_{(x,X)}(Z)$, and $(x_2,X_2) = \exp_{(x_1,X_1)}(Y)$.

\medskip\noindent We want to write this as  
$$
\int_{T^*(TF)_{(x,X)}  \times T(TF)_{(x,X)}} \hspace{-2.7cm}
e^{-i\langle V,(\lambda,\mu)\rangle}r(x,X,\lambda,\mu,t)_t \, \beta((x,X),(x',X'))\mathcal{T}^{-1}_{x,x'}u(\exp_{(x,X)}(V))\,   dV   d\mu d\lambda, 
$$
where $(x',X') =  \exp_{(x,X)}(V)$, $r(t)$ is an asymptotic symbol, and $\beta$ is a bump function. Then we need to compute the leading symbol of $r(t)$.  To do the first, we make several changes of variables.  Again for simplicity, the products of the various Jacobians associated with our changes of variables will be denoted simply $\maJ$.   {{It is of course possible to keep track of the variables on which the various $\maJ$ depend, but this does not clarify the computation. What is important here is to check that because of bounded geometry and the fact that we can make the support of $\alpha$ as close to the diagonal as we please, all the derivatives of all the $\maJ$ are uniformly bounded.}}   

The map $\mathcal{T}_{x, x_1}:TM_x  \to TM_{x_1}$ and its dual $\mathcal{T}_{x, x_1}:T^*M_x  \to T^*M_{x_1}$, which are parallel translation{s} along the geodesic $t \mapsto \exp_{x}(tZ_1 + tZ_2, tZ_{\nu})$, are isometries.  We extend these  to all of  $T(TF)_{(x,X)} $ and $T^*(TF)_{(x,X)}$ as follows.   Given  $W = (W_1,W_{\nu},W_2) \in T(TF)_{(x,X)}$, set
$$
\mathcal{T}_{x,x_1}(W) \,\, = \,\, ( \mathcal{T}_{x, x_1}(W_1,W_{\nu}),0) +  (0,\mathcal{T}_{x, x_1}(0,W_2))  \,\, \in \,\,  T(TF)_{(x_1,X_1)},
$$
where $(0,W_2)$ and  $\mathcal{T}_{x, x_1}(0,W_2)$ are in $TM = \nu \oplus TF$ considered as the second two factors in $T(TF) \simeq  TF \oplus  \nu \oplus TF$.  Since $\mathcal{T}_{x,x} = \I$, $\mathcal{T}_{x,x_1}$ is an isomorphism for $x_1$ sufficiently close to $x$.  We then have the  dual map 
$$
\mathcal{T}_{x,x_1}:T^*(TF)_{(x,X)}  \to T^*(TF)_{(x_1,X_1)},
$$
so by definition, $\mathcal{T}$ preserves  the pairing of these bundles.  Set $V = \mathcal{T}_{x,x_1}^{-1}(Y)$, and $(\lambda,\mu) = \mathcal{T}_{x,x_1}^{-1}(\kappa, \varrho)$.  Then $\langle  V,(\lambda,\mu)\rangle \, = \, \langle  Y,(\kappa,\varrho)\rangle$, and 
$\mathcal{T}_{x,x_1}^{-1}(dY d\varrho d\kappa)= \mathcal{J}dV  d\mu d\lambda$, so the right side of Equation \ref{comp1} becomes
$$
\int_{T^*(TF)_{(x,X)}  \times T(TF)_{(x,X)}}  \hspace{-1.7 cm} e^{-i\langle  V,(\lambda,\mu)\rangle }  
\int_{T^*(TF)_{(x,X)}  \times T(TF)_{(x,X)}}  \hspace{-2.5 cm}  e^{-i\langle  Z,(\xi,\sigma)\rangle} p(x,\xi,\sigma,t)_t   \, \wtit{\alpha}((x,X), \exp_{(x,X)}(Z) )\mathcal{T}^{-1}_{x,x_1} \Big(q(x_1,\mathcal{T}_{x,x_1}(\lambda,\mu),t)_t\Big)   
$$
$$
\wtit{\alpha}((x_1,X_1),\exp_{(x_1,X_1)}(\mathcal{T}_{x,x_1}(V) ) )\mathcal{T}_{0,1,2}     \mathcal{T}^{-1}_{x,x_2}  (u(\exp_{(x,X)}(\phi_Z(V)))\, \mathcal{J} dZ  d\sigma d\xi dV  d\mu d\lambda.
$$
Here $\mathcal{T}_{0,1,2} \,\, = \,\,  \mathcal{T}^{-1}_{x,x_1}\mathcal{T}^{-1}_{x_1,x_2}\mathcal{T}_{x,x_2}$, and $\phi_Z:T(TF)_{(x,X)} \to T(TF)_{(x,X)}$  is
$$
\phi_Z(V) \,\, = \,\, \exp^{-1}_{(x,X)}(\exp_{(x_1,X_1)}(\mathcal{T}_{x,x_1} (V))) \,\, = \,\,  (W_1,W_{\nu},W_2),
$$
where
$$
(W_1+W_2,W_{\nu}) = \exp^{-1}_{x}\exp_{x_1}\mathcal{T}_{x,x_1}  (V_1+V_2,V_{\nu}).
$$
This is because $\mathcal{T}_{x,x_1}  (V_1,V_{\nu},V_2) = (Y_1, Y_{\nu},Y_2)$, so $(Y_1 + Y_2, Y_{\nu}) =\mathcal{T}_{x,x_1}  (V_1+V_2,V_{\nu})$, and
$$
W_2 = -X +  \mathcal{T}^{-1}_{x,x_2} ( \mathcal{T}_{x_1,x_2}(X_1 + \mathcal{T}^{(2)}_{x,x_1} (0,V_2))),
$$
where $\mathcal{T}^{(2)}_{x,x_1}$ is the composition $ \nu \oplus TF \stackrel{\mathcal{T}_{x,x_1}}{\to}  \nu \oplus TF \stackrel{\pi}{\to} TF$, and $\pi$ is the projection.

Note that $\phi_Z(V)$ depends on $X$ in general, that $\phi_0  = \I$, and we are only interested in $\phi_Z(V)$ for 
$$
\wtit{\alpha}((x,X), \exp_{(x,X)}(Z) ) \neq 0 \quad \text{and} \quad 
\wtit{\alpha}((x_1,X_1),\exp_{(x_1,X_1)}(\mathcal{T}_{x,x_1}(V) ) ) \neq 0,
$$
that is for $Z$ and $V$ small.   We can control the size of the relevant $Z$ and $V$ by making the support of $\alpha$ close to the diagonal.   Thus, we may assume that $\phi_Z$ is a diffeomorphism on  a neighborhood of $V = 0$.   According to \cite{Gilkey}, p.\ 25 bottom,   $\phi_Z^{-1}(V) = \phi_Z^{-1}(0) + A_{Z,V}V$, where $A_{Z,V}$ is a linear map which is invertible in a neighborhood of $V=0$.  For $Z, \what{Z} \in T(TF)_{(x,X)}$, with $Z$ close to zero, define the linear map $W_Z$ on $T(TF)_{(x,X)}$ to be 
$$
W_Z(\what{Z}) \,\, = \,\, (\what{W}, 0, -\what{W}),
$$
where 
$$
\what{W} \,\, = \,\, 
\what{Z_2} - \Big(\mathcal{T}^{(2)}_{x, x_1} \, | \, _{TF_{(x,X)}}\Big)^{-1}\mathcal{T}_{x, x_1}(\what{Z_2}).$$ 
Note that in general this is not zero, since $\mathcal{T}_{x, x_1}$ is parallel translation in the bundle $TF$, while $\mathcal{T}^{(2)}_{x, x_1} \, | \, _{TF_{(x,X)}}$ is parallel translation in $TM$ followed by projection to $TF$.  It will be zero if $F$ is totally geodesic, but this is quite rare.  It is a straight forward computation to show that $ \phi_Z^{-1}(0) =  -Z - W_Z Z$, (solve $\phi_Z(V) = 0$ directly from the definition of $\phi_Z$, and use the fact that $x_2 = \exp_{x_1}(\mathcal{T}_{x,x_1}(V_1 + V_2, V_{\nu})) = x$),  so 
$$
 \phi_Z^{-1}(V) \,\, = \,\, -Z \,\, - \,\, W_Z Z \,\, + \,\, A_{Z,V}V
 \,\, = \,\, -(\I \,\, + W_Z)Z \,\, + \,\, A_{Z,V }V.
$$
We note for later use, that if $Z=0$, then $x_1 = x$, and so $W_0 =0$.

Make the change of variables $V \to \phi_Z^{-1}(V)$  to get
$$
\int  e^{-i\langle  \phi_Z^{-1}(V),(\lambda,\mu)\rangle}  
\int   e^{-i\langle  Z,(\xi,\sigma)\rangle} p(x,\xi,\sigma,t)_t  \, \wtit{\alpha}((x,X), \exp_{(x,X)}(Z))
\mathcal{T}^{-1}_{x,x_1} \Big(q(x_1,\mathcal{T}_{x,x_1}(\lambda,\mu),t)_t\Big)   
$$
$$
\wtit{\alpha}((x_1,X_1),\exp_{(x_1,X_1)}(\mathcal{T}_{x,x_1}( \phi_Z^{-1}(V)) ) )\mathcal{T}_{0,1,2}     \mathcal{T}^{-1}_{x,x_2}  (u(\exp_{(x,X)}(V)))\, \mathcal{J} dZ  d\sigma d\xi dV  d\mu d\lambda.
$$
{{Set
$$
\what{ \alpha}(x,X,Z,V) = \wtit{\alpha}((x,X), \exp_{(x,X)}(Z))
\wtit{\alpha}((x_1,X_1),\exp_{(x_1,X_1)}(\mathcal{T}_{x,x_1}( \phi_Z^{-1}(V)) ) ).
$$
Choose a bump function $\beta$ on $TF \times TF$ which is  supported in a neighborhood of the diagonal so that $ \beta((x,X), \exp_{(x,X)}(V)) = 1$ whenever $\what{ \alpha} \neq 0$.  We are assured that such a $\beta$ exists by choosing the support of $\alpha$ sufficiently close to the diagonal.}}
Replacing $\phi_Z^{-1}(V)$ by $-(\I + W_Z)Z+ A_{Z,V}V$ and making the change of variables $(\lambda,\mu) \to [A_{Z,V}^{-1}]^*(\lambda,\mu)$ gives
$$
\int  e^{-i\langle  V,(\lambda,\mu)\rangle}  r(x,X,V,\lambda,\mu,t)_t \, \beta((x,X), \exp_{(x,X)}(V))
\mathcal{T}^{-1}_{x,x_2}  (u(\exp_{(x,X)}(V))   dV  d\mu d\lambda,
$$
where 
\begin{multline*}
r(x,X, V,\lambda,\mu,t)_t  \,\, = \,\, 
\int   e^{i\langle  (\I + W_Z)Z, [A_{V,Z}^{-1}]^*  (\lambda,\mu)\rangle} 
e^{-i\langle  Z,(\xi,\sigma) \rangle}
p(x,\xi,\sigma,t)_t 
\\
\what{\alpha}(x,X,Z,V)
\mathcal{T}^{-1}_{x,x_1}\Big(q(x_1,\mathcal{T}_{x,x_1}[A_{Z,V}^{-1}]^*(\lambda,\mu),t)_t \Big)   \mathcal{T}_{0,1,2} \, \maJ  dZ  d\sigma d\xi  \,\, = \,\, 
\end{multline*}
\begin{multline*}
\int   e^{-i\langle Z,(\xi,\sigma) -(\I + W_Z)^* [A_{V,Z}^{-1}]^*  (\lambda,\mu)\rangle} 
 p(x,\xi,\sigma,t)_t 
\\
\what{\alpha}(x,X,Z,V)
\mathcal{T}^{-1}_{x,x_1}\Big(q(x_1,\mathcal{T}_{x,x_1}[A_{Z,V}^{-1}]^*(\lambda,\mu),t)_t \Big)   \mathcal{T}_{0,1,2} \, \maJ  dZ d\sigma d\xi.
\end{multline*}

\medskip
Next make the change of variables $(\xi,\sigma) \to (\xi,\sigma) + (\I + W_Z)^* [A_{V,Z}^{-1}]^*  (\lambda,\mu)$ to get 
\begin{multline*}
r(x,X, V,\lambda,\mu,t)_t  \,\, = \,\, 
\int   e^{-i\langle Z,(\xi,\sigma)\rangle} 
 p(x,(\xi,\sigma) + (\I + W_Z)^* [A_{V,Z}^{-1}]^*  (\lambda,\mu),t)_t 
\\
\what{\alpha}(x,X,Z,V)
\mathcal{T}^{-1}_{x,x_1}\Big(q(x_1,\mathcal{T}_{x,x_1}[A_{Z,V}^{-1}]^*(\lambda,\mu),t)_t \Big)   \mathcal{T}_{0,1,2} \, \maJ  dZ  d\sigma d\xi.
\end{multline*}

We now follow \cite{Treves}, the proof of Theorem 3.2,  to  show that $r$ satisfies the hypotheses of Lemma \ref{BFresults2}, (with $Z$ replaced by $V$ in the lemma) so it determines a symbol $\what{r}(t)$ with $P_t \circ Q_t = \what{\theta}^{\beta}(\what{r}(t)_t)$.   We give a brief outline, and leave the details to the reader, of how to obtain an estimate of the form 
$$
||  \,   \pa^{\alpha}_{\lambda} \pa^{\beta}_{\mu} \pa^{\delta}_{(x,X,V)}  r(x,X,V,\lambda,\mu,t) \, || \,\, \leq \,\, 
C_{\alpha,\beta,\delta}(1 + |\lambda|)^{m- |\alpha|}(1 + |\mu|)^{\ell-|\beta|}{,}
$$
as in  Lemma \ref{BFresults2}.  First note that the derivatives with respect to $x$, $X$,  and $V$ of $\what{\alpha}(x,X,Z,V)$,   and  $\maJ$ are uniformly bounded, so we may dispense with them.  We may assume that $0 \leq t \leq 1$, since we are only interested in $t$ in that interval.  Next note the crucial facts that $\phi_0 = \I$ and $W_0 = 0$, which imply that   $[A_{0,V}^{-1}]^* = \I$ and $\I + W_0 = \I$.   In addition,  $ \mathcal{T}_{x,x} = \I$.  So, by making the support of $\alpha$ close to the diagonal, we can make $[A_{Z,V}^{-1}]^*$ as close to $\I$ as we please and all its derivatives uniformly bounded.  Similar remarks apply to $\I + W_Z$, $ \mathcal{T}_{x,x_1}$,  $\mathcal{T}^{-1}_{x,x_1}$, and  $\mathcal{T}_{0,1,2}$.   So,  we may assume that $[A_{Z,V}^{-1}]^*, I + W_Z, \mathcal{T}_{x,x_1}, \mathcal{T}^{-1}_{x,x_1},$ and  $\mathcal{T}_{0,1,2}$ are the identity when computing estimates.  Thus we may assume that $r(t)_t$ has the form  
$$
r(x,X,V,\lambda,\mu,t)_t  \,\, = \,\, 
\int   e^{-i\langle  Z,(\xi,\sigma) \rangle}   \what{\alpha}(Z)  p(x,\xi +\lambda,\sigma + \mu,t)_t  \,
q(x_1,  \lambda,\mu,t)_t   \,  dZ d\sigma d\xi   \,\, = \,\, 
$$
$$
\int   e^{-i\langle  Z,(\xi,\sigma) \rangle}  \what{\alpha}(Z)   p(x, t\xi +t\lambda, t\sigma + t\mu,t)
q(x_1, t \lambda, t\mu,t)     dZ d\sigma d\xi.
$$
Here we have written $\what{\alpha}(Z)$ for $\what{\alpha}(x,X,Z,V)$.

{{The astute reader will note that this equation is incorrect, unless both $p$ and $q$ have no form components. However, since  we are only interested in $t$ for $t \in [0,1]$, we may replace the missing terms by $1$, and no harm is done to our estimates. More precisely, if either has form components, then the expression after the equals sign is missing terms of the form $t$ to a positive power.  We need to find estimates on  $r = [r_t]_{t^{-1}}$ and its derivatives, which (because we are using Clifford multiplication and not differential form multiplication) may also be missing terms of the form $t$ to a positive power. This is due to forms which would disappear under form multiplication, but do not under Clifford multiplication.  For example, $[dx_t \cdot dx_t]_{t^{-1}} = [-t^2||dx||^2]_{t^{-1}}  = -t^2 ||dx||^2$, while $[dx_t \wedge dx_t]_{t^{-1}} = [0]_{t^{-1}}  = 0$.  So, for $t\in [0, 1]$ we can easily estimate the extra terms.  }}

Now $r = [r_t]_{t^{-1}}$, so we have finally that 
$$
r(x,X,V,\lambda,\mu,t)  \,\, = \,\, 
\int   e^{-i\langle  Z,(\xi,\sigma) \rangle} \what{\alpha}(Z)  p(x, t\xi +\lambda, t\sigma + \mu,t)
q(x_1,  \lambda,\mu,t)     dZ d\sigma d\xi.
$$
  
The derivatives with respect to $x$ of $p$ are uniformly bounded, so we may ignore them in the computation.  So now consider $\pa^{\alpha}_{\lambda} \pa^{\beta}_{\mu} r(x,X, V,\lambda,\mu,t)$, which is a finite sum of terms which are constants times terms of the form
$$
\int   e^{-i\langle  Z,(\xi,\sigma) \rangle}    \what{\alpha}(Z)\pa^{\alpha_1}_{\lambda} \pa^{\beta_1}_{\mu}p(x,  t\xi +\lambda, t\sigma + \mu,t)
\pa^{\alpha_2}_{\lambda} \pa^{\beta_2}_{\mu}q(x_1,  \lambda,\mu,t)     dZ d\sigma d\xi
$$
where the $\alpha_i$ add up to $\alpha$, and  the $\beta_i$ add up to $\beta$.  In what follows, we will again ignore constants.  Note that, up to a constant, $e^{-i\langle  Z,(\xi,\sigma) \rangle}= (1 + |(\xi,\sigma)|^2)^{-N} (1 + \Delta_Z)^N   (e^{-i\langle  Z,(\xi,\sigma) \rangle})$, where $N$ is an integer to be specified soon.  Thus the above equals 
$$
\int  (1 + |(\xi,\sigma)|^2)^{-N} \Big[(1 + \Delta_Z)^N   e^{-i\langle  Z,(\xi,\sigma) \rangle}  \Big]  \what{\alpha}(Z)
\pa^{\alpha_1}_{\lambda} \pa^{\beta_1}_{\mu}p(x, t\xi+ \lambda, t\sigma+ \mu,t) 
\pa^{\alpha_2}_{\lambda}\pa^{\beta_2}_{\mu}q(x_1,\lambda,\mu,t)  dZ d\sigma d\xi,
$$
and integration by parts gives something of the form
$$
\int  (1 + |(\xi,\sigma)|^2)^{-N} e^{-i\langle  Z,(\xi,\sigma) \rangle}   \Big[(1 + \Delta_Z)^N   \what{\alpha}(Z) \Big]  
\pa^{\alpha_1}_{\lambda} \pa^{\beta_1}_{\mu}p(x, t\xi+ \lambda, t\sigma+ \mu,t) 
\pa^{\alpha_2}_{\lambda}\pa^{\beta_2}_{\mu}q(x_1,\lambda,\mu,t)  dZ d\sigma d\xi.
$$
 Note that we are integrating $Z$ over {{compact sets whose diameters and volumes are uniformly bounded,}} namely where $\what{\alpha}(Z) = \what{\alpha}(x,X,Z,V) \neq 0$, and we are integrating a uniformly bounded function $e^{-i\langle  Z,(\xi,\sigma) \rangle}  (1 + \Delta_Z)^N   \what{\alpha}(Z)$.  Suppose that  $p \in SC^{m,\ell}(M,E)$ and $q \in SC^{m',\ell'}(M,E)$.   Then the integral is bounded by a multiple of 
$$
\int (1 + |(\xi,\sigma)|^2)^{-N}  (1 + |t\xi + \lambda|)^{m - |\alpha_1|}(1 +|t\sigma+ \mu|)^{\ell - |\beta_1|}  (1 + |\lambda|)^{m'- |\alpha_2|}(1 + |\mu|)^{\ell'-|\beta_2|}   d\sigma d\xi.   
$$
Now Peetre's inequality gives that
$$
(1 + |t\xi + \lambda|)^{m - |\alpha_1|} \,\, \leq \,\,  C(1 + |t\xi|)^{|m - |\alpha_1||}(1 + |\lambda|)^{m - |\alpha_1|},
$$ 
and similarly  $(1 +|t\sigma+ \mu|)^{\ell - |\beta_1|} \leq  C (1 +|t\sigma|)^{|\ell - |\beta_1||}(1 +| \mu|)^{\ell - |\beta_1|}$.  Thus the integral is bounded by
$$
(1 + |\lambda|)^{m+m'- |\alpha|}(1 + |\mu|)^{\ell+\ell'-|\beta|} 
\int  (1 + |(\xi,\sigma)|^2)^{-N}(1 + |t\xi|)^{|m - |\alpha_1||}(1 +|t\sigma|)^{|\ell - |\beta_1||}    d\sigma d\xi.   
$$
If we choose $N$ large enough, then the integral converges, so we have that $||\pa^{\alpha}_{\lambda} \pa^{\beta}_{\mu} \pa^{\delta}_{(x,X,V)}   r(x,X,V,\lambda,\sigma,t) ||$ is bounded by a multiple of 
$$
(1 + |\lambda)|)^{m+m'- |\alpha|}(1 + |\mu|)^{\ell+\ell'-|\beta|} 
$$
as required.

To determine the  asymptotic expansion, we proceed as follows.   By Lemma \ref{BFresults3},  $\what{r}(t)\sim \varsigma(\what{\theta}^{\alpha}(p_t) \circ \what{\theta}^{\alpha}(q_t))_{t^{-1}}$, so we may work with $\varsigma(\what{\theta}^{\alpha}(p_t) \circ \what{\theta}^{\alpha}(q_t))$.  In particular we wish to invoke the results of Widom, \cite{Wid2, Wid1}.  Let $(Z_1,Z_{\nu},Z_2) \in T(TF)_{(x,X)}$, and consider the local diffeomorphism
$$
\wtit{\exp}_{(x,X)}(Z_1,Z_{\nu},Z_2)  \,\,=\,\,  
( \exp_{x}(Z_1,Z_{\nu}),\mathcal{T}_{x,\exp_{x}(Z_1,Z_{\nu})}(X+Z_2) ).
$$
Let $P$ be an operator on sections of $\what{E}_{\cS}$, $(x, X) \in TF$, $(\eta, \zeta, \sigma)  \in   T^*(TF)_{(x, X)}$, and $u_{(x, X)} \in (\what{E}_{\cS})_{(x, X)}$.  Set 
$$
\wtit{\varsigma}(P)(x, X,\eta, \zeta, \sigma)(u_{(x, X)}) \,\, = \,\, 
$$
$$ 
P \Big((x',X') \mapsto e^{i\langle \wtit{ \exp}^{-1}_{(x,X)}(x',X'),(\eta, \zeta, \sigma)\rangle}
\wtit{\alpha}((x,X),(x',X'))\mathcal{T}_{x,x'}(u_{(x, X)})\Big) \, |_{(x', X')=(x, X)}.
$$
It is immediate that 
$$
\varsigma(P)(x, X,\eta, \zeta, \sigma) \,\, = \,\, \wtit{\varsigma}(P)(x, X,\eta, \zeta, \sigma-\zeta).
$$

\begin{lemma}
Let $\pi:A \to M$ be a vector bundle with connection $\nabla^A$ over a manifold $M$ with connection $\nabla^M$.  Then for $Y \in A$ with $\pi(Y)=x$, and $(X,Z) \in TA_Y \simeq TM_x \oplus A_x$, the  path $\gamma(t) = \mathcal{T}_{x, \exp_x(tX)}(Y+tZ)$ is a geodesic in $A$ for the connection $\nabla = \pi^*\nabla^M \oplus \pi^*\nabla^A$.
\end{lemma}

\begin{proof}
Set $\sigma(t) = \exp_x (tX)$, and denote the derivatives of $\gamma$ and $\sigma$ by $\stackrel{.}{\gamma}$ and $\stackrel{.}{\sigma}$.  Then 
$$
\stackrel{.}{\gamma} \hspace{-0.1cm}(t)  \,\, = \,\, (\stackrel{.}{\sigma}\hspace{-0.1cm}(t), \mathcal{T}_{x, \sigma(t)}(Z))
$$
is the pull back of a section of $TM \oplus A$, also denoted $(\stackrel{.}{\sigma}\hspace{-0.1cm}(t), \mathcal{T}_{x, \sigma(t)}(Z))$. In addition,  $\pi_*(\stackrel{.}{\gamma}\hspace{-0.1cm}(t)) = \,\, \stackrel{.}{\sigma}\hspace{-0.1cm}(t)$, since $(0,\mathcal{T}_{x, \sigma(t)}(Z))$ is tangent to the fibers of $A \to M$.
Thus
$$
\nabla_{\stackrel{.}{\gamma}(t)} \stackrel{.}{\gamma} \hspace{-0.1cm}(t) \,\, = \,\, 
(\pi^*\nabla^M \oplus \pi^*\nabla^A)_{\stackrel{.}{\gamma}(t)} \pi^*(\stackrel{.}{\sigma}\hspace{-0.1cm}(t), \mathcal{T}_{x, \sigma(t)}(Z))  \,\, = \,\, 
$$
$$
(\nabla^M \oplus \nabla^A)_{\pi_*(\stackrel{.}{\gamma}(t))}(\stackrel{.}{\sigma}\hspace{-0.1cm}(t), \mathcal{T}_{x, \sigma(t)}(Z))  \,\, = \,\, 
(\nabla^M_{\stackrel{.}{\sigma}(t)}\stackrel{.}{\sigma}\hspace{-0.1cm}(t),  \nabla^A_{\stackrel{.}{\sigma}(t)} \mathcal{T}_{x, \sigma(t)}(Z))  \,\, = \,\, 
(0,0),
$$
as required.
\end{proof}

Thus $\wtit{\exp}$ is the usual exponential for the manifold $TF$ for the connection which is the pull back of the  Levi-Civita connection on $TM$ to $ \pi^*(TM)$ direct sum with the pull back to $\pi^*(TF)$ of the  connection on $TF$ induced from the Levi-Civita connection on $TM$.  
It follows that the function  $\langle \wtit{\exp}^{-1}_{(x,X)}(x',X'),(\eta, \zeta, \sigma)\rangle$ satisfies  Proposition 2.1 of \cite{Wid2}, see \cite{Remp}, locally.     In addition, $\wtit{\varsigma}(P)$ is the usual symbol associated to this connection for the operator $P$.  Thus we may combine Widom's results with the argument in \cite{BF}, pp.\ 22-23, to get that there are differential operators $\wtit{a}_n$ (which decrease the grading by $n$, and which do not differentiate in the $t$ variable) so that for $p(t)$ and $q(t)$ asymptotic symbols in $SC^{\infty,\infty}(TF, \what{E}_{\cS})$,
$$
\wtit{\varsigma}(\what{\theta}^{\alpha}(p_t) \circ \what{\theta}^{\alpha}(q_t)) 
\,\,  \sim \,\, 
\sum_{n=0}^{\infty} t^n \wtit{a}_n(p,q)_{t}.
$$
Thus we have
\begin{multline*}
\varsigma(\what{\theta}^{\alpha}(p_t) \circ \what{\theta}^{\alpha}(q_t))(x, X,\eta, \zeta, \sigma)    \,\, = \,\, 
\wtit{\varsigma}(\what{\theta}^{\alpha}(p_t) \circ \what{\theta}^{\alpha}(q_t)) (x, X,\eta, \zeta, \sigma-\zeta)
\,\,  \sim \,\, 
\\
\sum_{n=0}^{\infty} t^n \wtit{a}_n(p,q)_{t}(x, X,\eta, \zeta, \sigma-\zeta)  \,\, = \,\,
\sum_{n=0}^{\infty} t^n  \what{a}_n(p,q)_t(x, X,\eta, \zeta, \sigma).
\end{multline*}
where the $\what{a}_n$ are also differential operators which decrease the grading by $n$, and do not differentiate in the $t$ variable.  For asymptotic symbols $p(t)$ and $q(t)$ in $SC^{\infty,\infty}(M, E)$ we have 
$$
p \circ q \,\, := \,\,  \varsigma(\theta^{\alpha}(p_t)\circ \theta^{\alpha}(q_t))_{t^{-1}} 
\,\, := \,\, \varsigma(\what{\theta}^{\alpha}(\what{p}_t) \circ \what{\theta}^{\alpha}(\what{q}_t))_{t^{-1}}  \,\, \sim \,\, 
\sum_{n=0}^{\infty} t^n \what{a}_n(\what{p},\what{q})   \,\, = \,\, 
\sum_{n=0}^{\infty} t^n \what{a}_n(p,q).
$$
The $\what{a}_n$ acting on elements in $SC^{\infty,\infty}(M, E)$ are determined by how they act on symbols which are polynomial in $\zeta$, $\eta$, and $\sigma$, so they must be the $a_n$ of Theorem \ref{comp}.  Finally we have 
$$
p \circ q \,\, \sim \,\,\sum_{n=0}^{\infty} t^n a_n(p,q) \,\, \sim \,\,  
\sum_{n,k,k'=0}^{\infty} t^n a_n(t^k p_k, t^{k'} q_{k'}) \,\, = \,\,  
\sum_{n,k,k'=0}^{\infty} t^{n+k+k'} a_n(p_k, q_{k'}),
$$
giving the asymptotic expansion for $p \circ q$, and identifying its leading symbol as $a_0(p_0,q_0)$. 
\end{proof}
In the case of a Riemannian foliation, formula $(2)$ in Theorem \ref{mainlemma} simplifies, just as the formula in Theorem \ref{comp} does, and we get:

\begin{corollary}
Suppose $F$ is a Riemannian foliation, and write $p$ and $q$ as functions of $x, \eta, \zeta, \sigma$, and $t$.  Then, under the assumptions of Theorem \ref{mainlemma},  
$$
a_0(p_0,q_0)(x, \eta, \zeta, \sigma,t)  \,\, = \,\,  e^{-\frac{1}{4}\Omega_{\nu}(\pa/\pa \eta,\pa/\pa \eta') }p_0(x,\eta, \zeta,\sigma,t) \wedge q_0(x, \eta', \zeta,\sigma,t) \, |_{\eta'=\eta}.
$$
\end{corollary}

\appendix

\section{proof of Lemma \ref{simOps2}}

We now give the proof of \\ 
{\bf Lemma \ref{simOps2}}
{\it {
If a family of symbols $p(t) \in SC^{-\infty,-\infty}(TF,\what{E}_{\cS})$  and  $p(t) \sim 0$, then $P_t =\what{\theta}^{\alpha}(p(t)_t)\sim 0$.  The same result holds for $p(t) \in SC^{-\infty,-\infty}(M,E)$.
}}

Recall that  $p(t) \sim 0$ means that for all $N$, $\lim_{t \to 0} t^{-N}p(t) \,\, = \,\, 0 $ in $SC^{-\infty,-\infty} (TF, \what{E}_{\cS})$, and  note  that $p(t)$ is not assumed to be an asymptotic symbol.

\begin{proof}

First assume that the $(x,X)$  support of the symbol is a compact subset of a cube $I^{n+p} \subset TF$. 
The operator norm $||P_t||_{0,0}$ is bounded by the Hilbert-Schmidt norm of $P_t$, which in turn coincides with the $L^2$ norm over $I^{n+p}\times \R^{n+p} \subset T^*(TF)$ of the Schwartz kernel $K_t$ of $P_t$.   Now, $K_t$ is smooth and is given, when the bundle is trivialized, by
$$
K_t ((x,X),(x',X')) \,\, = \,\,  \wtit{\alpha}((x, X),(x',X'))\int_{\R^{n+p}} e^{i\langle(x-x',X-X'),(\xi,\sigma) \rangle} p(x,X,\xi, \sigma,t)_t  d\xi d\sigma.
$$
The norm of $\wtit{\alpha}((x, X),(x',X'))$ is uniformly bounded, so we will ignore it.
Using Plancherel, that is the fact that the Fourier transform is an isometry on $L^2$, we get 
$$
\int_{\R^{n+p}} || K_t((x,X),(x',X'))||^2 dx' dX' \,\, = \,\, \int_{\R^{n+p}}||p(x,X,\xi, \sigma,t)_t||^2 \, d\xi d\sigma .
$$
Thus modulo constants, we get 
$$
 ||P_t||^2_{0,0} \,\, \leq \,\, \int_{I^{n+p} \times \R^{n+p}} ||p(x,X,\xi, \sigma,t)_t ||^2 \, d\xi d\sigma  dx dX.
$$

Let $N > 0$ be given, and choose $\what{N} >N+ (n+p)/2$.  Recall that  we are assuming that $p(t) \in  SC^{-\infty,-\infty}(TF,\what{E}_{\cS})$ and $p(t) \sim 0$.  So in particular,  $p(t) \in  SC^{-n,-p}(TF,\what{E}_{\cS})$ and 
$$
\lim_{t \to 0} t^{-\what{N}} p(t) \,\, = \,\, 0 
$$
in $SC^{-n,-p}(TF,\what{E}_{\cS})$.   Thus for each $t$, there is a constant $C_t$ so that
$$
 t^{-\what{N}} || p(x,X,\xi,\sigma,t) || \,\, \leq \,\, C_t(1 + | \xi |)^{-n}(1+| \sigma |)^{-p},
$$
and $C_t \to 0$ as $t \to 0$.  
So, for $t$ small enough,   $||p(x,X,\xi, \sigma,t)_t || \leq   t^{\what{N}}(1+|t\xi|)^{-n}(1+|t\sigma|)^{-p}$.  Thus, modulo constants,
$$
 t^{-N}||P_t||_{0,0}   \,\, \leq \,\,   t^{-N}\Big[ \int_{I^{n+p} \times \R^{n+p}} ||p(x,X,\xi, \sigma,t)_t ||^2 \, d\xi d\sigma dx dX\Big]^{1/2} \,\, \leq \,\,
 $$
 $$ 
t^{\what{N}-N}  \Big[\int_{\R^{n+p}} (1+|t\xi|)^{-2n} (1+|t\sigma|)^{-2p}d\xi d\sigma  \Big]^{1/2}    \,\, = \,\, 
t^{\what{N}- (N+(n+p)/2)}  \Big[\int_{\R^{n+p}} (1+|\xi|)^{-2n} (1+|\sigma|)^{-2p}d\xi d\sigma  \Big]^{1/2}, 
$$
by making the change of coordinates $(\xi,\sigma) \to (\xi/t,\sigma/t)$.  This last goes to zero as $t \to 0$.

To extend to the case where the $(x,X)$ support is not necessarily compact, we note that the estimates on $p(t)_t$ are uniform on $TF$, and since the geometry of $TF$ is bounded, we may assume that we have a countable locally finite cover of $TF$ by cubes $I^{n+p}$ whose diameters and volumes are uniformly bounded.  Given any $L^2$ section $u$ of $\what{E}_{\cS}$, we may write it as a countable sum of $L^2$ sections, whose supports are pairwise disjoint, each being contained in a different cube.  The result for $||P_t||_{0,0}$ then follows from standard techniques.  {{Indeed, using again the bounded geometry assumption, there is a uniform upper bound on the local norms, and $||P_t||_{0,0}$ can then be estimated by the supremum of these local norms.}}

Next consider  $||P_t||_{s,k}= ||(1 + \nabla^*\nabla)^{k/2}P_t  (1 + \nabla^*\nabla)^{-s/2} ||_{0,0}$.  The operators 
$\what{\theta}^{\alpha}(1 + |\xi|^2 + |\sigma|^2)$ and $1 + \nabla^*\nabla$ are both second order (uniformly) elliptic differential operators on $\what{E}_{\mathcal{S}}$, so we may also {{use the equivalent norm}}
$$
||P_t||_{s,k}= ||\what{\theta}^{\alpha}((1 + |\xi|^2 + |\sigma|^2)^{k/2})P_t  \what{\theta}^{\alpha}((1 + |\xi|^2 + |\sigma|^2)^{-s/2}) ||_{0,0}.
$$
Now $(P_t  \what{\theta}^{\alpha}((1 + |\xi|^2 + |\sigma|^2)^{-s/2})  =  
\what{\theta}^{\alpha}(p_t)  \what{\theta}^{\alpha}((1 + |\xi|^2 + |\sigma|^2)^{-s/2})$, and if we set 
$q^{-s}(x,X,\xi,\sigma,t)  = (1 + |\xi/t|^2 + |\sigma/t|^2)^{-s/2}$, we have
$$
P_t  \what{\theta}^{\alpha}((1 + |\xi|^2 + |\sigma|^2)^{-s/2})  \,\, = \,\,   \what{\theta}^{\alpha}(p_t) \what{\theta}^{\alpha}(q^{-s}_t).
$$
As in the proof of Theorem  \ref{mainlemma}, we may assume that $[A_{Z,V}^{-1}]^*, I + W_Z, \mathcal{T}_{x,x_1}, \mathcal{T}^{-1}_{x,x_1},$ and  $\mathcal{T}_{0,1,2}$ are the identity when computing estimates.  Note that  $\what{\alpha}(V)$ actually depends on $Z$, as do other terms we are ignoring, so we write them as $\what{\alpha}(V,Z)$. 
Set
$$
r(x,X,Z,\xi,\sigma,t)_t  \,\, = \,\, 
\int_{T^*(TF)_{(x,X)} \times T(TF)_{(x,X)}}   \hspace{-3.0cm}    e^{-i\langle  V,(\lambda,\mu) \rangle}  \what{\alpha}(V,Z)   p(x,X, t\xi +t\lambda, t\sigma + t\mu,t)
(1 + |\xi|^2 + |\sigma|^2)^{-s/2}     dV d\lambda d\mu,
$$
and
$$
\what{r}(x,X,\xi,\sigma,t)_t   \,\, = \,\,  \int_{T^*(TF)_{(x,X)} \times T(TF)_{(x,X)}}   \hspace{-3.0cm} 
 e^{-i\langle  Z,(\what{\xi},\what{\sigma})\rangle} r(x,X, Z, t \what{\xi}+t\xi, t \what{\sigma}+t\sigma,t)\, 
 dZ d\what{\sigma} d\what{\xi} \,\, = \,\,
 $$
 $$
 \int \int  e^{-i\langle  Z,(\what{\xi},\what{\sigma})\rangle}
e^{-i\langle  V,(\lambda,\mu) \rangle}  \what{\alpha}(V,Z)   p(x,X, t\what{\xi} +t\xi +t\lambda, t\what{\sigma} +t\sigma + t\mu,t)
(1 + |\what{\xi}+ \xi|^2 + |\what{\sigma}+ \sigma|^2)^{-s/2}     dV d\lambda d\mu
 dZ d\what{\sigma} d\what{\xi}.
 $$
By Lemma \ref{BFresults2} and the proof of Theorem \ref{mainlemma}, and modulo constants,  
$$
\what{\theta}^{\alpha}(p_t) \what{\theta}^{\alpha}(q^{s}_t) \,\, = \,\,
\what{\theta}^{\alpha}(\what{r}(t)_t).
$$

{{Recall that, thanks to  the bounded geometry assumption,  the support of $\what{\alpha}$ is contained in a uniform ball bundle over the total manifold $TF$ and the Fourier transform $FT ((V, Z)\mapsto \what{\alpha}((x, X); (V,Z))$  is of Schwartz class  uniformly in the $(x, X)$ variables, i.e. the Schwartz semi-norms  are uniformly bounded in the $(x,X)$ variables.}}   Now 
$$
\what{r}(x,X,\xi,\sigma,t)_t  \,\, = \,\,  FT^{-1}\Big[FT(\what{\alpha}(V,Z))(\what{\xi},\what{\sigma},\lambda,\mu) p(x,X, t\what{\xi} +t\xi +t\lambda, t\what{\sigma} +t\sigma + t\mu,t)
$$
$$(1 + |\what{\xi}+ \xi|^2 + |\what{\sigma}+ \sigma|^2)^{-s/2} \Big](0,0),
$$
so, $\what{r}(t)_t$ has the same properties as $p(t)_t$.
Namely, $\what{r}(t)_t$ is of Schwartz class uniformly on each fiber in the sense described above, and for all $\what{N} > 0$,  there is $C_t \in \R$, so that $\lim_{t \to 0} C_t = 0$, and
$$
t^{-\what{N}} ||\what{r}(x,X,\xi,\sigma,t)_t|| \,\,  \leq \,\,  C_t (1+ |t\xi|)^{-n-|s|} (1+ |t\sigma|)^{-p-|s|} (1 + |\xi| + | \sigma |)^{-s}.
$$
To see this, note that for any $\what{N}$ and any $(m, \ell)\in \Z^2$, there exists $C_t=C_t(m,\ell, \what{N})$ which goes to zero as $t \to 0$, so that 
$$
t^{-\what{N}} ||p (x,X,\xi+\what{\xi}+\lambda, \sigma+\what{\sigma}+\mu,\sigma,t)_t||  \,\,  \leq \,\, C_t (1+ |t\xi|+|t\what{\xi}| + |t\lambda|)^{m} (1+ |t\sigma|+|t\what{\sigma}|+ |t\mu|)^{\ell}.
$$
Applying Peetre's inequality gives
$$
(1+ |t\xi|+|t\what{\xi}| + |t\lambda|)^{m} (1+ |t\sigma|+|t\what{\sigma}|+ |t\mu|)^{\ell}\,\,  \leq \,\, 
(1+ |t\what{\xi}| + |t\lambda|)^{|m|} (1+|t\what{\sigma}|+ |t\mu|)^{|\ell |} (1+ |t\xi|)^m (1+ |t\sigma|)^\ell,
$$
and 
$$
(1 + |\what{\xi}|+ |\xi| + |\what{\sigma}|+ |\sigma|)^{-s} \,\,  \leq \,\, 
(1+|\xi| +  |\sigma|)^{-s} (1 + |\what{\xi}|+  |\what{\sigma}|)^{|s|}.
$$
Set $\varphi:=FT(\what{\alpha})$, which is a rapidly decaying function.   Then using the estimate of $1+r^2$ by $(1+r)^2$ for $r \geq 0$, we have
\begin{multline*}
t^{-\what{N}} ||\what{r}(x,X,\xi,\sigma,t)_t||\,\,  \leq \,\, C_t (1+ |t\xi|)^m (1+ |t\sigma|)^\ell (1+|\xi| +  |\sigma|)^{-s} 
 \\
\int \varphi(\what{\xi},\what{\sigma}; \lambda, \mu) (1+ |t\what{\xi}| + |t\lambda|)^{|m|} (1+ |t\what{\sigma}|+ |t\mu|)^{|\ell|} (1 + |\what{\xi}| + |\what{\sigma}|)^{|s|} d\what{\xi}d\what{\sigma}d\lambda d\mu.
\end{multline*}
For $|t|\leq 1$, we deduce  that 
\begin{multline*}
t^{-\what{N}} ||\what{r}(x,X,\xi,\sigma,t)_t|| \,\,  \leq \,\, C_t (1+ |t\xi|)^m (1+ |t\sigma|)^\ell (1+|\xi| +  |\sigma|)^{-s} 
\\
\int \varphi(\what{\xi},\what{\sigma}; \lambda, \mu) (1+ |\what{\xi}| + |\lambda|+|\what{\sigma}|+ |\mu|)^{|m|+|\ell | +|s|}  d\what{\xi}d\what{\sigma}d\lambda d\mu.
\end{multline*}
Since $\varphi$ is rapidly decaying, the integral is a finite constant, and replacing $(m,\ell)$ by $(-n-|s|,-p-|s|)$ gives the estimate.

To get the estimate for $t^{-N} ||\theta^{\alpha}(\what{r}_t)||_{0,0}$,   proceed as follows.  Let $N > 0$ be given, and  choose $\what{N} > N + (n + p)/2 $.  Then for small $t$, {{as above}} and modulo constants,
$$
t^{-N} ||\theta^{\alpha}(\what{r}_t)||_{0,0} \,\, \leq  \,\,  t^{-N} \Big[ \int ||\what{r}(x,X,\xi,\sigma,t)_t||^2 d \sigma d\xi dx \Big]^{1/2}   \,\, \leq  \,\,
$$
$$
t^{\what{N}-N} \Big[ \int (1 +|t\xi|)^{-2n-2|s|}  (1 +|t\sigma |)^{-2p-2|s|} (1 + | \xi |^2 + | \sigma |^2)^{-s}  d \sigma d\xi dx \Big]^{1/2}   \,\, \leq  \,\,
$$
$$
t^{\what{N}-N} \Big[ \int (1 +|t\xi|)^{-2n-2|s|}  (1 +|t\sigma |)^{-2p-2|s|} (1 + | \xi |^2 + | \sigma |^2)^{(|s|-s)/2}  d \sigma d\xi dx \Big]^{1/2}   \,\, \leq  \,\,
$$
$$
t^{\what{N}-N} \Big[ \int (1 +|t\xi|)^{-2n-2|s|}  (1 +|t\sigma |)^{-2p-2|s|} t^{|s|-s}(t^2 + | t\xi |^2 + |t \sigma |^2)^{(|s|-s)/2} d \sigma d\xi  \Big]^{1/2}   \,\, \leq  \,\,
$$
$$
t^{\what{N}-N} \Big[ \int (1 +|t\xi|)^{-2n-2|s|}  (1 +|t\sigma |)^{-2p-2|s|} t^{|s|-s}(1 + | t\xi |^2 + |t \sigma |^2)^{(|s|-s)/2} d \sigma d\xi  \Big]^{1/2}   \,\, \leq  \,\,
$$

$$
t^{\what{N} + (|s|-s)/2  -(N + (n + p)/2 )} \Big[ \int (1 +|\xi|)^{-2n-2|s|}  (1 +|\sigma |)^{-2p-2|s|} (1 + | \xi |^2 + |\sigma |^2)^{(|s|-s)/2} d \sigma d\xi  \Big]^{1/2},
$$
which goes to zero as $t \to 0$.

Next, do the same analysis on $\what{\theta}^{\alpha}((1 + |\xi|^2 + |\sigma|^2)^{k/2})
\what{\theta}^{\alpha}(\what{r}(t)_t)$, which yields a symbol denoted $\what{w}(t)_t$ so that 
$$
\what{\theta}^{\alpha}(q^{-k}_t)\what{\theta}^{\alpha}(\what{r}_t) \,\, = \,\,  \what{\theta}^{\alpha}(\what{w}_t),
$$
and  $\what{w}(t)_t$ also has the same properties as $\what{r}$ above, mutatis mutandis.
As in the $||P_t||_{0,0}$ case, we first assume that the $(x,X)$  support of the symbol $\what{w}$ is a compact subset of a cube $I^{n+p} \subset TF$.  Then we get, modulo constants, 
$$
 ||\what{w}_t||^2_{0,0} \,\, \leq \,\, \int_{I^{n+p} \times \R^{n+p}} ||\what{w}(x,X,\xi, \sigma,t)_t ||^2 \, d\xi d\sigma  dx dX,
$$
and we may proceed as in the $||P_t||_{0,0}$ case to finish the proof.
\end{proof}

\section{Bifiltered calculus on complete foliations}\label{Bifiltered}


Suppose that $u \in C^{\infty}_c(\R^p\times \R^q=\R^n, \C^{a })$,  and denote its Fourier transform by ${u}$.   For all $s,k \in \R$, the Sobolev $s,k$ norm of $u$ is defined by:
$$
\|u\|^2_{s,k} \,\, = \,\, \int_{\varsigma \in \R^p,  \eta \in \R^q} \; \; |{u}(\varsigma,\eta)|^2(1 + |\varsigma|^2 + |\eta|^2)^s(1 + |\varsigma|^2)^k d\varsigma d\eta.
$$

\begin{definition}\cite{GreenleafUhlmann, K97}
The space $\oH^{s,k}(\R^n,\R^p,\C^{a })$ is the completion of $C^{\infty}_c(\R^n, \C^{a })$ under the norm $\| \cdot \|_{s,k}$. A similar definition works for any open subsets $U\subset \R^p$ and $V\subset \R^q$ yielding the space $\oH^{s,k}(U,V,\C^{a })$.
\end{definition}

Denote by $ M_{a }(\C)$ the $a$ by $a$ complex matrices.

\begin{definition}
An element
$k(z,x,y,\sigma,\varsigma,\eta) \in C^{\infty}(\I^p \times \I^p \times \I^q \times \R^p \times \R^p \times \R^q, M_{a }(\C))$ belongs to the class $S^{m,\ell}(\I^p \times \R^n, \R^p, M_{a }(\C))$, (with $n = p+q$),   if for any multi-indices $\alpha$, $\beta$,  and $\gamma$, there is a constant $C_{\alpha,\beta, \gamma} > 0$ so that 
\begin{equation}\label{LocalSymbol}
\| \pa^{\alpha}_{\varsigma,\eta} \pa^{\beta}_{\sigma} \pa^{\gamma}_{z,x,y} k(z,x,y,\sigma,\varsigma,\eta)\| \,\, \leq \,\, C_{\alpha,\beta, \gamma}(1 + |\varsigma|+|\eta|)^{m- |\alpha|}(1+|\sigma|)^{\ell - |\beta|}.
\end{equation}
\end{definition}
Such a $k$ defines an operator 
$A:C^{\infty}_c(\I^n, \C^{a }) \to C^{\infty}(\I^n, \C^{a })$ by the formula
\begin{equation}\label{theta}
 Au(x,y) \,\, = \,\, (2\pi)^{-2p-q} \int e^{i[(x-x'-z)\varsigma +(y-y')\eta + z\sigma]} k(z,x,y,\sigma,\varsigma,\eta)u(x',y') dz dx' dy' d\varsigma  d\eta d\sigma.
\end{equation}

The distributional Schwartz kernel of $A$ is thus the oscillating integral
$$
K_A (x,y;x',y') \,\, = \,\, (2\pi)^{-2p-q} \int e^{i[(x-x'-z)\varsigma +(y-y')\eta + z\sigma]} k(z,x,y,\sigma,\varsigma,\eta) dz d\varsigma  d\eta d\sigma.
$$
If  this Schwartz kernel is uniformly supported in $\I^n \times \I^n$, we write 
$A \in \Psi^{m,\ell}(\I^n,\I^p, \C^{a})$.\\

\begin{proposition}\label{SobolevBound}\cite{K97}
Any operator $A\in \Psi^{m,\ell}(\I^n,\I^p, \C^{a})$ defines, for any $s$ and $k$,  a continuous mapping 
$$
A: \oH^{s,k}(\I^n,\I^p, \C^{a})\longrightarrow \oH^{s-m,k-\ell}(\I^n,\I^p, \C^{a}),
$$
In particular, if $m\leq 0$ and $\ell \leq 0$ then $A$ extends to an $L^2$-bounded operator.
\end{proposition}

The proof is classical, see Theorem 3.3 in \cite{GreenleafUhlmann} and \cite{K97}.

We now extend  the previous definitions and properties to bounded geometry foliations. Let   $( M, \maF)$ be a smooth foliated Riemannian manifold.  We thus assume that the manifold $ M$ has $C^\infty$ bounded geometry and so is complete, and that all the leaves satisfy the same bounded geometry assumption.   We say in this case that the foliation has ($C^\infty$-)bounded geometry. 
All $\C$ vector bundles $\maE$ over $ M$ are assumed to also have  $C^\infty$-bounded geometry. In this case, we may choose a $C^\infty$-bounded Hermitian structure and consider the space $L^2 ( M, \maE)$ of $L^2$-sections of $\maE$. In fact, the Sobolev spaces associated with $\maE$ are also well defined, see for instance \cite{Shubin}. We review below the bigraded Sobolev spaces for our complete foliation. When $ M$ is compact, we recover the usual bigraded Sobolev spaces and the bifiltered calculus  as defined  in \cite{K97}. 

Let  $({ U}_i, { T}_i)_{i\in I}$ be a  good open cover of the foliation $( M, \maF)$ with finite multiplicity and such that ${ U}_i \simeq \R^p \times { T}_i$ and ${ T}_i\simeq \R^q$ so that ${ U}_i\simeq \R^p\times\R^q$. Using a classical lemma due to Gromov \cite{G81}, it is easy to check that such open cover always exists. Moreover, we may assume that the open sets ${ U}_i$ are metric balls  which are diffeomorphic images of the local exponential maps and such that any plaque in any ${ U}_i$ is the diffeomorphic image of the leafwise exponential map.  Let $\{\phi_i\}$  be a $C^\infty$-bounded partition of unity subordinate to the cover $\{{ U}_i\}$ of $ M$ \cite{Shubin}.  For $u \in C^{\infty}_c( M, \maE)$, and using the local trivializations of $\maE$ over the ${ U}_i$, we define its $s,k$ norm as 
$$
\|u\|_{s,k} \,\, = \,\,  \sum_{i}   \| \phi_i \cdot u \|_{s,k},
$$
where on the right we are thinking of the product $\phi_i \cdot u$  as an element in $ C^{\infty}_c(\R^n, \C^{a })$ using the trivializations, and the norm $ \| \cdot \|_{s,k}$ is pulled back from the norm of $\oH^{s,k}(\R^n,\R^p,\C^{a })$.
\begin{definition}
The bigraded Sobolev space $\oH^{s,k}( M,  \maF;  \maE)$ is the completion of $C^{\infty}_c( M, \maE)$ under the  norm  $\| \cdot \|_{s,k}$. 
\end{definition}

Classical arguments show that although the norms depend on the choices, the bigraded Sobolev spaces $\oH^{s,k}( M,  \maF; \maE)$ do not. 
Notice that the holonomy groupoid (which is assumed to be Hausdorff in the present paper) is also a foliated manifold of bounded geometry and thus admits the covers and partitions of unity as above which fit with the description given in  \cite{ConnesIntegration}.
Let ${ V} \simeq \I^p\times \I^q$ be a distinguished foliation chart for $\maF$.  Then  the restriction $\maE | _{{ V}} \simeq{ V}\times \C^a$. Let ${ V} \times_{\gamma}{ V}'\simeq \I^n\times \I^q$ be a  chart for the holonomy groupoid $\maG$ corresponding to $\gamma\in \maG_{{ V}}^{{ V}'}$ with ${ V}'$ a distinguished chart with the same properties.   Using these charts and trivializations, any element $A_0 \in \Psi^{m,\ell}(\I^n,\I^p, \C^{a })$,  defines an operator  
\begin{equation}\label{LocalOperator}
A :C^{\infty}_c({ V}, \maE) \longrightarrow C^{\infty}_c({ V}' ; \maE).
\end{equation}
Such operator is called an elementary operator of class $(m, \ell)$. 

\begin{definition}\
A linear map $A: C_c^\infty ( M; \maE) \rightarrow C_c^\infty ( M; \maE)$ with finite propagation is a pseudodifferential operator of class $(m, \ell)$ if it is an elementary operator in all local charts ${ V},{ V}'$ as above (with $C^\infty$-bounded coefficients with bounds independent of the chosen local charts). 
\end{definition}

The finite propagation assumption is defined using the geodesic distance and the completeness, but we could as well assume  that $A$ is uniformly supported in the sense of  \cite{NWX} without reference to the geodesic distance. Then the operator $A$ sends compactly supported smooth sections to compactly supported smooth sections. 
A uniform smoothing operator will be an operator with smooth Schwartz kernel $k$ which has finite propagation and such that $k$ is $C^\infty$-bounded. This latter property means that  we can estimate the derivatives of $k$ in local coordinates by uniform bounds over $ M\times M$. Such operator induces a bounded operator  between any {\underline{usual}} Sobolev spaces of the bounded-geometry manifold $ M$ as defined in \cite{Shubin}.   The space (obviously a $*$-algebra) of  uniform  smoothing operators is denoted by $\Psi^{-\infty} ( M, \maE)$. An easy partition of unity argument in the sense described above gives the following standard lemma for all bounded geometry foliations.

\begin{lemma}\cite{K97}
A uniform smoothing operator $T$ induces, for any $s, k, s', k'$, a bounded operator
$$
T: \oH^{s,k}( M,  \maE) \longrightarrow \oH^{s',k'}( M,  \maE).
$$
\end{lemma}

Denote by $\Psi^{m,\ell} ( M, \maF;  \maE)$ the space of operators of the form $T=A + R$ where $A$ is a uniformly supported pseudodifferential  operator of type $(m,\ell)$ and $R\in \Psi^{-\infty}( M, \maE)$ is a uniform smoothing operator. 
Notice that 
if $A \in \Psi^{m,\ell}( M,  \maF; \maE)$ then the formal adjoint $A^*$ also belongs to $\Psi^{m,\ell} ({ M}, \maF; \maE)$. Moeover,  if $B \in \Psi^{m',\ell'}({ M}, \maF; \maE)$, then  $A\circ B \in  \Psi^{m+m',\ell+\ell'}({ M},  \maF;  \maE)$.  The proof in the compact case given in \cite{K97}  extends again  to our setting. Indeed, by the first appendix in \cite{Shubin} we know that for any $R\in \Psi^{-\infty} ({ M}, \maE)$ and any $A\in \Psi^{m,\ell}({ M}, \maF; \maE)$, the composite operators $A\circ R$ and $R\circ A$  are uniform smoothing operators, hence belong to $\Psi^{-\infty} ({ M}, \maE)$. 
Therefore, we only need to check the same properties for locally elementary operators which are uniformly supported. Using a locally finite partition of unity of ${ M}$ as described above, this is reduced to considering an elementary operator  $A$ from sections over ${ V}$ to sections over ${ V}'$ as above.  But then we apply the techniques developped in \cite{GreenleafUhlmann}[Proposition 1.39]. Notice in addition  that  composition of compactly supported operators is compactly supported and adjoint of compactly supported is compactly supported. We can now state:

\begin{proposition}\label{SobolevBound}\cite{Shubin} 
Any operator $A\in \Psi^{m,\ell}({ M},  \maF; \maE)$ defines, for any $s$ and $k$,  a continuous mapping 
$$
A:\oH^{s,k}({ M},  \maF; \maE)\longrightarrow \oH^{s-m,k-\ell}({ M}, \maF;  \maE)),
$$
In particular, if $m\leq 0$ and $\ell \leq 0$ then $A$ extends to an $L^2$-bounded operator.
\end{proposition}

Since any $R\in \Psi^{-\infty} ({ M}, \maE)$ induces a bounded operator between any bigraded Sobolev spaces, this statement is again local by using a  partition of unity argument in the sense of \cite{Shubin}.


\begin{thebibliography}{C94/95}

\bibitem[AG83]{AG} L. Alvares-Gaum{{\'{e}}}, {\em Supersymmetry and the Atiyah-Singer index theorem},  Comm. Math. Physics
 {\bf 88}  (1983) 161--173.
 
 \bibitem[A76]{A} M. {F}. Atiyah, {\em Elliptic operators, discrete groups and von Neumann algebras},  Ast\'{e}risque
 {\bf 32/33}  (1976) 43--72.

\bibitem[ABP73]{ABP} M. {F}. Atiyah, {R}. Bott, and V. K. Patodi, 
{\em On the heat equation and the index theorem}, 
Invent. Math. {\bf 19} (1973) 279--330.

\bibitem[BH10]{BH-Lefschetz1} M.-T. Benameur and J. L.  Heitsch, {\em The higher fixed point theorem for foliations. I. Holonomy invariant currents.} J. Funct. Anal. 259 (2010), no. 1, 131Ð173.

\bibitem[BH16a]{BH2012} M-T. Benameur and J. L. Heitsch,  {\em {Transverse noncommutative geometry of foliations}}, 
work in progress.

\bibitem[BH16b]{BH2012b} M-T. Benameur and J. L. Heitsch,  {\em The Connes-Chern character for Riemannian foliations}, preprint, (2015).

\bibitem[BH16c]{BH-Lefschetz2} M.-T. Benameur and J. L.  Heitsch, {\em The higher fixed point theorem for foliations. II. Applications,} preprint.

\bibitem[BR15]{BR} M.-T. Benameur and I. Roy, {\em The Higson-Roe exact sequence and $\ell^2$ eta invariants.}  J. Funct. Anal. 268 (2015), no. 4, 974-1031.

\bibitem[B11]{BismutHypo} J.-M. Bismut, {\em Hypoelliptic Laplacian and orbital integrals}. Annals of Mathematics Studies, 177. Princeton University Press, Princeton, NJ, 2011. 

\bibitem[BlF90]{BF} J. Block and J. Fox, 
{\em Asymptotic pseudodifferential operators and index theory}, 
Contemp. Math. {\bf 105} (1990) 1--32.

\bibitem[B69]{B-H} J. Bokobza-Haggiag, {\em Op\'{e}rateurs pseudo-diff\'{e}rentiels sur une vari\'{e}t\'{e} diff\'{e}rentiable},  Ann. Inst. Four. (Grenoble)   {\bf 19}  (1969) 125--177.

\bibitem[C79]{ConnesIntegration} A.~Connes,
\newblock {\em Sur la th\'eorie non commutative de l'integration}, Lec. Notes in Math.  {\bf 725} Springer, Berlin (1979).

\bibitem[CM95]{CM95} A. Connes and H. Moscovici,  
{\em The local index formula in noncommutative geometry}, Geom. Funct. Anal. {\bf 5} (1995)  174--243.

\bibitem[CM98]{ConnesMoscoviciHopf} A. Connes and H. Moscovici,  
{\em Hopf algebras, cyclic cohomology and the transverse index theorem}, Comm. Math. Phys. 198 (1998), no. 1, 199-246.

\bibitem[G83]{Getz} E. Getzler, 
{\em Pseudodifferential operators on supermanifolds and the Atiyah-Singer Index Theorem}, Commun. Math. Phys.  {\bf 92} (1983) 163--178.
 
\bibitem[Gi84]{Gilkey} P. Gilkey,  {\em Invariance Theory, The Heat Equation, and the Atiyah-Singer Index Theorem}, Publish or Perish, Wilmington, Deleware (1984).

\bibitem[GlK91]{GlK91} J. F. Glazebrook  and F. W. Kamber, 
{\em Transversal Dirac families in Riemannian foliations},  Comm. Math. Phys. {\bf 140} (1991) 217--240.

\bibitem[GU90]{GreenleafUhlmann} A. Greenleaf and G. Uhlmann,  {\em Estimates for singular Radon transforms and pseudodifferential operators with singular symbols}, 
J. Funct. Anal. {\bf 89}, 202-232, 1990.

\bibitem[G81]{G81} M. Gromov, {\em Structures m\'etriques pour les vari\'et\'es riemannienne}, Edited by J. Lafontaine and P. Pansu. Textes Math\'ematiques [Mathematical Texts], 1. CEDIC, Paris, 1981. 

\bibitem[K97]{K97} Yu. Kordyukov, 
{\em Noncommutative spectral geometry of Riemannian foliations}, Manuscripta Math. {\bf 94} (1997) 45--73.

\bibitem[K07]{K07} Yu. Kordyukov, 
{\em The Egorov theorem for transverse Dirac type operators on foliated manifolds},  J. Geom. Phys. {\bf 57} (2007) 2345-2364.

\bibitem[LM89]{LawsonM} H. B. Lawson, Jr. and M-L. Michelsohn,
{\em Spin Geometry}, Princeton University Press, Princeton, N.J.  (1989).

\bibitem[L01]{Lescure} J.-M. Lescure, {\em Triplets spectraux pour les vari\'et\'es \`a singularit\'e conique isol\'ee. (French) [Spectral triples for pseudomanifolds with isolated conical singularity]}, Bull. Soc. Math. France 129 (2001), no. 4, 593-623.

\bibitem[NWX99]{NWX} V. Nistor, A. Weinstein and P. Xu, {\em Pseudodifferential operators on differential groupoids.}
Pacific J. Math. 189 (1999), no. 1, 117-152.  

\bibitem[P08]{Ponge}R. Ponge, {\em Heisenberg calculus and spectral theory of hypoelliptic operators on Heisenberg manifolds.} Mem. Amer. Math. Soc. 194 (2008), no. 906.

\bibitem[R88]{Remp}  J. Rempala, 
{\em A semi-global Taylor formula for manifolds},  Ann. Polonici Math. {\bf 47} (1988) 131--138.

\bibitem[R03]{RoeBook} J. Roe, {\em Lectures on coarse geometry.} University Lecture Series, 31. American Mathematical Society, Providence, RI, 2003. 

\bibitem[Sh92]{Shubin}  M. A. Shubin, {\em{Spectral theory of elliptic operators on non-compact manifolds}}, Ast\'erisque No. 207 (1992), 5, 35-108. 

\bibitem[T80]{Treves} F. Treves,
{\em Introduction to Pseudodifferential and Fourier Integral Operators, Vol.\ 1}, Plenum Press, New York, N.Y. (1980).

\bibitem[W78]{Wid1} H. Widom, {\em Families of pseudodifferential operators}, 
Topics in Functional Analysis, Adv. in Math. Suppl. Stud., 3, Academic Press, New York-London (1978) 345--395.

\bibitem[W80]{Wid2} H. Widom, {\em A complete symbolic calculus for pseudodifferential operators}, Bull. Sc. Math., 2nd Series {\bf 104} (1980) 19-63.  

\end{thebibliography}
\end{document}